\documentclass[a4paper,12pt]{report}

\usepackage{graphicx}

\usepackage{amsmath,amscd,amsthm}
\usepackage{ amssymb }
\usepackage{ mathrsfs }
\usepackage{tikz}
\usetikzlibrary{matrix,arrows,decorations.pathmorphing}
\usetikzlibrary{arrows,backgrounds,decorations.markings}

\usepackage{url}

\newcommand{\Ext}{\operatorname{Ext}}
\newcommand{\FE}{\operatorname{FE}}
\newcommand{\linspan}{\operatorname{span}}
\newcommand{\MCE}{\operatorname{MCE}}
\newcommand{\id}{\operatorname{id}}
\newcommand{\Aut}{\operatorname{Aut}}
\newcommand{\SH}{\operatorname{SH}}
\newcommand{\Sat}{\operatorname{S}}
\newcommand{\Obj}{\operatorname{Obj}}
\newcommand{\Mor}{\operatorname{Mor}}
\newcommand{\dom}{\operatorname{dom}}
\newcommand{\cod}{\operatorname{cod}}

\newtheorem{theorem}{Theorem}[section]
\newtheorem{lemma}[theorem]{Lemma}
\newtheorem{proposition}[theorem]{Proposition}
\newtheorem{corollary}[theorem]{Corollary}
\newtheorem{remark}[theorem]{Remark}
\newtheorem{definition}[theorem]{Definition}

\newtheorem{example}[theorem]{Example} 

\begin{document}

%
%

\begin{titlepage}
\begin{center}
\includegraphics[scale=0.25]{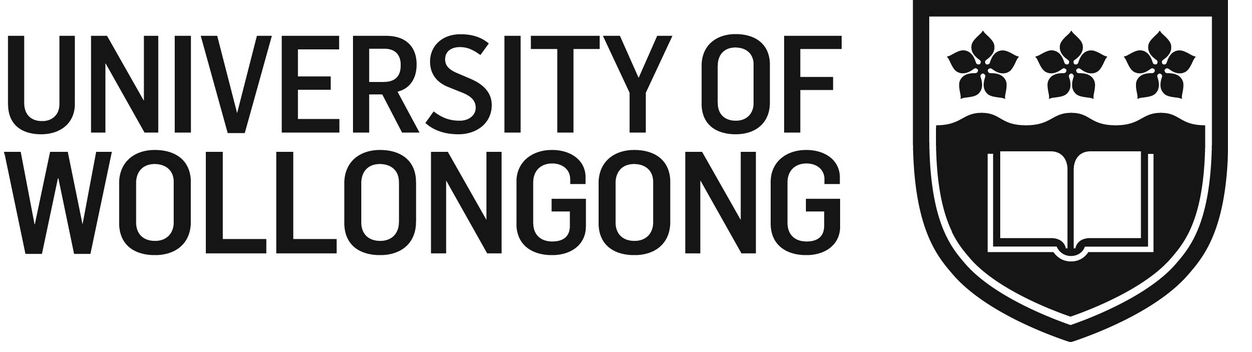}\\[2.5cm]
\end{center}
\begin{center}
{ \huge \textbf{\sc Twisted relative Cuntz-Krieger algebras associated to finitely aligned higher-rank graphs }}\\[1.2cm]
\textmd{By}\\[0.7cm]
{\huge{Benjamin Whitehead}\\[1.2cm]}

Supervisors: Professor Aidan Sims and Doctor Michael Whittaker \\

\large{\sc Bachelor of Mathematics Advanced (Honours)}\\
\textmd{\sc School of Mathematics and Applied Statistics}\\
\textmd{\sc 2012}
\end{center}
\end{titlepage}


%
%

\begin{abstract}
  To each finitely aligned higher-rank graph $\Lambda$ and each $\mathbb{T}$-valued 2-cocycle on $\Lambda$, we associate a family of twisted relative Cuntz-Krieger algebras. We show that each of these algebras carries a gauge action, and prove a gauge-invariant uniqueness theorem. We describe an isomorphism between the fixed point algebras for the gauge actions on the twisted and untwisted relative Cuntz-Krieger algebras. We show that the quotient of a twisted relative Cuntz-Krieger algebra by a gauge-invariant ideal is canonically isomorphic to a twisted relative Cuntz-Krieger algebra associated to a subgraph. We use this to provide a complete graph-theoretic description of the gauge-invariant ideal structure of each twisted relative Cuntz-Krieger algebra.
\end{abstract}

%
%

\newpage %
\pagenumbering{roman}

%
%

\mbox{} %
\vspace*{0.2\textheight} \\ %
\centerline{\bf Acknowledgements} I owe thanks to my supervisors, Professor Aidan Sims and Doctor Michael Whittaker, for their guidance, time and enthusiasm.
\mbox{}

%
%

\tableofcontents %

%
%
%
%

\chapter{Introduction} %
\label{chptr:introduction} \pagenumbering{arabic}

\section{Background}

A directed graph $E=(E^0,E^1,r,s)$ consists of a set $E^0$ of vertices, a set $E^1$ of edges and maps $r,s:E^1\rightarrow E^0$. We visualise the vertices as dots or points and the edges $e\in E^1$ as arrows pointing from $s(e)$, called its source, to $r(e)$, called its range. A $C^*$-algebra is a closed subspace of the bounded linear operators on a Hilbert space that is closed under multiplication and adjoints. We construct a $C^*$-algebra in the set of bounded linear operators on a Hilbert space $\mathcal{H}$ from $E$ as follows. To vertices $v\in E^0$ we associate projections $P_v$ of $\mathcal{H}$ onto mutually orthogonal subspaces $P_v\mathcal{H}$. To every edge $e\in E^1$ we associate a partial isometry $S_e$ that takes $P_{s(e)}\mathcal{H}$ isometrically into $P_{r(e)}\mathcal{H}$ and vanishes on $(1-P_{s(e)})\mathcal{H}$. We require that whenever $r^{-1}(v)$ is nonempty and finite, the subspaces $S_e\mathcal{H}$ with $r(e)=v$ are mutually orthogonal and span $\mathcal{H}_v$. Such a collection $\left\{ P_v ,S_e: v\in E^0, e\in E^1 \right \}$ is called a Cuntz-Krieger $E$-family. By taking the closure of all linear combinations of all finite products of $P_v$ and $S_e$ and their adjoints, we generate a $C^*$-subalgebra of the bounded linear operators on $\mathcal{H}$.

Every Cuntz-Krieger $E$-family generates a $C^*$-algebra. Different Cuntz-Krieger $E$-families can generate nonisomorphic $C^*$-algebras, but there is a unique universal $C^*$-algebra $C^*(E)$ generated by a universal Cuntz-Krieger $E$-family $\left\{ p_v, s_e : v\in E^0, e\in E^1 \right \}$. Loosely speaking, by ``universal" we mean that $C^*(E)$ is biggest $C^*$-algebra generated by a Cuntz-Krieger $E$-family: $C^*(E)$ contains all the complexities of any other $C^*$-algebra generated by a Cuntz-Krieger $E$-family. One can use the universal property to show that there is a group action $\gamma$ of the unit circle $\mathbb{T}\subset \mathbb{C}$ on $C^*(E)$ such that $\gamma_z(s_e)=zs_e$ and $\gamma_z(p_v)=p_v$ for each $e\in E^1$, $v\in E^0$ and $z\in \mathbb{T}$. This action is called the gauge action.

Kumjian, Pask, Raeburn and Renault introduce Cuntz-Krieger algebras of directed graphs, also known as graph algebras, in \cite{KPRR1997} following on from the seminal work of Cuntz and Krieger in \cite{CK1980} and Enomoto and Watatani in \cite{EW1980}. The sustained interest in graph algebras is motivated by the elegant relationships that exist between the connectivity and loop structure of the graph and the structure of the resulting graph algebra. For example, the ideals of $C^*(E)$ that are invariant for the gauge action $\gamma$ can be described in terms of subsets of $E^0$: Bates, Hong, Raeburn and Szyma\'nski in \cite{BHRS2002} describe sets $H$ of vertices, which they call saturated hereditary sets. They show that for each such $H$, the elements $\left\{ p_v : v\in H\right\}$ generate a gauge-invariant ideal $I_H$ of $C^*(E)$. Using a version of an Huef and Raeburn's gauge-invariant uniqueness theorem (\cite{HR1997}, Theorem 2.3), they show that $C^*(E)/I_H$ is canonically isomorphic to the Cuntz-Krieger algebra of what they call the quotient graph $E/H$. They then show that the gauge-invariant ideals of the Cuntz-Krieger algebra are indexed by pairs $(H,B)$, where $H$ is a saturated hereditary set and $B$ is a collection of so-called breaking vertices.

Robertson and Steger \cite{RS1999} introduce a class of higher-rank Cuntz-Krieger algebras arising from actions of $\mathbb{Z}^k$ on buildings. In \cite{KP2000}, Kumjian and Pask introduce higher-rank graphs and their Cuntz-Krieger algebra as a common generalisation of both Robertson and Steger's higher-rank Cuntz-Krieger algebras and graph algebras. Kumjian and Pask focus on Cuntz-Krieger algebras arising from row-finite higher-rank graphs with no sources: these are the higher-rank analogue of graphs in which a nonzero, finite number of edges enters every vertex. Raeburn, Sims and Yeend \cite{RSY2004} extend the analysis of Kumjian and Pask by associating Cuntz-Krieger algebras to finitely aligned higher-rank graphs, a more general class than row-finite higher-rank graphs with no sources.

In \cite{S20061}, Sims defines a class of relative Cuntz-Krieger algebras (\cite{S20061}, Definition 3.2) associated to a finitely aligned higher-rank graph, which includes the usual Cuntz-Krieger algebras as a special case. As in \cite{BHRS2002}, Sims identifies pairs of saturated hereditary sets and satiated (\cite{S20061}, Definition 4.1) sets, and associates to each such pair a gauge-invariant ideal.  Due the complex combinatorial nature of higher-rank graphs, it is no longer clear how to realise the quotient of the Cuntz-Krieger algebra as the Cuntz-Krieger algebra of a quotient graph. Instead, Sims shows that the quotient of the Cuntz-Krieger algebra by a gauge-invariant ideal associated to a saturated hereditary and satiated pair is canonically isomorphic to a relative Cuntz-Krieger algebra of a subgraph (\cite{S20062}, Corollary 5.2). Using this observation, Sims generalises the results of \cite{BHRS2002} for graph algebras to Cuntz-Krieger algebras of finitely aligned higher-rank graphs; that is, he classifies all the gauge-invariant ideals in the Cuntz-Krieger algebra of a higher-rank graph (\cite{S20061}, Theorem  5.5).

In \cite{KPS2012} and \cite{KPS2012b}, Kumjian, Pask and Sims introduce a cohomology theory for higher-rank graphs. They show how to twist the defining relations for a higher-rank graph $C^*$-algebra by a circle valued 2-cocycle (\cite{KPS2012b},  Definition 5.2). They use this to define the twisted Cuntz-Krieger algebra $C^*(\Lambda,c)$ (\cite{KPS2012b}, Proposition 5.3) associated to a higher-rank graph and a $\mathbb{T}$-valued 2-cocycle $c$. The class of twisted Cuntz-Krieger algebras contains many interesting examples, including all of the noncommutative tori (\cite{KPS2012}, Example 7.9) and the quantum spheres of \cite{BHMS2005}. Both \cite{KPS2012} and \cite{KPS2012b}, consider row-finite locally convex (\cite{RSY2003}, Definition 3.9) higher-rank graphs, rather than arbitrary finitely aligned higher-rank graphs. Kumjian, Pask and Sims produce a gauge-invariant uniqueness theorem (\cite{KPS2012b}, Corollary 7.7) using groupoid methods, and provide a criterion for simplicity (\cite{KPS2012b}, Corollary 8.2), but they do not address the broader question of ideal structure.

In this thesis we avoid groupoid methods in preference for a more elementary and direct analysis, similar to that of \cite{RSY2004}. This direct approach leads to sharper and more general results with cleaner arguments. Theorem 5.5 in \cite{S20062} gives a complete description of gauge-invariant structure of the Cuntz-Krieger algebra associated to a finitely aligned higher-rank graph. The aim of this thesis is to extend this result to all twisted relative Cuntz-Krieger algebras: this is a new result even for the untwisted case. This aim is achieved with Theorem~\ref{th:gauageinvariantidealstructure}.

\section{Overview}

This section gives a detailed account of the material contained in this thesis. We highlight definitions and results contained in each chapter, discuss their importance and connect them with the previous work of others.

In Chapter 2 we recall the definition of a $k$-graph (higher-rank graph) and develop the standard notation to deal with them. Each path $\lambda$ in a $k$-graph $\Lambda$ has a source $s(\lambda)$ and a range $r(\lambda)$ in $\Lambda^0$, the set of vertices. Unlike directed graphs, a $k$-graph $\Lambda$ not only consists of edges and vertices but all possible paths in the graph. There is a natural partial multiplication on $\Lambda$ given by concatenation; given paths $\lambda,\mu$ in $\Lambda$ such that $s(\lambda)=r(\mu)$, $\lambda \mu$ is another path in $\Lambda$. Each path $\lambda \in \Lambda$ has degree, or shape, $d(\lambda)\in\mathbb{N}^k$, where $k$ is the $k$ in $k$-graph; degree is a higher-dimensional analogue of length. Each $k$-graph satisfies the factorisation property: if $\lambda$ has degree $m+n$, then there exist unique paths $\mu$ and $\nu $ in $\Lambda$ such that $d(\mu)=m$, $d(\nu)=n$ and $\lambda =\mu\nu$. If $E$ is a directed graph, then the set $E^*$ of all paths in $E$ is a $1$-graph where the degree function $d:E^* \to \mathbb{N}$ measures the length of a path. Note that the factorisation property for $1$-graphs is automatic. So $k$-graphs generalise the notion of a directed graph.

Given a finitely aligned $k$-graph $\Lambda$, each finite collection $E$ of paths with a common range determines a ``gap projection" $Q^E$ in the twisted Toeplitz algebra (see Section~\ref{chptr:Gap_projections}) of $\Lambda$. Each twisted relative Cuntz-Krieger algebra of $\Lambda$ is a quotient of the twisted Toeplitz algebra by an ideal generated by such gap projections; so in each twisted relative Cuntz-Krieger algebra some, but typically not all, of the gap projections are zero. In Chapter~\ref{chptr:Twisted_relative_Cuntz-Krieger_algebras} we construct a spatial representation of the twisted Toeplitz algebra, and use it to characterise exactly which gap projections are nonzero in any given relative Cuntz-Krieger algebra.

In Section \ref{chptr:Toeplitz-Cuntz-Krieger_families}, for a finitely aligned $k$-graph $\Lambda$, we define a $\mathbb{T}$-valued 2-cocycle and the associated twisted Toeplitz algebra. A 2-cocycle is a map $$c:\left\{ (\lambda,\mu)\in \Lambda \times \Lambda : s(\lambda)=r(\mu) \right \}\rightarrow \mathbb{T}$$ that satisfies:
\begin{enumerate}
\item[(1)] $c(\lambda,\mu)c(\lambda \mu,\nu)=c(\mu,\nu)c(\lambda,\mu\nu)$ whenever $s(\lambda)=r(\mu)$ and $s(\mu)=r(\nu)$; and
\item[(2)] $c(\lambda,s(\lambda))=c(r(\lambda),\lambda)=1$ for all $\lambda\in \Lambda$.
\end{enumerate}
The first relation is called the 2-cocycle identity. We write $\underline{Z}^2(\Lambda,\mathbb{T})$ for the set of all $\mathbb{T}$-valued 2-cocycles.
The inspiration for this definition comes from the work of Kumjian, Pask and Sims in \cite{KPS2012} and \cite{KPS2012b} in developing a cohomology theory for $k$-graphs and their definition of a Cuntz-Krieger $(\Lambda,c)$-family (\cite{KPS2012b}, Definition 5.2), where $c$ is a 2-cocycle in the cohomology theory. Kumjian, Pask and Sims show there is a universal $C^*$-algebra $C^*(\Lambda,c)$, called the twisted Cuntz-Krieger algebra, generated by a Cuntz-Krieger $(\Lambda,c)$-family $\left\{ s_\lambda :\lambda \in \Lambda \right \}$. This $C^*$-algebra is said to be `twisted' due to the 2-cocycle appearing in the relation
\begin{equation}s_\lambda s_\mu = c(\lambda,\mu)s_{\lambda \mu}\label{eq:twistedrelationintro}\end{equation}
whenever $s(\lambda)=r(\mu)$. The 2-cocycle identity guarantees that the multiplication in \eqref{eq:twistedrelationintro} is associative.

We wish to extend the analysis of Kumjian, Pask and Sims to finitely aligned $k$-graphs. A Toeplitz-Cuntz-Krieger $(\Lambda,c)$-family is a collection $\left\{ t_\lambda :\lambda \in \Lambda \right \}$ of partial isometries in a $C^*$-algebra satisfying:
\begin{enumerate}
\item[(TCK1)] $\lbrace t_v : v\in \Lambda^0 \rbrace$ is a set of mutually orthogonal projections;
\item[(TCK2)] $t_{\mu}t_{\nu}=c(\mu,\nu) t_{\mu \nu}$ whenever $s(\mu)=r(\nu)$;
\item[(TCK3)] $t_{\lambda}^*t_{\lambda} = t_{s(\lambda)}$ for all $\lambda \in \Lambda$; and
\item[(TCK4)] $ t_\mu t^*_\mu t_\nu t^*_\nu
   = \displaystyle\sum_{\lambda=\mu \mu'=\nu\nu'\atop d(\lambda) \text{ minimal}} t_{\lambda} t^*_{\lambda}$.
 \end{enumerate}
In Definition 2.5 of \cite{RSY2004}, Raeburn, Sims and Yeend define a Cuntz-Krieger $\Lambda$-family for a finitely aligned $k$-graph $\Lambda$. Here we note how our definition of a Toeplitz-Cuntz-Krieger $(\Lambda,c)$-family differs from a Cuntz-Krieger $\Lambda$-family of \cite{RSY2004} and how the 2-cocycle is incorporated. Relation (TCK1) already appears in Definition 2.5 of \cite{RSY2004}. Relation (ii) in Definition 2.5 of \cite{RSY2004} is
$$t_\lambda t_\mu =t_{\lambda \mu}$$
whenever $s(\lambda)=r(\mu)$. Relation (TCK2) incorporates the 2-cocycle into the above relation, as in \eqref{eq:twistedrelationintro}. Relation (iii) in Definition 2.5 of \cite{RSY2004} is
$$t_\lambda^* t_\mu=\sum_{\lambda\alpha =\mu \beta \atop d(\lambda \alpha) \text{ minimal}} t_\alpha t_\beta^*$$
for all $\lambda,\mu \in\Lambda$. It turns out that the equivalent relation in the 2-cocycle setting is
\begin{equation} t_\lambda^* t_\mu=\sum_{\lambda\alpha =\mu \beta \atop d(\lambda \alpha) \text{ minimal}} \overline{c(\lambda,\alpha)}c(\mu,\beta)t_\alpha t_\beta^*\label{eq:RSYTCK3}\end{equation}
for all $\lambda,\mu \in\Lambda$. Since it is unclear why this is the correct relation, we remove it from the definition of a Toeplitz-Cuntz-Krieger $(\Lambda,c)$-family. We replace it by the familiar relations (TCK3), seen in Lemma 2.7 of \cite{RSY2004}, and (TCK4), seen in Remark 2.6 of \cite{RSY2004}. Lemma~\ref{le:equivofTCK} shows that replacing (TCK3) and (TCK4) by \eqref{eq:RSYTCK3} would give us an equivalent definition of a Toeplitz-Cuntz-Krieger $(\Lambda,c)$-family. Relation (iv) in Definition 2.5 of \cite{RSY2004},
\begin{equation} \prod_{\lambda \in E} \left(t_v-t_\lambda t_\lambda^*\right) =0 \text{ for all } E\in \FE(\Lambda),\label{eq:gapprojraeburnintro}\end{equation}
does not appear in the definition of Toeplitz-Cuntz-Krieger $(\Lambda,c)$-family, as in Definition 3.1.1 of \cite{SPhD}. The reason for this is explained below in the overview of Section~\ref{chptr:Relative_Cuntz-Krieger}.

In Section~\ref{chptr:twisted_Toeplitz_algebra} we show that there is a unique universal $C^*$-algebra generated by a Toeplitz-Cuntz-Krieger $(\Lambda,c)$-family $\left\{ s_\mathcal{T}(\lambda):\lambda \in \Lambda \right \}$ (Theorem~\ref{toeplitzalgebra}), we call this $C^*$-algebra the twisted Toeplitz algebra of $\Lambda$ and denoted it $\mathcal{T}C^*(\Lambda,c)$. Universal means that if $\left\{ t_\lambda : \lambda \in \Lambda \right \}$ is any other Toeplitz-Cuntz-Krieger $(\Lambda,c)$-family in a $C^*$-algebra $B$, then there is a homomorphism $\pi: \mathcal{T}C^*(\Lambda,c) \to B$ such that $\pi(s_\mathcal{T}(\lambda))=t_\lambda$ for each $\lambda \in \Lambda$. In Proposition~\ref{toeplitzrepresentation} we show there is a Toeplitz-Cuntz-Krieger $(\Lambda,c)$-family $\left\{ T_\lambda : \lambda \in \Lambda \right \}$ in $\mathcal{B}(\ell^2(\Lambda))$. The universal property of $\mathcal{T}C^*(\Lambda,c)$ determines a homomorphism $\pi_T^\mathcal{T}: \mathcal{T}C^*(\Lambda,c) \rightarrow \mathcal{B}(\ell^2(\Lambda))$ satisfying $\pi_T^\mathcal{T}(s_\mathcal{T}(\lambda))=T_\lambda $ for each $\lambda \in \Lambda$. This representation $\pi_T^\mathcal{T}$ of the twisted Toeplitz algebra on $\mathcal{B}(\ell^2(\Lambda))$ will be an invaluable tool throughout the thesis.

In Section \ref{chptr:Relative_Cuntz-Krieger}, for a vertex $v\in \Lambda^0$ and finite $E\subset v\Lambda:=\left\{ \lambda \in \Lambda : r(\lambda)=v\right \}$, we define the gap projection $$Q(t)^E:=\prod_{\lambda \in E}\left(t_v-t_\lambda t_\lambda^* \right)$$ (Definition~\ref{def:gapprojections}). Relation (TCK4) implies that the range projections $\left\{ t_\lambda t_\lambda^* : \lambda \in \Lambda \right \}$ pairwise commute, so that the gap projection $Q(t)^E$ is well-defined (Lemma~\ref{handylemma}(2)).
Definition~\ref{def:finitexhaustive} defines the collection $\FE(\Lambda)$ of finite exhaustive sets of $\Lambda$, as in Definition 2.4 of \cite{RSY2004}. Each element $E\in \FE(\Lambda)$ is finite and every path has a common range vertex $v$, so that $E\subset v \Lambda$. Similar to Definition 3.2 of \cite{S20061}, given $\mathcal{E}\subset \FE(\Lambda)$ we define a relative Cuntz-Krieger $(\Lambda,c;\mathcal{E})$-family to be a Toeplitz-Cuntz-Krieger $(\Lambda,c)$-family $\left\{t_\lambda:\lambda \in \Lambda \right \}$ that satisfies
\begin{enumerate} \item[(CK)] $Q(t)^E=0$ for all $E\in \mathcal{E}.$\end{enumerate}
Relation (CK) can be thought of as weaker version of \eqref{eq:gapprojraeburnintro}. We call a relative Cuntz-Krieger $(\Lambda,c;\FE(\Lambda))$-family a Cuntz-Krieger $(\Lambda,c)$-family. The map $$1:\left\{(\lambda,\mu)\in \Lambda \times \Lambda : s(\lambda)=r(\mu) \right\} \rightarrow \mathbb{T}$$ given by $1(\lambda,\mu)=1$ for $\lambda ,\mu \in \Lambda$ is a $\mathbb{T}$-valued 2-cocycle. The relative Cuntz-Krieger $(\Lambda,1;\FE(\Lambda))$-families are precisely the Cuntz-Krieger $\Lambda$-families of (\cite{RSY2004}, Definition 2.5). When $\Lambda$ is row-finite and locally convex, these are precisely the Cuntz-Krieger $\Lambda$-families of \cite{KP2000}. This assertion is nontrivial: we refer the reader to Appendix A of \cite{RSY2004} for a detailed explanation of Cuntz-Krieger relations associated to different classes higher-rank graphs and their equivalence.

Given $\mathcal{E}\subset \FE(\Lambda)$, Theorem~\ref{th:relativecuntzkriegeralgebra} shows that there is a unique universal $C^*$-algebra $C^*(\Lambda,c;\mathcal{E})$ generated by a relative Cuntz-Krieger $(\Lambda,c;\mathcal{E})$-family $\left\{s_\mathcal{E}(\lambda):\lambda \in \Lambda \right \}$. By definition, $C^*(\Lambda,c;\mathcal{E})= \mathcal{T}C^*(\Lambda,c) / J_\mathcal{E}$ and $s_\mathcal{E}(\lambda)=s_\mathcal{T}(\lambda)+J_\mathcal{E}$ for each $\lambda \in \Lambda$, where $J_\mathcal{E}$ is the smallest closed ideal containing the collection of gap projections $$\left\{ Q(s_\mathcal{T})^E: E\in \mathcal{E}\right \}.$$ The reason for using relative Cuntz-Krieger $(\Lambda,c;\mathcal{E})$-families will become apparent when we classify the gauge-invariant structure of $C^*(\Lambda,c;\mathcal{E})$.

In Section~\ref{chptr:Gap_projections} we establish two results concerning $C^*(\Lambda,c;\mathcal{E})$. Firstly, we show that the projections associated to vertices are all nonzero; that is, $s_\mathcal{E}(v)\neq 0$ for each $v\in \Lambda^0$. Secondly, we identity the collection $\overline{\mathcal{E}}$ of finite exhaustive sets containing $\mathcal{E}$ and such that for $E\in \FE(\Lambda)$,

\begin{equation}E\in \overline{\mathcal{E}} \text{ if and only if } Q(s_\mathcal{E})^E=0.\label{eq:satiatedgapprojintro}\end{equation}

The collection $\overline{\mathcal{E}}$ is called the satiation of $\mathcal{E}$, and is, by definition, the smallest satiated set containing $\mathcal{E}$ (see Definition 4.1 and Definition 5.1 of \cite{S20061}). Theorem~\ref{th:nonzerogapprojection} establishes \eqref{eq:satiatedgapprojintro}. This result is needed throughout the thesis, including in the proof of the gauge-invariant uniqueness theorem (Theorem~\ref{th:gaugeinvariant}). In \cite{S20061}, Sims outlines a method for constructing the satiation $\overline{\mathcal{E}}$ of $\mathcal{E}$. In doing so, proves the ``only if" direction of \eqref{eq:satiatedgapprojintro} in (\cite{S20061}, Corollary~5.6). In this thesis we adopt the same method (Proposition~\ref{le:satiation1}). The ``if" direction of~\eqref{eq:satiatedgapprojintro} is much more involved. In Definition 4.3 of \cite{S20061}, Sims defines the collection $\partial(\Lambda;\mathcal{E})$ of $\mathcal{E}$-compatible boundary paths; such paths are analogues of infinite paths in a $k$-graph with a range but no source, which satisfy a property relating to $\mathcal{E}$. In \cite{S20061}, Sims proves the ``if" direction of \eqref{eq:satiatedgapprojintro} in Corollary 4.9 of \cite{S20061} using a concrete representation $\left\{ S_\mathcal{E}(\lambda): \lambda \in \Lambda \right \}$ (\cite{S20061}, Definition~4.5) of $C^*(\Lambda;\mathcal{E})$ on $\mathcal{B}(\ell^2(\partial(\Lambda;\mathcal{E})))$. It seems that this representation cannot be extended to general twisted relative Cuntz-Krieger algebras. Each of the $S_\mathcal{E}(\lambda)$ in \cite{S20061} are determined by
\begin{equation}
S_\mathcal{E}(\lambda)\xi_x := \delta_{s(\lambda),r(x)} \xi_{\lambda x } \label{eq:simsrepintro}
\end{equation}
where $x$ denotes an $\mathcal{E}$-compatible boundary path and $\left\{ \xi_x : x\in \partial (\Lambda;\mathcal{E})\right\}$ denotes the canonical basis for $\ell^2(\partial(\Lambda;\mathcal{E}))$. The reader should compare \eqref{eq:simsrepintro} with the representation of $\mathcal{T}C^*(\Lambda,c)$ on $\mathcal{B}(\ell^2(\Lambda))$ given in Definition~\ref{toeplitzrepresentation}. In order for the family $\left\{ S_\mathcal{E}(\lambda):\lambda \in \Lambda \right \}$ to satisfy (TCK2) and be a representation of $C^*(\Lambda,c;\mathcal{E})$, we would need to make sense of a 2-cocycle evaluated at an $\mathcal{E}$-compatible boundary path. There seems to be no canonical way to do this. Instead, to prove \eqref{eq:satiatedgapprojintro} we need a different analysis.

As we do not have a concrete representation of $C^*(\Lambda,c;\mathcal{E})$, we analyse the structure of the ideals $J_\mathcal{E}\subset \mathcal{T}C^*(\Lambda,c)$ in order to understand the quotients $\mathcal{T}C^*(\Lambda,c)/ J_\mathcal{E}=C^*(\Lambda,c;\mathcal{E})$. By analysing these ideals, we distinguish between those gap projections associated to $\overline{\mathcal{E}}$ and those associated to $\FE(\Lambda) \setminus \overline{\mathcal{E}}$. Using this, we show the ``if" direction of \eqref{eq:satiatedgapprojintro} in Theorem~\ref{th:nonzerogapprojection}.

Chapter~\ref{chptr:Analysis_of_the_core} contains an analysis of the fixed point algebra of the gauge-action, called the core of $C^*(\Lambda,c;\mathcal{E})$ and denoted $C^*(\Lambda,c;\mathcal{E})^\gamma$. The gauge-action is a strongly continuous group action $\gamma : \mathbb{T}^k \rightarrow \Aut(C^*(\Lambda,c;\mathcal{E}))$ determined by \begin{equation}\gamma_z(s_\mathcal{E}(\lambda))=z^{d(\lambda)}s_\mathcal{E}(\lambda)\label{eq:thegaugeactionintro},\end{equation} where $z^{d(\lambda)}:=\prod_{i=1}^k z_i^{d(\lambda)_i}$. To see there is an action, given $z\in \mathbb{T}^k$ note that
$$\left\{z^{d(\lambda)}s_\mathcal{E}(\lambda):\lambda \in \Lambda \right\}$$
is a relative Cuntz-Krieger $(\Lambda,c)$-family. Thus the universal property implies that there is a homomorphism $\gamma_z: C^*(\Lambda,c;\mathcal{E}) \rightarrow C^*(\Lambda,c;\mathcal{E})$ satisfying \eqref{eq:thegaugeactionintro}. One checks that $\gamma_{\overline z}$ is an inverse for $\gamma_z$, so that each $\gamma_z$ is an automorphism.

 The main result of Chapter~\ref{chptr:Analysis_of_the_core} is the gauge-invariant uniqueness theorem (Theorem~\ref{th:gaugeinvariant}). If $\left\{ t_\lambda : \lambda \in \Lambda \right\}$ be a relative Cuntz-Krieger $(\Lambda,c;\mathcal{E})$-family, this theorem characterises when the homomorphism $\pi_t^\mathcal{E}: C^*(\Lambda,c;\mathcal{E}) \rightarrow C^*\left( \left\{ t_\lambda : \lambda \in \Lambda \right \} \right )$ given by the universal property is injective.

Section~\ref{chptr:groups_actions} outlines standard results in the literature concerning strongly continuous group actions and faithful conditional expectations. Section~\ref{chptr:gauge_action} constructs the gauge action determined by equation \eqref{eq:thegaugeactionintro}. We prove several standard results concerning the gauge action. For example, in Lemma~\ref{le:stronglycontinuous} we show the gauge-action is strongly continuous and in  Lemma~\ref{le:corestructure} we prove that
$$C^*(\Lambda,c;\mathcal{E})^\gamma =\overline{\linspan}\left\{ s_\mathcal{E}(\lambda)s_\mathcal{E}(\mu)^* : \lambda,\mu \in \Lambda, d(\lambda)=d(\mu) \right \}.$$

Due to this characterisation of the core, it has become standard in the context of graph algebras to study the subalgebra \begin{equation}\overline{\linspan} \left\{ t_\lambda t_\mu^* : \lambda,\mu \in \Lambda, d(\lambda)=d(\mu) \right \}\label{eq:coreAFintro}\end{equation} of $C^*\left\{ \left( t_\lambda : \lambda \in \Lambda \right \} \right)$. Given another relative Cuntz-Krieger $(\Lambda,c;\mathcal{E})$-family $\left\{ t_\lambda ' : \lambda \in \Lambda \right \}$, we want to establish conditions under which
\begin{equation}\overline{\linspan} \left\{ t_\lambda t_\mu^* : \lambda ,\mu\in \Lambda, d(\lambda)=d(\mu) \right \} \cong \overline{\linspan}  \left\{ t_\lambda' t_\mu^{\prime*}: \lambda,\mu \in \Lambda, d(\lambda)=d(\mu) \right \}.\label{eq:gencoreintro}\end{equation} The idea is to show that \eqref{eq:coreAFintro} is an increasing union of finite-dimensional subalgebras of the form $\overline{\linspan} \left\{ t_\lambda t_\mu^* : \lambda,\mu \in E,d(\lambda)=d(\mu) \right \}$ where $E\subset \Lambda$ is finite. We do this in Sections~\ref{chptr:orthogonalising_range}~and~\ref{chptr:matrix_units}. Almost all the content of these sections is taken from the work of Raeburn and Sims in \cite{RS2005} and Raeburn, Sims and Yeend in \cite{RSY2004}. Most of the proofs are changed only by the presence of the 2-cocycle. In Section~\ref{chpt:isoofthecore}, we give conditions under which \eqref{eq:gencoreintro} holds; even when $\left\{ t_\lambda ' : \lambda \in \Lambda \right \}$ is a relative Cuntz-Krieger $(\Lambda,b;\mathcal{E})$-family, for a 2-cocycle $b\in \underline{Z}^2(\Lambda,\mathbb{T})$, possibly with $b\neq c$.

Section~\ref{chptr:gaugeinvariantuniqueness} gives a statement of the gauge-invariant uniqueness theorem. This theorem says that if $\pi_t^\mathcal{E}|_{C^*(\Lambda,c;\mathcal{E}) ^\gamma}$ is injective and $C^*\left(\left\{ t_\lambda : \lambda \in \Lambda \right \} \right)$ carries a group action $\beta : \mathbb{T}^k \rightarrow \Aut\left(C^*\left(\left\{ t_\lambda : \lambda \in \Lambda \right \} \right)\right)$ that behaves similar to the gauge-action, then $\pi_t^\mathcal{E}$ is injective. This theorem is the main tool used to classify the gauge-invariant ideals in $C^*(\Lambda,c;\mathcal{E})$.

Chapter~\ref{chptr:Gauge-invariant_ideals} concerns the gauge-invariant ideal structure of $C^*(\Lambda,c;\mathcal{E})$. Section~\ref{chptr:idealstohereditarysets} starts with the definition of a hereditary and relatively saturated set $H$ of vertices. The definition of hereditary is standard, see Definition 3.1 in \cite{S20062}. The notion of a relatively saturated set, rather than the standard saturated set, is needed as we want to classify the gauge-invariant ideals of all twisted relative Cuntz-Krieger algebras, not just the twisted Cuntz-Krieger algebras. In Section~\ref{chptr:idealstohereditarysets} given an ideal $I\subset C^*(\Lambda,c;\mathcal{E})$, we find a relatively saturated and hereditary set $H_I$ and a satiated set $\mathcal{B}_I$. In Section~\ref{chptr:Quotients}, given a relatively saturated and hereditary subset $H$ and a satiated set $\mathcal{B}$, we construct a gauge-invariant ideal $I_{H,\mathcal{B}}$. We show that the quotient $C^*(\Lambda,c;\mathcal{E})/I_{H,\mathcal{B}}$ is canonically isomorphic to a twisted relative Cuntz-Krieger algebra of a subgraph of $\Lambda$. In Section~\ref{chapt:result} we show that every gauge-invariant ideal has the form $I_{H,\mathcal{B}}$; that is, we give a complete graph-theoretic description of the gauge-invariant ideals in $C^*(\Lambda,c;\mathcal{E})$.

\section{Main results}

The first major result, as was mentioned in \eqref{eq:satiatedgapprojintro}, is Theorem~\ref{th:nonzerogapprojection}: for $E\in \FE(\Lambda)$ we have
$$E\in \overline{\mathcal{E}} \text{ if and only if } Q(s_\mathcal{E})^E =0.$$

Let $b,c\in \underline{Z}^2(\Lambda,\mathbb{T})$. If $\left\{ t_\lambda : \lambda \in \Lambda \right \}$ is a relative Cuntz-Krieger $(\Lambda,b;\mathcal{E})$-family and if $\left\{ t_\lambda' : \Lambda \in \Lambda  \right \}$ is a relative Cuntz-Krieger $(\Lambda,c;\mathcal{E})$-family. Then Theorem~\ref{th:coreiso} gives us conditions under which
$$\overline{\linspan} \left\{ t_\lambda t_\mu^* : \lambda ,\mu\in \Lambda, d(\lambda)=d(\mu) \right \}  \cong \overline{\linspan}  \left\{ t_\lambda' t_\mu^{\prime*} : \lambda,\mu \in \Lambda, d(\lambda)=d(\mu) \right \} .$$
This Theorem is crucial in the proof of the gauge-invariant uniqueness theorem, although in the proof we only need the situation where $b$ are $c$ are identical. The case where $b\neq c$ has found application in computing the $K$-theory of twisted Cuntz-Krieger algebras associated to $k$-graphs \cite{KPS2012c}.

Theorem~\ref{th:gaugeinvariant} is a version of an Huef and Raeburn's gauge-invariant uniqueness theorem for twisted relative Cuntz-Krieger algebras. It is an essential tool in giving a graph-theoretic description of the gauge-invariant ideals of twisted relative Cuntz-Krieger algebras. The proof of this theorem for relative algebras is more difficult than the non-relative case (see \cite{S20061}), and even more so in the twisted relative case. The main technical result needed is Theorem~\ref{th:nonzerogapprojection} discussed above.

In Chapter~\ref{chptr:Gauge-invariant_ideals} we present our main result, a complete graph-theoretic description of the gauge-invariant ideal structure of each twisted relative Cuntz-Krieger algebra. This result extends the corresponding result of (\cite{S20062}, Theorem 5.5) for the Cuntz-Krieger algebra in two ways: here we consider all relative Cuntz-Krieger algebras, which includes the Cuntz-Krieger algebra; and we also incorporate the 2-cocycle.

\chapter{Higher-rank graphs} %
\label{chptr:Higher-rank graphs}

In this chapter we recall the definition of a $k$-graph (higher-rank graph). We establish the standard notation for dealing with $k$-graphs.

\section{Basic category theory}

Higher-rank graphs are defined using the language of categories. The following two definitions are taken from \cite{Mac1971}.
\begin{definition}
A \emph{small category} $\mathcal{C}$ is a sextuple $(\Obj(\mathcal{C}),\Mor(\mathcal{C}),\dom,\cod,\id,\circ)$ consisting of:
\begin{itemize}
\item[(1)] sets $\Obj(\mathcal{C})$ and $\Mor(\mathcal{C})$, called the \emph{object} and \emph{morphism} sets;
\item[(2)] the \emph{domain} and \emph{codomain} functions $\dom,\cod: \Mor(\mathcal{C}) \rightarrow \Obj(\mathcal{C})$;
\item[(3)] the \emph{identity} function $\id: \Obj(\mathcal{C}) \rightarrow \Mor(\mathcal{C})$; and
\item[(4)] the \emph{composition} function $\circ: \Mor(\mathcal{C}) \times_{\Obj(\mathcal{C})} \Mor(\mathcal{C}) \rightarrow \Mor(\mathcal{C})$, where
$\Mor(\mathcal{C}) \times_{\Obj(\mathcal{C})} \Mor(\mathcal{C})=\left\{(f,g)\in \Mor(\mathcal{C})^2: \dom(f)=\cod(g) \right \}$ is the set of \emph{composable pairs} in $\Mor(\mathcal{C})$;
\end{itemize}
 and satisfying
\begin{itemize}
\item[(1)] $\dom\left( \id(a)\right)=a=\cod\left( \id(a)\right)$ for all $a\in \Obj(\mathcal{C})$;
\item[(2)] $\dom(f \circ g)=\dom(g)$ and $\cod(f \circ g)=\cod(f)$ for all $(f,g)\in  \Mor(\mathcal{C}) \times_{\Obj(\mathcal{C})} \Mor(\mathcal{C})$;
\item[(3)] the \emph{unit law}, $\id\left( \cod(f)\right) \circ f=f=f \circ \id\left( \dom(f) \right)$ for all $f\in \Mor(\mathcal{C})$; and
\item[(4)] \emph{associativity}, $\left( f \circ g \right) \circ h = f \circ \left( g \circ h \right)$.
\end{itemize}
We say the category $\mathcal{C}$ is countable if $\Mor(\mathcal{C})$ is countable.
\end{definition}

Since the unit law implies that $\dom$ and $\cod$ are surjective, if $\mathcal{C}$ is countable, then $\Obj(\mathcal{C})$ is also countable.

\begin{definition}
A \emph{covariant functor} $T$ from a category $\mathcal{C}$ to a category $\mathcal{B}$ is a pair of functions (both denoted $T$); an \emph{object function} $T: \Obj(\mathcal{C}) \rightarrow \Obj(\mathcal{B})$, and a \emph{morphism function} $T: \Mor(\mathcal{C}) \rightarrow \Mor( \mathcal{B})$ satisfying:
\begin{itemize}
\item[(1)] $\dom\left( T(f) \right) =T \left( \dom(f) \right)$ and $\cod\left(T(f)\right)=T\left(\cod(f)\right)$ for all $f\in \Mor(\mathcal{C})$;
\item[(2)] $T\left( \id(a) \right) = \id\left( T(a)\right)$ for all $a\in \Obj(\mathcal{C})$; and
\item[(3)] $T(f) \circ T(g)=T(f \circ g)$ for all $(f,g) \in \Mor(\mathcal{C}) \times_{\Obj(\mathcal{C})} \Mor(\mathcal{C})$.
\end{itemize}
\end{definition}

We regard the morphisms of a category as arrows connecting their domain object to their codomain object. Composition of morphisms then becomes concatenation of arrows. Informally, a covariant functor is a map between the arrows of one category to the arrows of another category that preserves connectivity. There is also a notion of a \emph{contravariant functor}. Here we will only be dealing with covariant functors, and so we will just refer to them as \emph{functors}.

Throughout this thesis $\mathbb{N}$ denotes the set $\left\{ 0,1,2,3,4, \cdots\right \}$. We will write $\mathbb{N}\setminus \left\{ 0 \right \}$ for the usual set of natural numbers. If $A$ and $B$ are sets, we write $A\subset B$ if each element of $A$ is a member of $B$; we write $A\subsetneq B$ if $A$ is strictly contained in $B$.

\begin{example}
A monoid $(S,e,\cdot)$ can be regarded as the morphisms of a category $\mathcal{C}$ with a single object $a_0$; that is $\Obj(\mathcal{C})=\left\{a_0\right\}$, $\Mor(\mathcal{C})=S$, $\dom(s)=\cod(s)=a_0$ for all $s\in S$, $\id(a_0)=e$, and $s\circ t =s\cdot t$ for all $s,t\in S$. For us, the most important category arising from a monoid is $\mathbb{N}^k$ under addition for $k\in \mathbb{N}\setminus \left\{ 0 \right \}$.
\end{example}

\section{Definitions and notation}

In this section we recall the definition of a $k$-graph (higher-rank graph) and establish the standard notation for dealing with them.

We regard $\mathbb{N}^k$ as a monoid under addition, with additive identity $0$. We define a partial order $\leq$ by $n\leq m$ if and only if $m-n\in \mathbb{N}^k$. We write $n<m$ if $m-n\in \mathbb{N}^k \setminus \left\{ 0 \right \}$. Fix $m,n\in \mathbb{N}^k$. We write $n_i$ for the $i^{\text{th}}$ coordinate of $n$. We write $m\vee n$ for the coordinate maximum of $m$ and $n$ and $m \wedge n$ for there coordinate minimum; that is, $(m\vee n )_i=\max\left\{ m_i,n_i \right \}$ and $(m\wedge n)_i=\min\left\{ m_i, n_i \right \}$.

\begin{definition}[\cite{KP2000}, Definition 1.1]\label{def:KPgraphs}
A $k$\emph{-graph} (or higher-rank graph) is a pair $(\Lambda,d)$ that consists of a countable category $\Lambda$ and a functor $d:\Lambda \rightarrow \mathbb{N}^k$, which satisfies the \emph{factorisation property}: for every $\lambda \in \Mor\left(\Lambda\right)$ and $n,m\in \mathbb{N}^k$ with $d(\lambda)=m+n$, there are unique elements $\mu,\nu \in \Mor\left(\Lambda\right)$ such that $\lambda = \mu \nu$ and $d(\mu)=m$, $d(\nu)=n$.
\end{definition}

Suppose $\lambda \in \Mor(\Lambda)$ and $d(\lambda)=l$. If $0 \leq n \leq m \leq l$, then two applications of the factorisation property ensure there are unique morphisms $\mu, \nu, \tau \in \Mor(\Lambda)$ such that $\lambda = \mu \nu \tau$ and $d(\mu)=n$, $d(\nu)=m-n$, $d(\tau)=l-m$. We write $\lambda(0,n)$ for $\mu$, $\lambda(n,m)$ for $\nu$ and $\lambda(m,l)$ for $\tau$.

Let $(\Lambda,d)$ be a $k$-graph. We claim that \begin{equation}\left\{ \lambda \in \Mor(\Lambda):d(\lambda)=0 \right\} = \left\{ \id(v):v\in \Obj(\Lambda)\right\}\label{eq:morandobj}.\end{equation} To see this, recall that $d$ is a functor and $\mathbb{N}^k$ has just one object $a_o$. Thus, for each $v\in \Obj(\Lambda)$ we have $d(\id(v))=\id(d(v))=\id(a_o)=0$. This shows the $``\supset"$ containment in \eqref{eq:morandobj}. Fix $\lambda \in \Mor(\Lambda)$ with $d(\lambda)=0$. Write $v$ for $\cod(\lambda)$. As $0+0=0$, the factorisation property implies that there exist unique morphisms $\mu,\nu\in \Mor(\Lambda)$ such that $d(\mu)=d(\nu)=0$ and $\lambda =\mu \nu$. Both $\mu=\id(v)$, $\nu =\lambda$ and $\mu=\lambda$, $\nu=\id(\dom(\lambda))$ provide such factorisations. The uniqueness of factorisations implies that $\lambda =\id(v)$, showing the $``\subset "$ containment in \eqref{eq:morandobj}.

Due to \eqref{eq:morandobj}, we regard $\Obj(\Lambda)$ as a subset of $\Mor(\Lambda)$ and we regard $\Lambda$ as consisting entirely of its morphisms; that is, we write $\lambda \in \Lambda$ for $\lambda \in \Mor(\Lambda).$

For $n\in \mathbb{N}^k$, we write
$$\Lambda^n= \left\{ \lambda \in \Lambda : d(\lambda)=n\right\}.$$
In particular, $\Lambda^0=\left\{ \lambda \in \Lambda : d(\lambda)=0\right \}$, which we regard as the collection of vertices via \eqref{eq:morandobj}. We write $s: \Lambda \rightarrow \Lambda^0$ for the domain function $\dom$, and $r:\Lambda \rightarrow \Lambda^0$ for the codomain function $\cod$. Given $\lambda \in \Lambda$, we call $s(\lambda)$ the source of $\lambda$ and $r(\lambda)$ the range of $\lambda$. For $\lambda \in \Lambda$ and $E\subset \Lambda$, we write $\lambda E$ for $\left\{ \lambda \mu : \mu\in E, s(\lambda)=r(\mu)\right\}$ and $E\lambda$ for $\left\{ \mu \lambda: \mu \in E, s(\mu)=r(\lambda)\right\}$. We generally reserve lower-case Greek letters $(\lambda,\mu,\nu, \sigma,\tau, \zeta,\eta,...)$ for paths in $\Lambda \setminus \Lambda^0$. Notable exceptions include: $\delta$, which we use for the Kronecker delta; $\xi$, which we use for basis elements of Hilbert spaces; and $\gamma$, which we use for the gauge-action (see Section~\ref{chptr:gauge_action}). Lower-case English letters $(u,v,w,...)$ typically denote elements of $\Lambda^0$.

\begin{definition}
Let $(\Lambda,d)$ be a $k$-graph. For $\mu,\nu \in \Lambda$ we define
$$\Lambda^{\text{min}}(\mu,\nu):= \left \{ (\alpha,\beta)\in \Lambda \times \Lambda : \mu\alpha=\nu \beta \text{ and } d(\mu\alpha)=d(\mu)\vee d(\nu) \right \}.$$
A $k$-graph $(\Lambda,d)$ is said to be \emph{finitely aligned} if the set $\Lambda^{\textrm{min}}(\mu,\nu)$ is finite for each $\mu,\nu \in \Lambda$.
\end{definition}

\begin{definition}\label{def:MCEEXT}
Let $(\Lambda,d)$ be a $k$-graph. For $\mu,\nu \in \Lambda$, define the set of \emph{minimal common extensions} of $\mu$ and $\nu$, denoted $\MCE_\Lambda(\mu,\nu)$, to be collection of $\lambda\in \Lambda$ such that $\lambda= \mu \mu' = \nu \nu'$ for some $\mu',\nu'\in \Lambda$ with $d(\lambda)=d(\mu)\vee( \nu)$. If the $k$-graph $\Lambda$ is clear from context, we write $\MCE$ for $\MCE_\Lambda$. Fix $v\in \Lambda^0$. For $E\subset v \Lambda$ and $\mu\in v\Lambda$ we define $$\Ext(\mu;E):=\left\{ \alpha : (\alpha,\beta)\in \Lambda^{\text{min}}(\mu,\lambda) \text{ for some } \lambda \in E \right \}.$$
\end{definition}

\begin{proposition}\label{prop:MCE}
Let $(\Lambda,d)$ be a finitely aligned $k$-graph and suppose $\mu,\nu \in \Lambda$. The map $$\phi:\Lambda^{\text{min}}(\mu,\nu) \rightarrow \MCE(\mu,\nu)$$ given by $\phi(\mu',\nu') = \mu \mu'$ is a bijection. In particular, $\MCE(\mu,\nu)$ is finite for all $\mu,\nu \in \Lambda$.
\end{proposition}

\begin{proof}
Fix $\lambda \in \MCE(\mu,\nu)$. Let $N=d(\mu) \vee d(\nu)$. Then $(\lambda(d(\mu),N),\lambda(d(\nu),N)) \in \Lambda^{\text{min}}(\mu,\nu)$, and $\phi\left(\lambda(d(\mu),N),\lambda(d(\nu),N)\right)=\mu \lambda(d(\mu),N)=\lambda$. So $\phi$ is surjective. To see that it is injective suppose $\phi(\mu', \nu') = \phi(\mu'' , \nu'')$. The factorisation property implies that $\mu' = \mu''$. Now as $\nu \nu' = \mu \mu' = \mu \mu'' = \nu \nu''$, another application of the factorisation property shows that $(\mu',\nu')=(\mu'',\nu'')$. So $\phi$ is injective and hence a bijection. Since $\Lambda$ is finitely aligned, $|\MCE(\mu,\nu)|=|\Lambda^{\text{min}}(\mu,\nu)|<\infty$ for all $\mu,\nu \in \Lambda$.
\end{proof}

\section{Graph morphisms}

For each $m\in \left(\mathbb{N}\cup \left\{ \infty \right \}\right)^k$, we define a $k$-graph $\Omega_{k,m}$. For a $k$-graph $(\Lambda,d)$, we define graph morphisms $x: \Lambda \rightarrow \Omega_{k,m}$, which we regard as potentially-infinite paths in $\Lambda$. We use such paths for our analysis of the ideal structure of $\mathcal{T}C^*(\Lambda,c)$ in Section~\ref{chptr:Gap_projections}.
\begin{example}[\cite{KP2000}, Examples 1.7(ii)]
Let $k\in \mathbb{N} \setminus \left \{ 0 \right \}$ and $m\in \left( \mathbb{N} \cup \left \{ \infty \right\} \right)^k$. Define
$$\Obj(\Omega_{k,m}):= \left \{ n\in \mathbb{N}^k: n_i \leq m_i \text{ for all } i\right \},$$
$$\Mor(\Omega_{k,m}):=\left\{ (n_1,n_2)\in \Obj(\Omega_{k,m})\times \Obj(\Omega_{k,m}): n_1\leq n_2 \right \},$$
$$\cod((n_1,n_2)):=n_1 \quad \text{and} \quad \dom((n_1,n_2)):=n_2.$$
For $n_1\leq n_2\leq n_3\in \Obj(\Omega_{k,m})$, we define $\id(n_1):=(n_1,n_1)$ and $(n_1,n_2)\circ (n_2,n_3)=(n_1,n_3)$. Then $\Omega_{k,m}$ is a countable category. Define $d((n_1,n_2)):=n_2-n_1$. The pair $(\Omega_{k,m},d)$ is a $k$-graph.
\end{example}

\begin{definition}\label{def:graphmorphisms}
Let $(\Lambda_1,d_1)$ and $(\Lambda_2,d_2)$ be a $k$-graphs. A \emph{graph morphism} is a functor $x: \Lambda_1\rightarrow \Lambda_2$ such that $d_2\circ x = d_1$. We define $$\Lambda^*=\left\{ x: \Omega_{k,m} \rightarrow \Lambda : m\in (\mathbb{N}\cup \left\{ \infty\right \})^k,\text{ } x \text{ is a graph morphism} \right \}.$$
For $x\in \Lambda^*$, we write $r(x)$ for $x(0)$ and $d(x)$ for $m$.
\end{definition}

For $m\in \mathbb{N}^k$, there is a bijection between $\Lambda^m$ and the set of graph morphisms $x:\Omega_{k,m}\rightarrow \Lambda^m$: if $d(\lambda)=m$, set $x_\lambda(p,q)=\lambda(p,q)$ (see the paragraph following Definition~\ref{def:KPgraphs}). The uniqueness clause in factorisation property implies that $x_\lambda(p,q)$ well-defined. If $m$ contains some infinite coordinates, then graph morphisms $x:\Omega_{k,m}\rightarrow \Lambda^m$ are thought of as partially infinite paths in $\Lambda$.

\chapter{Twisted relative Cuntz-Krieger algebras} %
\label{chptr:Twisted_relative_Cuntz-Krieger_algebras}

For each finitely aligned $k$-graph $\Lambda$ and each $\mathbb{T}$-valued 2-cocycle $c$ on $\Lambda$, we investigate collections of partial isometries indexed by $\Lambda$ called Toeplitz-Cuntz-Krieger $(\Lambda,c)$-families. We show that there is a universal $C^*$-algebra $\mathcal{T}C^*(\Lambda,c)$, called the twisted Toeplitz algebra, generated by a Toeplitz-Cuntz-Krieger $(\Lambda,c)$-family $\left\{ s_\mathcal{T}(\lambda):\lambda \in \Lambda \right \}$. For each finite exhaustive set $E\subset v \Lambda$, where $v\in \Lambda^0$, we define the associated gap projection $Q(s_\mathcal{T})^E$ in $\mathcal{T}C^*(\Lambda,c)$. For a collection $\mathcal{E}$ of finite exhaustive sets, we investigate the associated twisted relative Cuntz-Krieger algebra $C^*(\Lambda,c;\mathcal{E})$. By definition, $C^*(\Lambda,c;\mathcal{E})$ is the quotient of $\mathcal{T}C^*(\Lambda,c)$ by the smallest closed ideal $J_\mathcal{E}$ containing all the gap projections associated to elements of $\mathcal{E}$. We define relative Cuntz-Krieger $(\Lambda,c;\mathcal{E})$-families, which are Toeplitz-Cuntz-Krieger $(\Lambda,c)$-families satisfying an additional relation, and show that $C^*(\Lambda,c;\mathcal{E})$ is universal for such a family $\left\{ s_\mathcal{E}(\lambda):\lambda \in \Lambda \right \}$. The gap projections associated to elements of $\mathcal{E}$ are, by definition of $J_\mathcal{E}$, zero in $C^*(\Lambda,c;\mathcal{E})$. Using a concrete representation of $\mathcal{T}C^*(\Lambda,c)$ on $\ell^2(\Lambda)$, we determine for which $\mathcal{E}$ the converse is true; that is, for which $\mathcal{E}$ we have if $Q(s_\mathcal{E})^E \in J_\mathcal{E}$ then $E\in \mathcal{E}$.

\section{Toeplitz-Cuntz-Krieger $(\Lambda,c)$-families}
\label{chptr:Toeplitz-Cuntz-Krieger_families}

We recall the notion of a $\mathbb{T}$-valued 2-cocycle on a $k$-graph $\Lambda$. We identity collections of partial isometries indexed by $\Lambda$ called Toeplitz-Cuntz-Krieger $(\Lambda,c)$-families, and study their basic properties. We construct a Toeplitz-Cuntz-Krieger $(\Lambda,c)$-family in $\mathcal{B}(\ell^2(\Lambda))$.

\begin{definition}
Let $(\Lambda,d)$ be a $k$-graph. A $\mathbb{T}$-valued \emph{2-cocycle} on $\Lambda$ is a map $$c: \left\{ (\lambda,\mu)\in\Lambda \times \Lambda : s(\lambda)=r(\mu) \right\}\rightarrow \mathbb{T}$$ that satisfies:
\begin{enumerate}
\item[(C1)] $c(\lambda, \mu)c(\lambda \mu, \nu)= c(\mu, \nu)c( \lambda, \mu \nu)$ whenever $s(\lambda)=r(\mu)$ and $s(\mu)=r(\nu)$; and
\item[(C2)] $c(\lambda,s(\lambda))=c(r(\lambda),\lambda)=1$ for all $\lambda \in \Lambda$.
\end{enumerate}
We call $(C1)$ the \emph{2-cocycle identity}. We write $\underline{Z}^2(\Lambda,\mathbb{T})$ for the set of all $\mathbb{T}$-valued 2-cocycles on $\Lambda$.
\end{definition}

\begin{definition}\label{def:toelambdacfamily}
Let $(\Lambda,d)$ be a finitely aligned $k$-graph and let $c\in\underline{Z}^2(\Lambda,\mathbb{T})$. A subset $\lbrace t_\lambda : \lambda \in \Lambda \rbrace$ of a $C^*$-algebra is a \emph{Toeplitz-Cuntz-Krieger} $(\Lambda,c)$-family if:
\begin{enumerate}
\item[(TCK1)] $\lbrace t_v : v\in \Lambda^0 \rbrace$ is a set of mutually orthogonal projections;
\item[(TCK2)] $t_{\mu}t_{\nu}=c(\mu,\nu) t_{\mu \nu}$ whenever $s(\mu)=r(\nu)$;
\item[(TCK3)] $t_{\lambda}^*t_{\lambda} = t_{s(\lambda)}$ for all $\lambda \in \Lambda$; and
\item[(TCK4)] $ t_\mu t^*_\mu t_\nu t^*_\nu
   = \displaystyle\sum_{\lambda \in \textrm{MCE}(\mu,\nu)} t_{\lambda} t^*_{\lambda}$ for all $\mu,\nu \in \Lambda$.
 \end{enumerate}

\end{definition}

The following lemma lists some direct consequences of Definition~\ref{def:toelambdacfamily}.
\begin{lemma}\label{handylemma}
Let $(\Lambda,d)$ be a finitely aligned $k$-graph and let $c\in\underline{Z}^2(\Lambda,\mathbb{T})$. Suppose that $\lbrace t_\lambda : \lambda \in \Lambda \rbrace$ is a Toeplitz-Cuntz-Krieger $(\Lambda,c)$-family. Then:
\begin{enumerate}

\item[(1)] $t_{\lambda}^* t_{\mu}=\displaystyle\sum_{(\alpha,\beta)\in \Lambda^{\textrm{min}}(\lambda,\mu)}\overline{c(\lambda,\alpha)}c(\mu,\beta) t_{\alpha} t_{\beta}^*$ for all $\lambda , \mu \in \Lambda$;

\item[(2)] the range projections $\lbrace t_\lambda t^*_\lambda : \lambda \in \Lambda \rbrace$ pairwise commute;

\item[(3)] if $s(\mu)\neq s(\nu)$ then $t_\mu t_\nu^*=0$;

\item[(4)] if $s(\mu)=s(\nu)$ and $s(\zeta)=s(\eta)$ then $$t_\mu t_\nu^* t_\eta t_\zeta^* =\sum_{(\nu',\eta')\in \Lambda^{\text{min}}(\nu,\eta) } c(\eta,\eta')c(\mu,\nu')\overline{c(\nu,\nu')}\overline{c(\zeta,\eta')} t_{\mu \nu'}t_ {\zeta \eta'}^* ;$$

\item[(5)] $C^*(\lbrace t_\lambda : \lambda \in \Lambda \rbrace)=\overline{\linspan}\lbrace t_\mu t^*_\nu : \mu,\nu \in  \Lambda, s(\mu)=s(\nu)  \rbrace;$

\item[(6)] $\lbrace t_\lambda t_\lambda^* : d(\lambda)=n \rbrace$ is a set of mutually orthogonal projections for each $n\in \mathbb{N}^k$; and

\item[(7)] For $\lambda \in \Lambda$, $t_\lambda \neq 0$ if and only if $t_{s(\lambda)}\neq 0$.
\end{enumerate}
\end{lemma}

\begin{proof}
For $(1)$, fix $\lambda,\mu \in \Lambda$. We calculate
\begin{align*}
t^*_\lambda t_\mu &= t^*_\lambda t_\lambda t^*_\lambda t_\mu t^*_\mu t_\mu \qquad \text{since } t_\lambda \text{ and } t_\mu \text{ are partial isometries}\\
&= t^*_\lambda \left( \sum_{\sigma\in \MCE(\lambda,\mu)} t_\sigma t^*_\sigma \right) t_\mu \qquad \text{ by (TCK4)}\\
&= t^*_\lambda \left( \sum_{(\alpha,\beta)\in \Lambda^{\text{min}}(\lambda,\mu)} t_{\lambda \alpha} t^*_{\mu \beta}\right) t_\mu\qquad \text{ by Proposition~\ref{prop:MCE}}\\
&=t^*_\lambda \left( \sum_{(\alpha,\beta)\in \Lambda^{\text{min}}(\lambda,\mu)} \overline{c(\lambda,\alpha)}c(\mu,\beta) t_{\lambda} t_{\alpha} t_{\beta}^*  t^*_\mu   \right) t_\mu \qquad \text{by (C1)}\\
&=  \sum_{(\alpha,\beta)\in \Lambda^{\text{min}}(\lambda,\mu)} \overline{c(\lambda,\alpha)}c(\mu,\beta) t_{s(\lambda)} t_{\alpha} t_{\beta}^*  t_{s(\mu)}    \qquad \text{by (TCK3)}\\
&=  \sum_{(\alpha,\beta)\in \Lambda^{\text{min}}(\lambda,\mu)} \overline{c(\lambda,\alpha)}c(\mu,\beta) c(s(\lambda),\alpha) \overline{c(s(\mu),\beta)} t_{s(\lambda)\alpha} t_{s(\mu) \beta}^*  \qquad  \text{by (TCK2)}  \\
&=  \sum_{(\alpha,\beta)\in \Lambda^{\text{min}}(\lambda,\mu)} \overline{c(\lambda,\alpha)}c(\mu,\beta)  t_{\alpha} t_{\beta}^*  \qquad\text{by (C2).}    \\
\end{align*}
Since $\MCE(\lambda,\mu)=\MCE(\mu,\lambda)$ for all $\lambda , \mu \in \Lambda$, statement $(2)$ is a consequence of (TCK4). For (3), fix $\mu,\nu \in \Lambda$ such that $s(\mu)\neq s(\nu)$. It follows from (TCK1) that $t_\lambda t_\mu^*=t_\lambda t_\lambda^*  t_\lambda t_\mu^*  t_\mu t_\mu^*=t_\lambda t_{s(\lambda)} t_{s(\mu)} t_\mu^*=0$, giving (3).

For $(4)$, fix $\mu ,\nu,\eta,\zeta \in \Lambda$ such that $s(\mu)=s(\nu)$ and $s(\zeta)=s(\eta)$. Then
\begin{align*}
t_\mu t_\nu^* t_\eta t_\zeta^* &= t_\mu \left(  \sum_{(\nu',\eta' )\in \Lambda^{\text{min}}(\nu,\eta)} \overline{c(\nu,\nu')}c(\eta,\eta')t_{\nu'}t^*_{\eta'} \right)  t_\zeta^*\qquad \text{by part (1)}\\
&=  \sum_{(\nu',\eta' )\in \Lambda^{\text{min}}(\nu,\eta)} \overline{c(\nu,\nu')}c(\eta,\eta')t_\mu t_{\nu'} ( t_\zeta t_{\eta'})^*\\
&=  \sum_{(\nu',\eta' )\in \Lambda^{\text{min}}(\nu,\eta)}c(\mu,\nu')c(\eta,\eta')\overline{c(\zeta,\eta')c(\nu,\nu')} t_{\mu\nu'} t_{\zeta \eta'}^* \qquad \text{by (C1).}\\
\end{align*}

We now show (5). We have $$X:=\linspan\lbrace t_\mu t^*_\nu : \mu,\nu \in \Lambda, s(\mu)= s(\nu) \rbrace \subset C^*(\lbrace t_\lambda : \lambda \in \Lambda \rbrace),$$ and so it suffices to prove that $\overline{X}$ is a $C^*$-algebra. For this we need only prove that $X$ is closed under multiplication and taking adjoints. Since $(t_\mu t_\nu^*)^*=t_\nu t_\mu^*$ for each $\mu,\nu \in \Lambda$, the subspace $X$ is closed under adjoints. Part (4) implies that $X$ is closed under multiplication.

Fix $\lambda \in \Lambda$. We have $(t_\lambda t_\lambda^*)^2=t_\lambda t_{s(\lambda)} t_\lambda^*=c(\lambda,s(\lambda))t_\lambda t_\lambda^*=t_\lambda t_\lambda^*=(t_\lambda t_\lambda^*)^*$, so $t_\lambda t_\lambda^*$ is a projection. So for (6), it suffices to show that for each $n\in \mathbb{N} ^k$ the projections $t_\lambda t_\lambda^*$, $d(\lambda)=n$ are mutually orthogonal. Fix $n\in \mathbb{N} ^k$. Suppose that $d(\mu)=d(\nu)=n$. We show that $\MCE(\mu,\nu)=\left \{ \mu \right \}$ if $\mu=\nu$ and $\MCE(\mu,\nu)=\emptyset$ otherwise. To see this suppose $\lambda \in \MCE(\mu,\nu)$. There exist $\mu',\nu'\in \Lambda$ such that $\lambda = \mu \mu'= \nu \nu'$. Since $\lambda \in \MCE(\mu,\nu)$ we have $d(\lambda)=d(\mu) \vee d(\nu)=n$. The factorisation property therefore implies that $\mu'=\nu'=s(\mu)$, so $\lambda=\mu=\nu$ as claimed. Now (TCK4) gives $t_\mu t_\mu^* t_\nu t_\nu^*=\delta_{\mu,\nu}t_\mu t_\mu^*$. For (7), the $C^*$-identity implies that for $\lambda \in \Lambda$ $$\Vert t_{s(\lambda)} \Vert=\Vert t_\lambda^* t_\lambda \Vert =\Vert t_\lambda \Vert^2.$$
So for $\lambda \in \Lambda$, $t_\lambda \neq 0$ if and only if $t_{s(\lambda)}\neq0$.
\end{proof}

\begin{remark}\label{re:handyre} If $\left\{ t_\lambda : \lambda \in \Lambda \right \}$ is a Toeplitz-Cuntz-Krieger $(\Lambda,c)$-family, then $t_\lambda t_\mu^* \neq 0 $ if and only if $s(\lambda)=s(\mu)$ and $t_{s(\mu)}\neq 0$. To see this we first need a calculation. Suppose that $s(\lambda)=s(\mu)$, two applications of the $C^*$-identity gives
\begin{equation}\Vert t_\lambda t_\mu^* \Vert ^2=\Vert t_\mu t_{s(\lambda)}t_\mu^* \Vert=\Vert t_\mu^* \Vert^2=\Vert t_\mu \Vert^2.\label{eq:twocstarsaps}\end{equation}
Suppose $t_\lambda t_\mu^* \neq 0$. Lemma~\ref{handylemma}(3) implies that $s(\lambda)=s(\mu)$. Equation~\eqref{eq:twocstarsaps} implies that $t_\mu \neq 0$, and so $t_{s(\mu)}\neq 0$ by Lemma~\ref{handylemma}(7). Suppose $s(\lambda)=s(\mu)$ and $t_{s(\mu)} \neq 0$. Lemma~\ref{handylemma}(7) implies that $t_\mu \neq 0$. Equation~\ref{eq:twocstarsaps} implies that $t_\lambda t_\mu^*\neq 0$.

\end{remark}

The following Lemma shows that replacing (TCK3) and (TCK4) in Definition~\ref{def:toelambdacfamily} by the relation in Lemma~\ref{handylemma}(1) would result in an equivalent definition of a Toeplitz-Cuntz-Krieger $(\Lambda,c)$-family.

\begin{lemma}\label{le:equivofTCK}
Let $(\Lambda,d)$ be a finitely aligned $k$-graph and let $c\in \underline{Z}^2(\Lambda,\mathbb{T})$. Suppose that $\left \{ t_\lambda : \lambda \in \Lambda \right \}$ is a collection of partial isometries in a $C^*$-algebra satisfying (TCK1) and (TCK2). Then $\left \{ t_\lambda : \lambda \in \Lambda \right \}$ is a Toeplitz-Cuntz-Krieger $(\Lambda,c)$-family if and only if \begin{equation}t_{\lambda}^* t_{\mu}=\sum_{(\alpha,\beta)\in \Lambda^{\textrm{min}}(\lambda,\mu)}\overline{c(\lambda,\alpha)}c(\mu,\beta) t_{\alpha} t_{\beta}^*\label{eq:TCK3.1}\end{equation} for all $\lambda , \mu \in \Lambda$.
\end{lemma}

\begin{proof}
Suppose that $\left \{ t_\lambda : \lambda \in \Lambda \right \}$ is a Toeplitz-Cuntz-Krieger $(\Lambda,c)$-family. Relation \eqref{eq:TCK3.1} follows from Lemma~\ref{handylemma}(1). Suppose that \eqref{eq:TCK3.1} holds. We show (TCK3). Fix $\lambda \in \Lambda$. Equation~\eqref{eq:TCK3.1} implies that
\begin{align*}
t_\lambda^*t_\lambda &=\sum_{(\alpha,\beta)\in \Lambda^\text{min}(\lambda,\lambda)}\overline{c(\lambda,\alpha)}c(\lambda,\beta)t_\alpha t_\beta\\
&=\overline{c(\lambda,s(\lambda))}c(\lambda,s(\lambda))t_{s(\lambda)} t_{s(\lambda)}^*\\
&=t_{s(\lambda)} \qquad \text{by (TCK1)}.
\end{align*}
So $\left \{ t_\lambda : \lambda \in \Lambda \right \}$ satisfies (TCK3). We now show (TCK4). Fix $\lambda,\mu \in \Lambda.$ Then
\begin{align*}
t_\lambda t_\lambda^* t_\mu t_\mu^*&=t_\lambda \left( \sum_{(\alpha,\beta)\in \Lambda^{\textrm{min}}(\lambda,\mu)}\overline{c(\lambda,\alpha)}c(\mu,\beta) t_{\alpha} t_{\beta}^*\right) t_\mu^* \qquad \text{by } \eqref{eq:TCK3.1}\\
&= \sum_{(\alpha,\beta)\in \Lambda^{\textrm{min}}(\lambda,\mu)}\overline{c(\lambda,\alpha)}c(\lambda,\alpha)c(\mu,\beta) \overline{c(\mu,\beta)}t_{\lambda\alpha} t_{\mu\beta}^* \qquad \text{by (TCK2)}\\
&=\sum_{\sigma\in \MCE(\lambda,\mu)} t_\sigma t_\sigma^* \qquad \text{by Proposition \ref{prop:MCE}}.
\end{align*}
So $\left \{ t_\lambda : \lambda \in \Lambda \right \}$ satisfies (TCK4), and hence is a Toeplitz-Cuntz-Krieger $(\Lambda,c)$-family.
\end{proof}

To construct a Toeplitz-Cuntz-Krieger $(\Lambda,c)$-family in $\mathcal{B}(\ell^2(\Lambda))$, we need a technical lemma.

\begin{lemma}\label{le: adjointextention}
Let $\mathcal{H}$ be a separable Hilbert space with basis $\lbrace \xi_\alpha : \alpha \in I \rbrace$. Suppose $T\in \mathcal{B}(\mathcal{H})$ and that $S: \linspan \lbrace \xi_\alpha : \alpha \in I \rbrace \rightarrow\linspan \lbrace \xi_\alpha : \alpha \in I \rbrace $ is linear and satisfies $\langle  \xi_\alpha \mid S \xi_\beta \rangle = \langle T \xi_\alpha \mid \xi_\beta \rangle$ for all $\alpha,\beta \in I$. Then $S$ is bounded with $\Vert S \Vert \leq \Vert T \Vert $, and the unique extension of $S$ to $\mathcal{H}$ is $T^*$, the adjoint of $T$.
\end{lemma}

\begin{proof}
Fix $h\in \linspan \left\{ \xi_\alpha : \alpha \in I \right\}$ and $k\in \mathcal{H}$. There is finite $F\subset I$ such that $h=\sum_{\alpha \in F} h_\alpha \xi_\alpha$. We calculate
\begin{align}
\nonumber\langle k \mid  Sh\rangle &= \left \langle \sum_{\beta \in I} \langle k \mid \xi_\beta \rangle \xi_\beta \Big| \text{ }Sh \right \rangle \qquad \text{ by the reconstruction formula}\\
&= \sum_{\beta \in I} \langle k \mid \xi_\beta \rangle \langle  \xi_\beta\mid Sh  \rangle \nonumber \\
&= \sum_{\beta \in I}  \langle k \mid \xi_\beta \rangle \left \langle \xi_\beta \Big | \text{ } S\left( \sum_{\alpha \in F}h_\alpha \xi_\alpha \right)  \right \rangle \nonumber\\
&= \sum_{\beta \in I}\sum_{\alpha \in F}  \langle k \mid \xi_\beta \rangle h_\alpha \langle \xi_\beta  \mid S \xi_\alpha \rangle \qquad \text{ as } S \text{ is linear}\nonumber\\
&= \sum_{\beta \in I}\sum_{\alpha \in F}   \langle k \mid \xi_\beta \rangle h_\alpha \langle T\xi_\beta   \mid \xi_\alpha  \rangle\qquad \text{ by hypothesis }\nonumber\\
&= \langle Tk\mid  h \rangle \qquad \text{since } T \text{ is linear and continuous. } \label{eq:hilbert1}
\end{align}
Hence
\begin{align*}
\Vert Sh \Vert &= \sup_{\Vert k \Vert =1} \langle Sh \mid k \rangle\\
&= \sup_{\Vert k \Vert =1} \langle h \mid Tk \rangle \qquad \text{ by } \eqref{eq:hilbert1}\\
&\leq\sup_{ \Vert k \Vert =1} \Vert Tk \Vert \Vert h \Vert \qquad \text{ by the Cauchy-Schwarz inequality}\\
&\leq \Vert T \Vert \Vert h \Vert.
\end{align*}
So $S$ is bounded with $\Vert S \Vert \leq \Vert T \Vert $. Denote the unique extension of $S$ to $\mathcal{H}$ by $\overline S$. Now \eqref{eq:hilbert1} gives
\begin{equation} \langle k \mid \overline S h \rangle = \langle T k \mid h \rangle  \text{ for all }h \in \linspan \lbrace \xi_\alpha : \alpha \in I \rbrace \text{ and } k\in \mathcal{H}. \label{eq:adjoint1} \end{equation}
For fixed $k\in \mathcal{H}$, the maps $h \mapsto \langle k \mid \overline S h \rangle$ and $h \mapsto \langle  Tk \mid h \rangle$ on $\mathcal{H}$ are continuous, and by \eqref{eq:adjoint1} agree on a dense subset of $\mathcal{H}$. Therefore they agree on all of $\mathcal{H}$; that is, $ \langle h \mid \overline S k \rangle = \langle T h \mid k \rangle  \text{ for all } h,k\in \mathcal{H}$. By the uniqueness of adjoints we have $\overline S = T^*$.
\end{proof}

\begin{proposition}\label{toeplitzrepresentation}
Let $(\Lambda,d)$ be a finitely aligned $k$-graph and let $c\in\underline{Z}^2(\Lambda,\mathbb{T})$. For $\lambda \in \Lambda$, let $\xi_\lambda: \Lambda \rightarrow \mathbb{C}$ be the point mass function for $\lambda$. Consider the Hilbert space $\ell^2(\Lambda) =\overline{\linspan}\left\{ \xi_\alpha : \alpha \in \Lambda \right \}$ with the usual inner product. For each $\lambda \in \Lambda$, there is a unique bounded linear operator $T_\lambda$ on $\ell^2(\Lambda)$ such that
\begin{equation} T_\lambda \xi_\alpha =\begin{cases} c(\lambda,\alpha) \xi_{\lambda \alpha} &\text{ if } s(\lambda)=r(\alpha) \\ 0 &\text{ otherwise} \end{cases} \label{eq: toeplitzrepresentation} \end{equation}
for all $\alpha \in \Lambda$. The set $\lbrace T_\lambda : \lambda \in \Lambda \rbrace$ is a Toeplitz-Cuntz-Krieger $(\Lambda,c)$-family in $\mathcal{B}\left( \ell^2(\Lambda)\right)$.
\end{proposition}

\begin{proof}
To see that \eqref{eq: toeplitzrepresentation} defines a bounded linear operator  $T_\lambda$ on $\linspan\left\{ \xi_\alpha : \alpha \in \Lambda \right \}$ for each $\lambda \in \Lambda$, we will show that $T_\lambda$ is a partial isometry and hence $\Vert T_\lambda \Vert =1 $. Fix $\lambda,\alpha,\beta \in \Lambda$. Then
\begin{align*}
\langle T_\lambda \xi_\alpha \mid T_\lambda \xi_\beta \rangle  &= \begin{cases} \langle c(\lambda,\alpha) \xi_{\lambda \alpha} \mid c(\lambda,\beta)\xi_{\lambda \beta} \rangle &\text{ if } s(\lambda)=r(\alpha)=r(\beta) \\ 0 &\text{ otherwise} \end{cases}\\
&= \begin{cases}  c(\lambda,\alpha)\overline{c(\lambda,\beta)}\langle \xi_{\lambda \alpha} \mid \xi_{\lambda \beta} \rangle &\text{ if } s(\lambda)=r(\alpha)=r(\beta) \\ 0 &\text{ otherwise} \end{cases}\\
&= \begin{cases}  c(\lambda,\alpha)\overline{c(\lambda,\beta)} &\text{ if } s(\lambda)=r(\alpha)=r(\beta) \text{ and } \alpha=\beta \\ 0 &\text{ otherwise} \end{cases}\\
&= \begin{cases}  c(\lambda,\alpha)\overline{c(\lambda,\beta)}\langle \xi_\alpha \mid  \xi_\beta \rangle&\text{ if } s(\lambda)=r(\alpha)=r(\beta) \\ 0 &\text{ otherwise.} \end{cases}\\
\end{align*}
So if $s(\lambda)=r(\alpha)$ we have $\Vert T_\lambda \xi_\alpha \Vert^2 = \langle T_\lambda \xi_\alpha \mid T_\lambda \xi_\alpha \rangle = c(\lambda,\alpha)\overline{c(\lambda,\alpha)}\langle \xi_\alpha \mid  \xi_\alpha \rangle=\Vert \xi_\alpha \Vert^2=1$. So if $h=\sum_{i=1}^{n} a_i \xi_{\alpha_i} \in \linspan \lbrace \xi_\alpha : s(\lambda)=r(\alpha) \rbrace$ then,
\begin{align*}
\Vert T_\lambda h \Vert^2 &= \sum_{i,j=1}^na_i \overline{a_j} \langle T_\lambda \xi_{\alpha_i} \mid T_\lambda \xi_{\alpha_j} \rangle\\
&= \sum_{i,j=1}^n c(\lambda,\alpha_i)\overline{c(\lambda,\alpha_j)}a_i \overline{a_j}\langle \xi_{\lambda \alpha_i} \mid  \xi_{\lambda \alpha_j} \rangle\\
&= \sum_{i=1}^n c(\lambda,\alpha_i)\overline{c(\lambda,\alpha_i)}a_i \overline{a_i} \qquad\text{by the factorisation property}\\
&=\sum_{i=1}^n a_i \overline{a_i}\\
&=\sum_{i,j=1}^n a_i \overline{a_j}\langle \xi_{\alpha_i} \mid  \xi_{\alpha_j} \rangle\\
&=\langle h \mid h\rangle \\
&= \Vert h \Vert^2.
\end{align*}
If $h=\sum_{i=1}^{n} a_i \xi_{\alpha_i} \in \linspan \lbrace \xi_\alpha : s(\lambda)=r(\alpha) \rbrace^\perp=\linspan \lbrace \xi_\alpha : s(\lambda)\neq r(\alpha) \rbrace$ then $T_\lambda h =0$. So $T_\lambda$ is a partial isometry and is bounded with $\Vert T_\lambda \Vert \leq 1$. The calculation $\Vert T_\lambda \xi_{s(\lambda)}\Vert = \Vert \xi_{\lambda} \Vert =1$ shows that $\Vert T_\lambda \Vert =1$. By continuity, $T_\lambda$ extends uniquely to a partial isometry in $\mathcal{B}\left( \ell^2(\Lambda)\right)$.

We now will show that for each $\lambda \in \Lambda$ the unique linear extension $S_{\lambda} : \linspan\left\{ \xi_\alpha : \alpha \in \Lambda \right \} \rightarrow \linspan\left\{ \xi_\alpha : \alpha \in \Lambda \right \}$ of
\begin{equation}
\xi_{\alpha}\mapsto \begin{cases} \overline{c(\lambda,\alpha')} \xi_{\alpha'} &\text{ if } \alpha=\lambda \alpha' \\ 0 &\text{ otherwise} \end{cases} \nonumber
\end{equation}
extends by continuity to the adjoint of $T_\lambda$. By Lemma~\ref{le: adjointextention} it suffices to show that $\langle T_\lambda \xi_\alpha\mid \xi_\beta \rangle = \langle \xi_\alpha \mid S_\lambda \xi_\beta \rangle$ for all $\alpha, \beta \in \Lambda$. That is,
$$\langle T_\lambda \xi_\alpha \mid  \xi_\beta \rangle= \begin{cases} c(\lambda,\beta') \langle \xi_{\alpha} \mid \xi_{\beta'} \rangle &\text{ if } \beta = \lambda \beta'\\ 0 &\text{ else.} \end{cases}$$
 Fix $\lambda, \alpha,\beta \in \Lambda$ and we calculate,
\begin{align*}
\langle T_\lambda \xi_{\alpha}  \mid \xi_{\beta} \rangle &= \sum_{x \in \Lambda}(T_{\lambda}\xi_{\alpha})(x)\overline{\xi_{\beta}(x)}\\
&=\begin{cases}\sum_{x\in \Lambda} c(\lambda, \alpha) \xi_{\lambda \alpha}(x) \overline{\xi_{\beta}(x)} &\text{ if } r(\alpha)=s(\lambda) \\ 0 &\text{ otherwise}  \end{cases} \\
&= \begin{cases}  c(\lambda, \alpha) &\text{ if }  \beta = \lambda \alpha \\ 0 &\text{ otherwise} \end{cases} \\
&= \begin{cases}  c(\lambda, \beta') &\text{ if } \beta'=\alpha \text{ and } \beta = \lambda \beta' \\ 0 &\text{ otherwise} \end{cases} \\
&= \begin{cases}  c(\lambda, \beta')\langle \xi_\alpha \mid \xi_{\beta'} \rangle  &\text{ if }  \beta = \lambda \beta' \\ 0 &\text{ otherwise.} \end{cases} \\
\end{align*}
Now we will establish the relations (TCK1)-(TCK4) for $ \lbrace T_\lambda: \lambda \in \Lambda \rbrace$. By linearity and continuity it suffices to check these relations on basis vectors. For (TCK1) we show that $T_{s(\lambda)}^2 = T_{s(\lambda)} = T_{s(\lambda)}^*$ and $T_{s(\mu)}T_{s(\nu)}=0$ for each $\lambda,\mu,\nu \in \Lambda$ with $s(\mu) \neq s(\nu)$, so that $\lbrace T_{s(\lambda)} :\lambda \in \Lambda \rbrace$ is a set of mutually orthogonal projections.  Fix $\lambda, \alpha \in \Lambda$. Since $c(r(\alpha),\alpha)=1$ for all $\alpha \in \Lambda$,
\begin{align*}
T^2_{s(\lambda)} \xi_\alpha &= \begin{cases} c(s(\lambda),\alpha)T_{s(\lambda)}\xi_{\alpha} &\text{if } s(\lambda)=r(\alpha) \\ 0 &\text{otherwise} \end{cases}\\
&=\begin{cases} c(r(\alpha)),\alpha)^2 \xi_{\alpha} &\text{if } s(\lambda)=r(\alpha) \\ 0 &\text{otherwise} \end{cases}\\
&=\begin{cases} \xi_{\alpha} &\text{if } s(\lambda)=r(\alpha) \\ 0 &\text{otherwise} \end{cases}\\
&=T_{s(\lambda)}\xi_\alpha
\end{align*}
and
\begin{align*}
T^*_{s(\lambda)} \xi_\alpha &= \begin{cases} \overline{c(s(\lambda),\alpha)}\xi_{\alpha} &\text{if } s(\lambda)=r(\alpha) \\ 0 &\text{otherwise} \end{cases}\\
\begin{cases} \overline{c(r(\alpha),\alpha)}\xi_{\alpha} &\text{if } s(\lambda)=r(\alpha) \\ 0 &\text{otherwise} \end{cases}\\
&=T_{s(\lambda)}\xi_\alpha.
\end{align*}
So $\lbrace T_{s(\lambda)} : \lambda \in \Lambda \rbrace$ is a set of projections. To see that they are mutually orthogonal fix $\mu,\nu,\alpha \in \Lambda$ with $s(\mu) \neq s(\nu)$. Then
\begin{align*}
T_{s(\mu)}T_{s(\nu)} \xi_\alpha &=\begin{cases} T_{s(\mu)}\xi_{\alpha} &\text{if } s(\nu)=r(\alpha) \\ 0 &\text{otherwise} \end{cases}\\&=\begin{cases} \xi_{\alpha} &\text{if } s(\mu)=s(\nu)=r(\alpha) \\ 0 &\text{otherwise} \end{cases}\\&=0.
\end{align*}

This establishes (TCK1). For (TCK2) we must show that $T_\mu T_\nu = c(\mu,\nu)T_{\mu\nu}$ for all $\mu,\nu \in \Lambda$ with $s(\mu)=s(\nu)$. Fix $\mu ,\nu\in \Lambda$ with $s(\mu)=r(\nu)$. Then
\begin{align*}
T_\mu T_\nu \xi_\alpha &= \begin{cases} c(\nu, \alpha) T_\mu \xi_{\nu \alpha} &\text{ if } \nu \alpha \in \Lambda \\ 0 &\text{ otherwise} \end{cases}\\
&= \begin{cases}  c(\nu, \alpha)c(\mu,\nu\alpha) \xi_{ \mu \nu \alpha} &\text{ if } \nu \alpha \in \Lambda \\ 0 &\text{ otherwise} \end{cases}\\
&= \begin{cases}  c(\mu, \nu)c(\mu\nu,\alpha) \xi_{ \mu \nu \alpha} &\text{ if } \nu \alpha \in \Lambda \\ 0 &\text{ otherwise} \end{cases}\\
&=c(\mu,\nu) T_{\mu \nu} \xi_\alpha.
\end{align*}
This establishes (TCK2). For (TCK3) we must show $T_\lambda ^* T_\lambda = T_{s(\lambda)}$ for all $\lambda \in \Lambda$. Fix $\lambda \in \Lambda$. Using (C1), we have
\begin{align*}
T^*_\lambda T_\lambda \xi_\alpha &= \begin{cases} c(\lambda,\alpha)T^*_\lambda \xi_{\lambda \alpha}  &\text{ if }  s(\lambda)=r(\alpha) \\ 0  &\text{ otherwise} \end{cases}\\
&=\begin{cases} c(\lambda,\alpha)\overline{c(\lambda,\alpha)} \xi_{\alpha}  &\text{ if }  s(\lambda)=r(\alpha) \\ 0  &\text{ otherwise} \end{cases}\\
&=\begin{cases}  c(s(\lambda),\alpha)\xi_{\alpha}  &\text{ if }  s(\lambda)=r(\alpha) \\ 0  &\text{ otherwise} \end{cases}\\
&=T_{s(\lambda)}\xi_\alpha.
\end{align*}
This establishes (TCK3). For (TCK4) we must show that $T_\mu T_\mu^* T_\nu T_\nu^* = \displaystyle\sum_{\lambda \in \MCE (\mu,\nu)} T_\lambda T_\lambda^*$ for all $\mu,\nu \in \Lambda$. It will be helpful to consider the following calculation. Fix $\lambda, \alpha \in \Lambda$. Then
\begin{align}
T_\lambda T^*_\lambda \xi_\alpha &= \begin{cases} \overline{c(\lambda,\alpha')}T_\lambda \xi_{\alpha'}  &\text{ if } \alpha = \lambda \alpha' \\ 0  &\text{ otherwise} \end{cases}\label{TTstar}\\
&= \begin{cases} c(\lambda,\alpha')\overline{c(\lambda,\alpha')} \xi_{\lambda \alpha'}  &\text{ if } \alpha = \lambda \alpha' \\ 0  &\text{ otherwise} \end{cases}\nonumber\\
&= \begin{cases} \xi_{\alpha}  &\text{ if } \alpha = \lambda \alpha' \\ 0  &\text{ else.} \end{cases}\nonumber
\end{align}

We will now show that for any $\mu,\nu \in \Lambda$, $\langle T_\mu T_\mu^* T_\nu T_\nu^* \xi_\alpha \mid \xi_\beta \rangle=\langle \displaystyle\sum_{\sigma\in \MCE(\mu,\nu)} T_\sigma T_\sigma^* \xi_\alpha \mid \xi_\beta \rangle$ for all $\alpha,\beta \in \Lambda$ and hence $T_\mu T_\mu^* T_\nu T_\nu^*=\displaystyle\sum_{\sigma\in \MCE(\mu,\nu)} T_\lambda T_\lambda^*$. Fix $\mu,\nu, \alpha, \beta \in \Lambda$. Using \eqref{TTstar} we have,
\begin{align*}
\langle T_\mu T_\mu^* T_\nu T_\nu^* \xi_\alpha \mid \xi_\beta \rangle &=\langle T_\nu T_\nu^* \xi_\alpha \mid  T_\mu T_\mu^* \xi_\beta \rangle \\
&=\begin{cases}  \langle \xi_\alpha \mid \xi_\beta \rangle &\text{if } \beta = \mu \beta' \text{ and } \alpha = \nu \alpha' \\ 0 &\text{otherwise} \end{cases}\\
&=\begin{cases}  1 &\text{if } \beta = \mu \beta'=\alpha = \nu \alpha' \\ 0 &\text{otherwise} \end{cases}\\
&=\begin{cases}  1 &\text{if } \beta = \alpha = \sigma \alpha'' \text{ for some } \sigma\in \MCE(\mu,\nu) \\ 0 &\text{otherwise} \end{cases}\\
&=\begin{cases} \langle \xi_\alpha \mid \xi_\beta \rangle &\text{if } \alpha = \sigma \alpha'' \text{ for some } \sigma\in \MCE(\mu,\nu) \\ 0 &\text{else.} \end{cases}\\
\end{align*}
If $\alpha = \sigma \alpha'$ for some $\sigma \in \MCE(\mu,\nu)$ then the factorisation property ensures that there is exactly one such $\sigma$. So
\begin{align*}
\langle T_\mu T_\mu^* T_\nu T_\nu^* \xi_\alpha \mid \xi_\beta \rangle&=\begin{cases} \langle \xi_\alpha \mid \xi_\beta \rangle &\text{if } \alpha = \sigma \alpha'' \text{ for some } \sigma\in \MCE(\mu,\nu) \\ 0 &\text{otherwise} \end{cases}\\
&=\left \langle \sum_{\sigma\in \MCE(\mu,\nu)} T_\sigma T_\sigma^* \xi_\alpha \mid\xi_\beta \right \rangle,
\end{align*}
establishing (TCK4).

\end{proof}

\section{The twisted Toeplitz algebra}
We show there is a unique universal $C^*$-algebra $\mathcal{T}C^*(\Lambda,c)$, called the twisted Toeplitz algebra, generated by a Toeplitz-Cuntz-Krieger $(\Lambda,c)$-family $\left\{ s_\mathcal{T}(\lambda):\lambda \in \Lambda \right \}$ in Theorem~\ref{toeplitzalgebra}. First we need some technical results.

\label{chptr:twisted_Toeplitz_algebra}
\begin{lemma}\label{le:injectivemap}
Let $(\Lambda,d)$ be a finitely aligned $k$-graph, let $c\in\underline{Z}^2(\Lambda,\mathbb{T})$ and let $\lbrace T_\lambda : \lambda \in \Lambda \rbrace \subset \mathcal{B}(\ell^2(\Lambda))$ be the Toeplitz-Cuntz-Krieger $(\Lambda,c)$-family of Proposition~\ref{toeplitzrepresentation}. Then $\lbrace T_\mu T_\nu^* : \mu , \nu \in \Lambda, s(\mu)=s(\nu) \rbrace$ are linearly independent.
\end{lemma}

\begin{proof}
Let $F$ be a finite subset of $\lbrace (\mu,\nu) : \mu , \nu \in \Lambda \text{ and } s(\mu)=s(\nu) \rbrace$. Let $\lbrace \xi_\lambda : \lambda \in \Lambda \rbrace$ be the usual basis of $\ell^2(\Lambda)$. We show that $\lbrace T_\mu T_\nu^* : (\mu,\nu)\in F \rbrace$ are linearly independent. The proof is by induction on $|F|$. Suppose that $\sum_{(\mu,\nu)\in F} a_{\mu,\nu} T_\mu T_\nu^* h = 0 $ for all $h\in \ell^2(\Lambda)$. If $|F|=1$, let $(\mu_0,\nu_0)$ be the unique pair such that $(\mu_0,\nu_0)\in F$. Then
\begin{align*}
0&=a_{\mu_0,\nu_0}T_{\mu_0} T_{\nu_0} ^* \xi_{\nu_0}\\
&=a_{\mu_0,\nu_0}\overline{c(\nu_0,s(\nu_0))}T_{\mu_0}\xi_{s(\nu_0)}\\
&=a_{\mu_0,\nu_0}c(\mu_0,s(\nu_0))\xi_{\mu_0}\\
&=a_{\mu_0,\nu_0}\xi_{\mu_0}.
\end{align*}
Hence $a_{\mu_0,\nu_0}=0$. As an inductive hypothesis suppose that $\lbrace T_\mu T_\nu^* : (\mu,\nu)\in F \rbrace$ is linearly independent whenever $|F|\leq l$. Suppose $|F|=l+1$. Fix $(\mu_0,\nu_0)\in F$ such that $\nu_0 \neq \sigma \nu_0'$ for all $(\rho,\sigma)\in F \setminus \lbrace \mu_0 , \nu_0 \rbrace$. There is such a pair $(\mu_0,\nu_0)$ since $F$ is finite.
Then
\begin{align*}
0&=\sum_{(\mu,\nu)\in F} a_{\mu,\nu} T_\mu T_\nu^* \xi_{\nu_0} \\
&= \sum_{(\mu,\nu_0)\in F} a_{\mu,\nu_0} \overline{c(\nu_0,s(\nu_0))}T_\mu \xi_{s(\nu_0)} \\
&=\sum_{(\mu,\nu_0)\in F} a_{\mu,\nu_0} c(\mu,s(\nu_0)) \xi_{\mu}\\
&=\sum_{(\mu,\nu_0)\in F} a_{\mu,\nu_0}\xi_{\mu}.
\end{align*}
As the $\xi_\lambda$ are linearly independent with have $a_{\mu,\nu_0}=0$ for each $(\mu,\nu_0)\in F$. Let $$G:=\lbrace (\mu,\nu)\in F:\nu \neq \nu_0 \rbrace.$$ By the inductive hypothesis $\lbrace T_\mu T_\nu^* : (\mu,\nu )\in G\rbrace$ is linearly independent. So
\begin{align*}
0&=\sum_{(\mu,\nu)\in F} a_{\mu,\nu} T_\mu T_\nu^*  \\
&=\sum_{(\mu,\nu)\in G} a_{\mu,\nu} T_\mu T_\nu^* \\
\end{align*}
implies that $a_{\mu,\nu}=0$ for all $(\mu,\nu)\in G$. Thus $a_{\mu,\nu}=0$ for all $(\mu,\nu)\in F$.
\end{proof}

\begin{proposition}\label{constructionofuni}
Let $(\Lambda,d)$ be a finitely aligned $k$-graph and let $c\in\underline{Z}^2(\Lambda,\mathbb{T})$. Let $\Lambda *_s \Lambda = \lbrace (\lambda,\mu)\in \Lambda \times \Lambda : s(\lambda)=s(\mu) \rbrace$ and let $A_0$ be the complex vector space $C_c(\Lambda *_s \Lambda)=\lbrace f :\Lambda *_s \Lambda\rightarrow \mathbb{C}: f \text{ has finite support} \rbrace$ under pointwise addition and scalar multiplication. For each $(\mu,\nu)\in A_0$, let $\xi_{(\mu,\nu)}$ be the point mass function for $(\mu,\nu)$. There is a unique multiplication on $A_0$ satisfying
\begin{equation} \xi_{(\mu,\nu)} \xi_{(\eta,\zeta)} =\sum_{(\nu',\eta')\in \Lambda^{\text{min}}(\nu,\eta) } c(\eta,\eta')c(\mu,\nu')\overline{c(\nu,\nu')}\overline{c(\zeta,\eta')} \xi_{(\mu \nu', \zeta \eta')}\label{eq:multiplicationxi}\end{equation}
for all $(\mu,\nu),(\eta,\zeta)\in \Lambda *_s \Lambda$. There is an involution $* : A_0 \rightarrow A_0$ given by $a^*(\mu,\nu)=\overline{a(\nu,\mu)}$. Under these operations $A_0$ is a *-algebra.
\end{proposition}

\begin{proof}
We check:
\begin{enumerate}
\item[(1)] that $*$ determines a self-inverse conjugate-linear map;
\item[(2)] that the multiplication uniquely extends bilinearly to all of $A_0$;
\item[(3)] that this multiplication is associative and commutes with scalar multiplication; and
\item[(4)] that $(ab)^*=b^*a^*$ for all $a,b\in A_0$.
\end{enumerate}

For (1), fix $a,b\in A_0$, $\omega \in \mathbb{C}$ and $(\mu,\nu)\in \Lambda *_s \Lambda $. We calculate.
\begin{align*}
(a+\omega b)^*(\mu,\nu)&=\overline{(a+\omega b )(\nu,\mu)}\\
&=\overline{a(\nu,\mu)+(\omega b )(\nu,\mu)}\\
&=\overline{a(\nu,\mu)+\omega b(\nu,\mu)}\\
&=\overline{a(\nu,\mu)}+\bar \omega \overline{ b(\nu,\mu)}\\
&=a^*(\mu,\nu)+\bar \omega b^*(\mu,\nu)\\
&=(a^*+\bar \omega b^*)(\mu,\nu)\\
\end{align*}
So $*$ is conjugate linear. The map $*$ is self-inverse since complex conjugation is, establishing (1). The $\xi_{\mu,\nu}$ are linearly independent and span $A_0$, so each $a\in A_0$ has the unique expression
$$a=\sum_{(\eta,\zeta)\in \Lambda *_s \Lambda} a_{\eta,\zeta} \xi_{(\eta,\zeta)}.$$ For (2), the formula
\begin{align}
(ab)(\mu,\nu)&=\left( \left( \sum_{(\eta,\zeta)\in \Lambda *_s \Lambda}a_{\eta,\zeta}\xi_{(\eta,\zeta)}\right)\left( \sum_{(\rho,\sigma)\in \Lambda *_s \Lambda}b_{\rho,\sigma}\xi_{(\rho,\sigma)}\right) \right)(\mu,\nu)\nonumber\\
&= \sum_{(\eta,\zeta),(\rho,\sigma)\in \Lambda *_s \Lambda}a_{\eta,\zeta}b_{\rho,\sigma}(\xi_{(\eta,\zeta)}\xi_{(\rho,\sigma)} )(\mu,\nu)\nonumber\\
&= \sum_{(\eta,\zeta),(\rho,\sigma)\in \Lambda *_s \Lambda}\sum_{(\zeta',\rho')\in \Lambda^{\textrm{min}}(\zeta,\rho)}c(\rho,\rho')c(\eta,\zeta') \overline{c(\zeta,\zeta')c(\sigma,\rho')}a_{\eta,\zeta}b_{\rho,\sigma}\xi_{(\eta \zeta',\sigma \rho')} (\mu,\nu)\nonumber\\
&= \sum_{(\eta,\zeta),(\rho,\sigma)\in \Lambda *_s \Lambda}\sum_{(\zeta',\rho')\in \Lambda^{\textrm{min}}(\zeta,\rho)\atop \mu=\eta\zeta',\nu=\sigma \rho'}c(\rho,\rho')c(\eta,\zeta') \overline{c(\zeta,\zeta')c(\sigma,\rho')}a(\eta,\zeta)b(\rho,\sigma)  \label{eq:multiplication}
\end{align}
extends multiplication to all of $C_c(\Lambda *_s \Lambda)$. Multiplication is bilinear as the vector space operations are pointwise, establishing (2).

For (3), it suffices to check $\xi_{(\mu,\nu)}(\xi_{(\eta,\zeta)}\xi_{(\rho,\sigma)})=(\xi_{(\mu,\nu)}\xi_{(\eta,\zeta)})\xi_{(\rho,\sigma)}$ and that $\xi_{(\mu,\nu)}(z\xi_{(\eta,\zeta)})=z(\xi_{(\mu,\nu)}\xi_{(\eta,\zeta)})$ for all $(\mu,\nu),(\eta,\zeta),(\rho,\sigma)\in \Lambda *_s \Lambda $ and $z\in \mathbb{C}$, the latter follows from \eqref{eq:multiplicationxi}.

Fix $(\mu,\nu),(\eta,\zeta),(\rho,\sigma)\in \Lambda *_s \Lambda$. To show associativity, we use our Toeplitz-Cuntz-Krieger $(\Lambda,c)$-family $\left\{ T_\lambda : \lambda \in \Lambda \right \} $ in $\mathcal{B}\left( \ell^2(\Lambda)\right)$ and Lemma~\ref{le:injectivemap}. Define $\phi : C_c(\Lambda *_s \Lambda) \rightarrow \mathcal{B}\left( \ell^2(\Lambda)\right)$ by $\phi \left( \displaystyle\sum_{(\mu,\nu)\in \Lambda *_s \Lambda} a_{\mu,\nu}\xi_{(\mu,\nu)}\right) = \displaystyle\sum_{(\mu,\nu)\in \Lambda *_s \Lambda} a_{\mu,\nu} T_\mu T_\nu^*$. This is well-defined since the $\xi_{(\mu,\nu)}$ are linearly independent. The map $\phi$ is linear by definition of the vector space operations in $C_c(\Lambda *_s \Lambda)$ and $\mathcal{B}\left( \ell^2(\Lambda)\right)$. We claim that $\phi$ also preserves multiplication. As $\phi$ is linear and multiplication is bilinear, it suffices to show that $\phi \left( \xi_{(\mu,\nu)}\xi_{(\eta,\zeta)}\right)=\phi \left( \xi_{(\mu,\nu)}\right) \phi \left( \xi_{(\eta,\zeta)}\right)$. We have
\begin{align*}
\phi \left( \xi_{(\mu,\nu)}\xi_{(\eta,\zeta)}\right)&=  \phi\left( \sum_{(\nu',\eta')\in \Lambda^{\text{min}}(\nu,\eta) } c(\eta,\eta')c(\mu,\nu')\overline{c(\nu,\nu')}\overline{c(\zeta,\eta')} \xi_{(\mu \nu', \zeta \eta')}\right) \\
&=\sum_{(\nu',\eta')\in \Lambda^{\text{min}}(\nu,\eta) } c(\eta,\eta')c(\mu,\nu')\overline{c(\nu,\nu')}\overline{c(\zeta,\eta')} T_{\mu \nu'} T_{\zeta \eta'}^*\\
&= T_\mu T_\nu^* T_\eta T_\zeta^* \qquad\text{by Lemma } \ref{handylemma}(4)\\
 &=\phi \left( \xi_{(\mu,\nu)}\right) \phi \left( \xi_{(\eta,\zeta)}\right).
\end{align*}
We also claim that $\phi$ is injective. Fix $ \displaystyle\sum_{(\mu,\nu)\in \Lambda *_s \Lambda} a_{\mu,\nu}\xi_{(\mu,\nu)}  \in C_c(\Lambda *_s \Lambda)$. Then
\begin{align*}
\phi\left(  \sum_{(\mu,\nu)\in \Lambda *_s \Lambda} a_{\mu,\nu}\xi_{(\mu,\nu)}\right) = 0 \Longrightarrow  \sum_{(\mu,\nu)\in \Lambda *_s \Lambda} a_{\mu,\nu}T_\mu T_\nu^*=0.
\end{align*}
Lemma~\ref{le:injectivemap} implies that $a_{\mu,\nu}=0$ for all $(\mu,\nu)\in \Lambda *_s \Lambda$. Hence $\sum_{(\mu,\nu)\in \Lambda *_s \Lambda} a_{\mu,\nu}\xi_{(\mu,\nu)}= 0$. We can now use the map $\phi$ to show that the multiplication in $C_c(\Lambda *_s \Lambda)$ is associative. As multiplication in $\mathcal{B}\left( \ell^2(\Lambda)\right)$ is associative, we have
\begin{align*}
\phi \left( \xi_{(\mu,\nu)}\left(\xi_{(\eta,\zeta)}\xi_{(\rho,\sigma)}\right) \right) &=\phi \left( \xi_{(\mu,\nu)}\right) \phi \left(\xi_{(\eta,\zeta)}\xi_{(\rho,\sigma)} \right) \\
&=T_\mu T_\nu^* \left( \left( T_\eta T_\zeta^*\right) \left( T_\rho T_\sigma^*\right) \right)\\
&=\left(\left( T_\mu T_\nu^* \right)\left( T_\eta T_\zeta^*\right) \right) T_\rho T_\sigma^*\\
&=\phi \left( \xi_{(\mu,\nu)}\xi_{(\eta,\zeta)} \right) \phi \left( \xi_{(\rho,\sigma)}\right)\\
&=\phi \left(\left(  \xi_{(\mu,\nu)}\xi_{(\eta,\zeta)} \right) \xi_{(\rho,\sigma)}\right).
\end{align*}
Since $\phi$ is injective we have $\left(  \xi_{(\mu,\nu)}\xi_{(\eta,\zeta)} \right) \xi_{(\rho,\sigma)}=\xi_{(\mu,\nu)}\left(\xi_{(\eta,\zeta)}\xi_{(\rho,\sigma)}\right)$, establishing (3).
For (4), it suffices to check that $\left(\xi_{(\mu,\nu)}\xi_{(\eta,\zeta)}\right)^*=\xi_{(\eta,\zeta)}^* \xi_{(\mu,\nu)}^*$ for all $(\mu,\nu),(\eta,\zeta) \in \Lambda *_s \Lambda$. Fix $(\mu,\nu),(\eta,\zeta) \in \Lambda *_s \Lambda$. We have
\begin{flalign*}
&&\left(\xi_{(\mu,\nu)}\xi_{(\eta,\zeta)}\right)^*&=\left( \sum_{(\nu',\eta')\in \Lambda^{\text{min}}(\nu,\eta) } c(\eta,\eta')c(\mu,\nu')\overline{c(\nu,\nu')}\overline{c(\zeta,\eta')} \xi_{(\mu \nu', \zeta \eta')}\right)^*&\\
&&&= \sum_{(\eta',\nu')\in \Lambda^{\text{min}}(\eta,\nu) } c(\nu,\nu')c(\zeta,\eta')\overline{c(\eta,\eta')}\overline{c(\mu,\nu')} \xi_{(\zeta \eta', \mu \nu')}&\\
&&&=\xi_{(\zeta,\eta)}\xi_{(\nu,\mu)}&\\
&&&=\xi_{(\eta,\zeta)}^* \xi_{(\mu,\nu)}^*.&\qedhere
\end{flalign*}
\end{proof}


We can now construct the twisted Toeplitz algebra $\mathcal{T}C^*(\Lambda,c)$.

\begin{theorem} \label{toeplitzalgebra}
Let $(\Lambda,d)$ be a finitely aligned $k$-graph and let $c\in\underline{Z}^2(\Lambda,\mathbb{T})$. There exists a $C^*$-algebra $\mathcal{T}C^*(\Lambda,c)$ generated by a Toeplitz-Cuntz-Krieger $(\Lambda,c)$-family $\lbrace s_\mathcal{T}(\lambda) : \lambda \in \Lambda \rbrace$ that is universal in the sense that for each Toeplitz-Cuntz-Krieger $(\Lambda,c)$-family $\lbrace t_\lambda : \lambda \in \Lambda \rbrace$ in a $C^*$-algebra $B$, there exists a homomorphism $\pi^{\mathcal{T}}_t: \mathcal{T}C^*(\Lambda,c) \rightarrow B$ such that $\pi^{\mathcal{T}}_t(s_{\mathcal{T}}(\lambda))=t_\lambda$ for all $\lambda \in \Lambda$. The $C^*$-algebra $\mathcal{T}C^*(\Lambda,c)$ is unique in the sense that if $B$ is a $C^*$-algebra generated by a Toeplitz-Cuntz-Krieger $(\Lambda,c)$-family $\lbrace t_\lambda : \lambda \in \Lambda \rbrace$ satisfying the universal property, then there is an isomorphism of $\mathcal{T}C^*(\Lambda,c)$ onto $ B$ such that $s_\mathcal{T}(\lambda)\mapsto t_\lambda$ for all $\lambda \in \Lambda$. We call $\mathcal{T} C^*(\Lambda,c)$ the twisted Toeplitz algebra of $\Lambda$.
\end{theorem}

\begin{proof}
For this proof, we abbreviate Toeplitz-Cuntz-Krieger by TCK and write $t$ is a TCK $(\Lambda,c)$-family for $\left\{ t_\lambda : \lambda \in \Lambda \right \}$ is a TCK $(\Lambda,c)$-family.

First we show that each TCK $(\Lambda,c)$-family induces a homomorphism of $C_c(\Lambda *_s \Lambda)$. Fix a TCK $(\Lambda,c)$-family t. Define a map $\pi_t^0 :C_c(\Lambda *_s \Lambda) \rightarrow C^*(\lbrace t_\lambda :\lambda \in \Lambda \rbrace)$ by
$$ \pi_t^0(a)=\sum_{(\mu,\nu)\in \Lambda *_s \Lambda} a(\mu,\nu) t_\mu t_\nu^*.$$
 We claim $\pi_t^0$ is a homomorphism. One easily checks that $\pi_t^0$ is linear. Fix $a,b\in C_c(\Lambda *_s \Lambda)$. There are finite $F,G\subset \Lambda *_s \Lambda$ such that $a=\sum_{(\mu,\nu )\in F} a_{\mu,\nu} \xi_{\mu,\nu} $ and $b=\sum_{(\eta,\zeta )\in G} b_{\eta,\zeta} \xi_{\eta,\zeta} $. Then
\begin{align*}
\pi_t^0(ab)&=\pi_t^0 \left( \sum_{(\mu,\nu)\in F \atop (\eta,\zeta)\in G}\sum_{(\nu',\eta')\in \Lambda^{\textrm{min}}(\nu,\eta)}c(\eta,\eta')c(\mu,\nu') \overline{c(\nu,\nu')c(\zeta,\eta')}a_{\mu,\nu}b_{\eta,\zeta}\xi_{(\mu \nu',\zeta \eta')} \right)\qquad \text{ by \eqref{eq:multiplication}}\\
&=\sum_{(\mu,\nu)\in F \atop (\eta,\zeta)\in G}\sum_{(\nu',\eta')\in \Lambda^{\textrm{min}}(\nu,\eta)}c(\eta,\eta')c(\mu,\nu') \overline{c(\nu,\nu')c(\zeta,\eta')}a_{\mu,\nu}b_{\eta,\zeta}t_{\mu \nu'}t_{\zeta \eta'}^*\\
&=\sum_{(\mu,\nu)\in F \atop (\eta,\zeta)\in G}a_{\mu,\nu}b_{\eta,\zeta}t_{\mu} t_{\nu}^* t_{\eta} t_{\zeta}^* \qquad \text{by Lemma }~\ref{handylemma}(4)\\
&=\left( \sum_{(\mu,\nu)\in F } a(\mu,\nu) t_\mu t_\nu^* \right) \left( \sum_{(\eta,\zeta)\in G} b(\eta,\zeta) t_\eta t_\zeta^* \right)\\
&=\pi_t^0(a) \pi_t^0(b),
\end{align*}
and
\begin{align*}
\pi_t(a^*)&=\pi_t\left(\sum_{(\mu,\nu )\in F} \overline{a_{\mu,\nu}} \xi_{\nu,\mu}\right)\\
&=\sum_{(\mu,\nu )\in F} \overline{a_{\mu,\nu}} t_{\nu}t_{\mu}^*\\
&=\left( \sum_{(\mu,\nu )\in F} a_{\mu,\nu} t_{\mu}t_{\nu}^* \right)^*\\
&=\pi_t(a)^*.
\end{align*}
So $\pi_t^0$ is a homomorphism.

For each $a\in C_c(\Lambda *_s \Lambda)$, we claim that $$\lbrace \Vert \pi_t^0(a) \Vert : t \text{ is a TCK }(\Lambda,c)\text{-family}\rbrace$$ is bounded. Fix $a\in C_c(\Lambda *_s \Lambda)$ and a TCK $(\Lambda,c)$-family $t$. There is finite $F\subset \Lambda *_s \Lambda$ such that $a=\sum_{(\mu,\nu )\in F} a_{\mu,\nu} \xi_{\mu,\nu}$. For each $\mu,\nu \in \Lambda$, we have
\begin{align*}
 \Vert \pi_t^0(\xi_{\mu,\nu}) \Vert^2 &=\Vert \pi_t^0(\xi_{\mu,\nu} )^* \pi_t^0(\xi_{\mu,\nu})\Vert \\
&=\Vert \pi_t^0(\xi_{\nu,\mu} ) \pi_t^0(\xi_{\mu,\nu})\Vert \\
&= \Vert t_\nu t_\mu^* t_\mu t_\nu^*\Vert \\
&= \Vert c(\nu,s(\mu))c(\mu, s(\mu) \overline{ c(\nu,s(\mu))c(\mu,s(\mu))}  t_{\nu s(\mu)} t_{\nu s(\mu)}^* \Vert \qquad \text{by Lemma }~\ref{handylemma}(4)\\
&= \Vert t_\nu^* t_\nu \Vert\\
&=\Vert t_{s(\nu)} \Vert.
\end{align*}
As $t_{s(\nu)}$ is a projection it has norm $0$ or $1$, and the same is true of $\pi_t^0(\xi_{\mu,\nu}) $. We have
\begin{align*}
\Vert \pi_{t}^0(a) \Vert &\leq \sum_{(\mu,\nu) \in F } |a(\mu,\nu)| \Vert \pi_t^0 ( \xi_{\mu,\nu}) \Vert  \\
&\leq \sum_{(\mu,\nu) \in F } |a(\mu,\nu)|,
\end{align*}
which establishes our claim.

For each $a\in C_c(\Lambda *_s \Lambda)$, define $$N(a):=\sup \lbrace \Vert \pi_t^0(a) \Vert : t \text{ is a TCK } (\Lambda,c)\text{-family}\rbrace.$$ The function $N$ is a *-algebra seminorm on $C_c(\Lambda *_s \Lambda)$ since it is defined as the supremum of a collection of seminorms.
For each TCK $(\Lambda,c)$-family $t$, the kernel of $\pi_{t}^0$ is a two-sided *-ideal in $C_c(\Lambda *_s\Lambda)$. So $$I:=\ker(N) = \bigcap \lbrace \ker \pi_t^0 :  t\text{ is a TCK } (\Lambda,c)\text{-family} \rbrace $$ is also a two-sided *-ideal in $C_c(\Lambda *_s \Lambda)$. Let $A:=C_c(\Lambda *_s \Lambda) / I$. Define $\Vert \cdot \Vert : A \rightarrow \mathbb{R}$ by $\Vert a+I\Vert = N(a)$; this is well defined. To see this suppose that $b\in I$. Then
\begin{align*}
N(a+b)&=\sup\left\{ \Vert \pi_t^0(a+b)\Vert:  t \text{ is a TCK } (\Lambda,c)\text{-family}\right\} \\
&= \sup \left\{\Vert \pi_t^0(a)+ \pi_t^0(b) \Vert: t \text{ is a TCK } (\Lambda,c)\text{-family}\right\}\\
&= \sup\left\{\Vert \pi_t^0(a)\Vert:  t \text{ is a TCK } (\Lambda,c)\text{-family} \right\}.
\end{align*}
The calculations
\begin{align*}
\Vert (a+I)^*(a+I)\Vert &= \Vert a^* a + I\Vert\\
&= N(a^* a)\\
&= \sup\left\{\Vert \pi_t^0(a^*a)\Vert: t \text{ is a TCK } (\Lambda,c)\text{-family} \right \} \\
&=\sup \left\{ \Vert \pi_t^0(a) \Vert^2: t \text{ is a TCK } (\Lambda,c)\text{-family} \right \} \\
&= N(a)^2\\
&=\Vert a+I\Vert^2,
\end{align*}
and
\begin{align*}
\Vert (a+&I)(b+I)\Vert = \Vert ab + I\Vert\\
&= N(a b)\\
&= \sup\left\{ \Vert \pi_t^0(ab) \Vert:t \text{ is a TCK } (\Lambda,c)\text{-family} \right \} \\
&\leq \sup \left\{ \Vert \pi_t^0(a) \Vert : t \text{ is a TCK } (\Lambda,c)\text{-family} \right \}  \sup\left\{ \Vert \pi_t^0(b) \Vert: t \text{ is a TCK } (\Lambda,c)\text{-family} \right \} \\
&= N(a)N(b)\\
&=\Vert a+I\Vert \Vert b+I\Vert
\end{align*}
show that $\Vert \cdot \Vert $ is a Banach *-algebra norm and satisfies the $C^*$-identity. Denote the completion of $A$ by $\mathcal{T}C^*(\Lambda,c)$, which is a $C^*$-algebra by construction.

For $\lambda \in \Lambda$, define $ s_\mathcal{T}(\lambda)  :=\xi_{\lambda,s(\lambda)}+I\in \mathcal{T}C^*(\Lambda,c)$. We claim that $\lbrace s_\mathcal{T}(\lambda)  : \lambda \in \Lambda \rbrace$ is a TCK $(\Lambda,c)$-family that generates $\mathcal{T}C^*(\Lambda,c)$.  If $v\in \Lambda^0$, then $$s_\mathcal{T}(v)^2=\xi_{v,v} \xi_{v,v} + I= \xi_{v,v}+ I=s_\mathcal{T}(v)=\xi_{v,v}^*+ I=s_\mathcal{T}(v)^*$$ If $v,w\in \Lambda^0$ with $v\neq w$ then $s_\mathcal{T}(v) s_\mathcal{T}(w)=\xi_{v,v} \xi_{w,w}+ I=0$. This gives (TCK1). If $\mu,\nu \in \Lambda$ with $s(\mu)=r(\nu)$, then

\begin{align*}
s_\mathcal{T}(\mu) s_\mathcal{T}(\nu) &= \xi_{\mu,s(\mu)} \xi_{\nu,s(\nu)} + I\\
&= \sum_{(s(\mu)',\nu')\in \Lambda^{\text{min}}(s(\mu),\nu)} c(\nu,\nu')c(\mu,s(\mu)')\overline{c(s(\mu),s(\mu)')c(s(\nu),\nu')}\xi_{(\mu s(\mu)', s(\nu) \nu')}+ I\\
&= c(\nu,s(\nu))c(\mu,\nu)\overline{c(s(\mu),\nu)c(s(\nu),s(\nu))}\xi_{(\mu \nu, s(\nu) s(\nu))}+ I\\
&= c(\mu,\nu)\xi_{(\mu \nu,s(\mu\nu))}+ I\\
&=c(\mu,\nu) s_\mathcal{T}(\mu \nu),
\end{align*}
establishing (TCK2). For $\lambda \in \Lambda$, we have
\begin{align*}
s_\mathcal{T}(\lambda)^* s_\mathcal{T}(\lambda) &= \xi_{(\lambda, s(\lambda))}^* \xi_{\lambda, s(\lambda)}+ I\\
&= \xi_{(s(\lambda), \lambda)} \xi_{\lambda, s(\lambda)}+ I\\
&= \xi_{(s(\lambda),s(s(\lambda)))}+ I\\
&=s_\mathcal{T}({s(\lambda)}).
\end{align*}
This gives (TCK3). For (TCK4), fix $\mu,\nu \in \Lambda$. We have

\begin{align*}
s_\mathcal{T}(\mu) s^*_\mathcal{T}(\mu) s_\mathcal{T}(\nu) s^*_\mathcal{T}(\nu)&=\xi_{(\mu,s(\mu))} \xi_{(\mu,s(\mu))}^* \xi_{(\nu,s(\nu))} \xi_{(\nu,s(\nu))}^*\\
&=\xi_{(\mu,s(\mu))} \xi_{(s(\mu),\mu)} \xi_{(\nu,s(\nu))} \xi_{(s(\nu),\nu)}\\
&=\xi_{(\mu,\mu)}\xi_{(\nu,\nu)} \qquad \text{ by } \eqref{eq:multiplicationxi}\\
&=\sum_{(\mu',\nu')\in \Lambda^{\text{min}}(\mu,\nu)} c(\nu,\nu') c(\mu,\mu')\overline{c(\mu,\mu')c(\nu,\nu')}\xi_{(\mu \mu',\nu \nu')}\qquad \text{ by } \eqref{eq:multiplicationxi}\\
&=\sum_{\sigma\in \MCE(\mu,\nu)} \xi_{(\sigma,\sigma)}\\
&=\sum_{\sigma\in \MCE(\mu,\nu)} \xi_{(\sigma,s(\sigma))}\xi_{(\sigma,s(\sigma))}^*\\
&=\sum_{\sigma \in \textrm{MCE}(\mu,\nu)} s_\mathcal{T}({\sigma}) s^*_\mathcal{T}({\sigma}).
\end{align*}

This gives (TCK4). If $\mu,\nu \in \Lambda$ with $s(\mu)=s(\nu)$, then as $$s_\mathcal{T}(\mu) s_\mathcal{T}(\nu)^* = \xi_{(\mu, s(\mu))} \xi_{(\nu, s(\nu))}^* +I = \xi_{(\mu,\nu)} + I,$$ the TCK $(\Lambda,c)$-family $\lbrace s_\mathcal{T}(\lambda) : \lambda \in \Lambda \rbrace$ generates $\mathcal{T}C^*(\Lambda,c)$.
We must now check that $\mathcal{T}C^*(\Lambda,c)$ and $\lbrace s_\mathcal{T}(\lambda) : \lambda \in \Lambda \rbrace$ have the universal property. Fix a TCK $(\Lambda,c)$-family $t$. For $a\in C_c(\Lambda *_s \Lambda)$, we have
\begin{equation}\label{normequ}
\Vert \pi_t^0(a) \Vert \leq \sup\left\{\Vert \pi_{t'}^0(a) \Vert : t \text{ is a TCK } (\Lambda,c)\text{-family} \right \} = N(a).
\end{equation}
Hence if $b\in I$ then $ \pi_t^0(b)=0 $, and so $I \subset \ker(\pi_t^0)$. Hence there is a well-defined linear map $\pi_t:A \rightarrow C^*(\lbrace t_\lambda :\lambda \in \Lambda \rbrace)$ given by $\pi_t(a+I)=\pi_t^0(a)$. This $\pi_t$ is a homomorphism by definition of the operations in the quotient $A$ and because $\pi_t^0$ is homomorphism. Equation~\eqref{normequ} implies that for $a+I \in A$, we have $\Vert \pi_t(a+I)\Vert = \Vert \pi_t^0(a) \Vert \leq N(a) = \Vert a + I \Vert$. So the homomorphism $\pi_t$ is continuous, and hence extends to the desired homomorphism $\pi_t : \mathcal{T}C^*(\Lambda,c) \rightarrow C^*(\lbrace t_\lambda :\lambda \in \Lambda \rbrace)$.

For the final statement, suppose $B$ is a $C^*$-algebra generated by  TCK $(\Lambda,c)$-family $t$ satisfying the universal property. As $\mathcal{T}C^*(\Lambda,c)$ is universal there is a homomorphism $\pi_{t}^\mathcal{T} : \mathcal{T}C^*(\Lambda,c) \to B$ such that $\pi_t\left( s_\mathcal{T}(\lambda \right))=t_\lambda$ for each $\lambda \in \Lambda$. As $B$ is universal there is a homomorphism $\psi : B \rightarrow \mathcal{T}C^*(\Lambda,c)$ such that $\psi(t_\lambda) = s_\mathcal{T}(\lambda)$ for each $\lambda \in \Lambda$. Since $\pi_\mathcal{T} \circ \psi$ and $  \psi \circ \pi_\mathcal{T}$ are the identity maps on the generators of $B$ and $\mathcal{T}C^*(\Lambda,c)$ respectively, $\pi_\mathcal{T}$ and $\psi$ are mutually inverse and hence isomorphisms.
\end{proof}

\section{Relative Cuntz-Krieger $(\Lambda,c;\mathcal{E})$-families}
This section introduces the twisted relative Cuntz-Krieger algebras associated to finitely aligned $k$-graph $\Lambda$. For each $c\in \underline{Z}^2(\Lambda,\mathbb{T})$ and each collection $\mathcal{E}$ of finite exhaustive sets, the twisted relative Cuntz-Krieger algebra $C^*(\Lambda,c;\mathcal{E})$ is the universal $C^*$-algebra generated by a Toeplitz-Cuntz-Krieger $(\Lambda,c)$-family $\left\{ s_\mathcal{E}(\lambda):\lambda \in \Lambda \right \}$ satisfying $$\prod_{\lambda \in E} \left(s_\mathcal{E}(r(E))-s_\mathcal{E}(\lambda)s_\mathcal{E}(\lambda)^* \right)=0$$ whenever $E\in \mathcal{E}$. When $\mathcal{E}=\emptyset$, the associated twisted relative Cuntz-Krieger algebra $C^*(\Lambda,c;\emptyset)$ is the Toeplitz algebra $\mathcal{T}C^*(\Lambda,c)$ studied in the last section. When $1\in \underline{Z}^2(\Lambda,\mathbb{T})$ denotes the identity 2-cocycle, we obtain the relative Cuntz-Krieger algebra $C^*(\Lambda;\mathcal{E})$ studied in \cite{S20061}. When $\mathcal{E}$ is the collection of all finite exhaustive sets $\FE(\Lambda)$ \cite{RSY2004}, we obtain the twisted Cuntz-Krieger algebra $C^*(\Lambda,c):=C^*(\Lambda,c;\FE(\Lambda))$. The twisted relative Cuntz-Krieger algebras $C^*(\Lambda,c;\mathcal{E})$ interpolate between the Toeplitz algebra $\mathcal{T}C^*(\Lambda,c)$, and the Cuntz-Krieger algebra $C^*(\Lambda,c;\FE(\Lambda))$. More precisely, for each $\mathcal{E}\subset \FE(\Lambda)$ there is an ideal $J_\mathcal{E}\subset \mathcal{T}C^*(\Lambda,c)$ such that $$C^*(\Lambda,c;\mathcal{E})= \mathcal{T}C^*(\Lambda,c)/J_\mathcal{E},$$
and, in particular, there is $J_{\FE(\lambda)}$ such that
$$C^*(\Lambda,c)=\mathcal{T}C^*(\Lambda,c)/J_{\FE(\Lambda)}.$$

\label{chptr:Relative_Cuntz-Krieger}
\begin{definition}[\cite{RSY2004}, Definition 2.4]\label{def:finitexhaustive}
Let $(\Lambda,d)$ be a $k$-graph and let $v\in \Lambda^0$. A subset $E\subset v\Lambda$ is said to be \emph{exhaustive} if for every $\mu\in v\Lambda$ there exists a $\lambda\in E$ such that $\Lambda^{\text{min}}(\lambda,\mu) \neq \emptyset$. Define $\FE(\Lambda):=\lbrace E\subset v\Lambda : E \text{ is finite and exhaustive and } v\notin E \rbrace$.
\end{definition}

\begin{definition}\label{def:gapprojections}
Let $(\Lambda,d)$ be a finitely aligned $k$-graph and let $c\in\underline{Z}^2(\Lambda,\mathbb{T})$. Let $\lbrace t_\lambda : \lambda \in \Lambda \rbrace$ be a Toeplitz-Cuntz-Krieger $(\Lambda,c)$-family. Suppose that $v\in \Lambda^0$ and that $E\subset v\Lambda$ is nonempty and finite. We write $r(E)$ for the unique $v\in \Lambda^0$ such that $E\subset v\Lambda$. We call the product
$$Q(t)^E:=\prod_{\lambda \in E} (t_{r(E)} - t_\lambda t_\lambda^*)$$
the gap projection of $t$ associated to $E$.
\end{definition}

\begin{remark} The product $$\prod_{\lambda \in E} (t_{r(E)} - t_\lambda t_\lambda^*)$$ is well-defined as the projections $\left\{ t_\lambda t_\lambda^* : \lambda \in \Lambda\right\}$ pairwise commute. Hence for $v=r(E)$ and $\lambda,\mu \in E$,
\begin{align*}
(t_v-t_\lambda t_\lambda^*)(t_v-t_\mu t_\mu^*)&=t_v^2-t_vt_\mu t_\mu^*-t_\lambda t_\lambda^* + t_\lambda t_\lambda^* t_\mu t_\mu^*\\
&=t_v^2-t_v t_v^*t_\mu t_\mu^*-t_\lambda t_\lambda^*t_v t_v^* + t_\lambda t_\lambda^* t_\mu t_\mu^*\\
&=t_v^2-t_\mu t_\mu^*t_v t_v^*-t_v t_v^*t_\lambda t_\lambda^* +t_\mu t_\mu^* t_\lambda t_\lambda^* \qquad\text{by Lemma~\ref{handylemma}(2)}\\
&=t_v^2 -t_v t_\lambda t_\lambda^*-t_\mu t_\mu^*t_v +t_\mu t_\mu^* t_\lambda t_\lambda^* \\
&=(t_v-t_\mu t_\mu^*)(t_v-t_\lambda t_\lambda^*).
\end{align*}
An induction argument over $|E|$ therefore shows that $\prod_{\lambda \in E} (t_{r(E)} - t_\lambda t_\lambda^*)$ is independent of the order of multiplication.
\end{remark}
The following definition is (\cite{S20061}, Definition 3.2) but with the 2-cocycle incorporated.
\begin{definition}
Let $(\Lambda,d)$ be a finitely aligned $k$-graph, let $c\in \underline{Z}^2(\Lambda,\mathbb{T})$ and let $\mathcal{E} \subset \FE(\Lambda)$. A \emph{relative Cuntz-Krieger} $(\Lambda,c;\mathcal{E})$\emph{-family} is a Toeplitz-Cuntz-Krieger $(\Lambda,c)$-family $\lbrace t_\lambda : \lambda \in \Lambda \rbrace$ that satisfies
\begin{equation} Q(t)^E=0 \textrm{ for all } E \in \mathcal{E}.\tag{CK}\end{equation} \nonumber
When $\mathcal{E}=FE(\Lambda)$, we call the resulting relative Cuntz-Krieger $(\Lambda,c;\mathcal{E})$-family a \emph{Cuntz-Krieger} $(\Lambda,c)$\emph{-family}.
\end{definition}

Note that every Toeplitz-Cuntz-Krieger $(\Lambda,c)$-family is a relative Cuntz-Krieger $(\Lambda,c;\emptyset)$-family and every relative Cuntz-Krieger $(\Lambda,c)$-family is a Toeplitz-Cuntz-Krieger $(\Lambda,c)$-family.

\begin{definition}
Let $(\Lambda,d)$ be a finitely aligned $k$-graph, let $c\in \underline{Z}^2(\Lambda,\mathbb{T})$ and let $\mathcal{E}\subset \FE(\Lambda)$. Define $J_\mathcal{E}\subset \mathcal{T}C^*(\Lambda,c)$ to be the intersection of all closed ideals containing
$$\left \{ Q(s_\mathcal{T})^E : E \in \mathcal{E} \right \}.$$
\end{definition}

\begin{lemma}\label{le:ideal}
Let $A,B$ be $C^*$-algebras and let $X\subset A$. Let $I_X$ denote the intersection of all closed ideals containing $X$. Then $I_X$ is the smallest ideal containing $X$ and $$I_X=\overline{\linspan}\lbrace  bxc : b,c\in A , x \in X \rbrace .$$ We call $I_X$ the \emph{ideal generated by} $X$. Suppose $\pi : A \rightarrow B$ is a homomorphism such that $\pi(x)=0$ for all $x\in X$. Then $\pi(a)=0$ for all $a\in I_X$. There is a unique homomorphism $\tilde \pi : A / I_X \rightarrow B$ satisfying $\tilde \pi ( a+ I_X) = \pi(a)$ for all $a\in A$.
\end{lemma}

\begin{proof}
It is routine to check that an arbitrary intersection of closed ideals containing $X$ is a closed ideal containing $X$. The ideal $I_X$ is the intersection of all closed ideals containing $X$ and hence the smallest closed ideal containing $X$. The set $I:=\overline{\linspan}\lbrace  bxc : b,c\in A , x \in X \rbrace $ is a closed ideal. Let $(e_\lambda)_{\lambda \in \Lambda}$ be an approximate identity for $A$. Then $e_\lambda x e_\lambda \rightarrow  x$ for each $x\in X$. So $I$ contains $X$. Since every closed ideal that contains $X$ must contain $I$, we have $I \subset I_X$. Since $I_X$ is the smallest closed ideal containing $X$, we have $I=I_X$. Suppose that $b,c\in A$ and that $ x\in X$. We have
\begin{align*}
\pi(bxc)=\pi(b)\pi(x)\pi(c)=0.
\end{align*}
So $\pi(a)=0$ for all $a\in \linspan\lbrace  bxc : b,c\in A , x \in X \rbrace $. By continuity, we have $\pi(a)=0$ for all $a\in I_X$.

To see that the map $\tilde \pi$ is well-defined, suppose that $a+I_X=b+I_X$. Then $a=b+n$ for some $n\in I_X$, and so $\pi(a)=\pi(b+n)=\pi(b)$. Since a closed ideal in a $C^*$-algebra is automatically self-adjoint, there is a well-defined involution $*: A/I_X \to A/I_X$ satisfying $(a+I_X)^*=a^*+I_X$ for all $a\in A$. That $\tilde \pi$ is a homomorphism then follows from the definitions of the operations in the quotient $A/I_X$ and that $\pi$ is a homomorphism.
\end{proof}

\begin{theorem}\label{th:relativecuntzkriegeralgebra}
Let $(\Lambda,d)$ be a finitely aligned $k$-graph, let $c\in\underline{Z}^2(\Lambda,\mathbb{T})$ and let $\mathcal{E}\subset \FE(\Lambda)$. There exists a $C^*$-algebra $C^*(\Lambda,c;\mathcal{E})$ generated by a relative Cuntz-Krieger $(\Lambda,c;\mathcal{E})$-family $\lbrace s_{\mathcal{E}}(\lambda) : \lambda \in \Lambda \rbrace$ that is universal in the sense that if $\lbrace t_\lambda : \lambda \in \Lambda \rbrace$ is a relative Cuntz-Krieger $(\Lambda,c;\mathcal{E})$-family in a $C^*$-algebra $B$, then there exists a homomorphism $\pi^{\mathcal{E}}_t: C^*(\Lambda,c ; \mathcal{E}) \rightarrow B$ such that $\pi^{\mathcal{E}}_t(s_{\mathcal{E}}(\lambda))=t_\lambda$ for all $\lambda \in \Lambda$. The $C^*$-algebra $C^*(\Lambda,c;\mathcal{E})$ is unique in the sense that, if $B$ is a $C^*$-algebra generated by a relative Cuntz-Krieger $(\Lambda,c;\mathcal{E})$-family $\lbrace t_\lambda : \lambda \in \Lambda \rbrace$ satisfying the universal property, there is an isomorphism of $C^*(\Lambda,c;\mathcal{E})$ onto $ B$ such that $s_\mathcal{E}(\lambda)\mapsto t_\lambda$ for all $\lambda \in \Lambda$. We call $C^*(\Lambda,c;\mathcal{E}) $ the twisted relative Cuntz-Krieger algebra of $\Lambda$.
\end{theorem}

\begin{proof}
Let $C^*(\Lambda,c; \mathcal{E}):=\mathcal{T}C^*(\Lambda,c)/J_{\mathcal{E}}$ and let $s_\mathcal{E}(\lambda):=s_\mathcal{T}(\lambda) + J_\mathcal{E}$. To prove that $\lbrace s_\mathcal{E}(\lambda) : \lambda \in \Lambda \rbrace$ is a relative Cuntz-Krieger $(\Lambda,c;\mathcal{E})$-family, we only need to show that $Q(s_\mathcal{E})^E=0$ for each $E\in \mathcal{E}$. If $E\in \mathcal{E}$, then $Q(s_\mathcal{T})^E\in J_\mathcal{E}$ by definition, so $Q(s_\mathcal{E})^E=Q(s_\mathcal{T})^E+J_\mathcal{E}=0+J_\mathcal{E}$.

We will now construct the homomorphism $\pi_t^\mathcal{E}$. Fix a relative Cuntz-Krieger $(\Lambda,c ; \mathcal{E})$-family $\lbrace t_\lambda : \lambda \in \Lambda \rbrace$ in a $C^*$ algebra $B$. Let $\pi_t^\mathcal{T}: \mathcal{T}C^*(\Lambda,c) \rightarrow B$ be the homomorphism from Theorem~\ref{toeplitzalgebra}. If $E\in \mathcal{E}$, then
\begin{equation} \pi_t^{\mathcal{T}} \left( Q(s_\mathcal{T})^E \right) =Q(t)^E=0. \end{equation} \nonumber
Lemma~\ref{le:ideal} implies that $\pi_t^\mathcal{T}(a)=0$ for all $a\in J_\mathcal{E}$, and that there is a homomorphism $\pi_t^\mathcal{E} : \mathcal{T}C^*(\Lambda,c)/J_{\mathcal{E}} \rightarrow B$ satisfying $\pi_t^\mathcal{E}(a+J_\mathcal{E})=\pi_t^\mathcal{T}(a)$ for all $a\in \mathcal{T}C^*(\Lambda,c)$. Fix $\lambda \in \Lambda$. Then $\pi_t^\mathcal{E}(s_\mathcal{E}(\lambda))=\pi_t^\mathcal{E}(s_\mathcal{T}(\lambda)+J_\mathcal{E})=\pi_t^\mathcal{T}\left(s_\mathcal{T}(\lambda)\right)=t_\lambda$. The proof that $C^*(\Lambda,c;\mathcal{E})$ is unique is identical to the proof that $\mathcal{T}C^*(\Lambda,c;\mathcal{E})$ is unique in Theorem~\ref{toeplitzalgebra}.
\end{proof}

\section{Gap projections and boundary paths}
For a collection $\mathcal{E}\subset  \FE(\Lambda)$, we recall the definition of the satiation $\overline{\mathcal{E}}$ of $\mathcal{E}$ from \cite{S20061} in Definition~\ref{de:sat}. We show that for each collection $\mathcal{E}\subset \FE(\Lambda)$, the gap projection $Q(s_\mathcal{E})^E$ in $C^*(\Lambda,c;\mathcal{E})$ associated to a finite exhaustive set $E$ is zero if and only if $E\in \overline{\mathcal{E}}$.

\label{chptr:Gap_projections}
For $v\in \Lambda^0$, we wish to know precisely which finite subsets $E$ of $v\Lambda$ satisfy
\begin{equation}Q(s_\mathcal{E})^E\neq 0.\label{eq:gapprojnonzero1}\end{equation}
 The following remark says that if $v\in E$ then $Q(s_\mathcal{E})^E=0$. If $E\subset v\Lambda \setminus \left\{ v \right \}$ is not exhaustive, Lemma~\ref{le:nonzerogapfornonextsets} implies that $Q(s_\mathcal{E})^E\neq 0$. So in order to find out precisely which finite subsets $E$ of $v\Lambda$ satisfy~\eqref{eq:gapprojnonzero1}, we only need to consider the case when $E\subset v \Lambda \setminus \left\{ v \right \}$ is exhaustive; that is, $E\in \FE(\Lambda)$.

\begin{remark} Suppose that $F\subset v \Lambda $ is finite and that $v\in F$. Then since $t_v t_v^*=t_v$, we have
$$Q(t)^F=\prod_{\mu\in F} \left( t_v -t_\mu t_\mu^* \right) =(t_v-t_v t_v^*)\prod_{\mu\in F\setminus \lbrace v \rbrace} \left( t_v -t_\mu t_\mu^* \right) =0$$
for all Toeplitz-Cuntz-Krieger $(\Lambda, c)$-families $\lbrace t_\lambda : \lambda \in \Lambda \rbrace$.
\end{remark}

\begin{lemma}\label{le:nonzerogapfornonextsets}
Let $(\Lambda,d)$ be a finitely aligned $k$-graph and let $c\in \underline{Z}^2(\Lambda,\mathbb{T})$. Let $v\in \Lambda^0$ and suppose that $\lbrace t_\lambda: \lambda \in \Lambda \rbrace$ is a Toeplitz-Cuntz-Krieger $(\Lambda,c)$-family. Suppose that $F\subset v\Lambda\setminus \lbrace v \rbrace$ is finite and is not exhaustive. Then
$$Q(t)^F=\prod_{\mu \in F} \left( t_v - t_\mu t_\mu^* \right) \neq 0 $$
whenever $t_w$ is nonzero for all $w\in \Lambda^0$.
\end{lemma}

\begin{proof}
Since $F$ is not exhaustive there exists $\lambda \in r(F) \Lambda$ such that for all $\mu \in F$ we have $\MCE(\mu,\lambda)=\emptyset$. By (TCK4) we have $t_\mu t_\mu^* t_\lambda t_\lambda^* =0$ for all $\mu\in F$. As $t_\lambda t_\lambda^*$ is a projection $t_\lambda t_\lambda^*=\left(t_\lambda t_\lambda^*\right)^{|F|}$. Since $t_\lambda t_\lambda^*$ commutes with $(t_v-t_\mu t_\mu^*)$ for each $\mu \in F$, we have
\begin{align*}
\left[ \prod_{\mu\in F} \left( t_v - t_\mu t_\mu^*\right) \right] t_\lambda t_\lambda^* &= \prod_{\mu \in F} \left(t_v t_\lambda t_\lambda^* - t_\mu t_\mu^* t_\lambda t_\lambda^* \right) \\
&= \left( t_v t_\lambda t_\lambda^*\right)^{|F|}\\
&=\left( c(v,\lambda) t_\lambda t_\lambda^*\right)^{|F|}\\
&=t_\lambda t_\lambda^*.
\end{align*}
Since $t_{s(\lambda)}\neq 0$, Remark~\ref{re:handyre} implies that $t_\lambda t_\lambda^*\neq0$, and so $\prod_{\mu\in F} \left( t_v - t_\mu t_\mu^*\right) \neq 0$.
\end{proof}

The next lemma shows that in the twisted Toeplitz algebra the gap projections associated to $F\subset v \Lambda \setminus \lbrace v \rbrace$ are all nonzero.
\begin{lemma}\label{le:allgapprojectionsarenonzero}
Let $(\Lambda,d)$ be a finitely aligned $k$-graph and let $c\in \underline{Z}^2(\Lambda,\mathbb{T})$. Let $\lbrace T_\lambda : \lambda \in \Lambda \rbrace$ be the Toeplitz-Cuntz-Krieger $(\Lambda,c)$-family of Proposition~\ref{toeplitzrepresentation}. For all $\lambda \in \Lambda$ and all finite $F\subset s(\lambda) \Lambda \setminus \lbrace s(\lambda) \rbrace$, we have
$$\prod_{\mu \in F} \left( T_\lambda T_\lambda^* - T_{\lambda \mu}T_{\lambda \mu}^* \right) \neq 0. $$
The same is true for the family $\left \{ s_\mathcal{T}(\lambda) : \lambda \in \Lambda \right \}$.
\end{lemma}

\begin{proof}
Fix $\lambda \in \Lambda$ and finite $F\subset s(\lambda) \Lambda \setminus \lbrace s(\lambda) \rbrace$. We prove by induction on $|F|$ that $$\prod_{\mu \in F} \left( T_\lambda T_\lambda^* - T_{\lambda \mu} T_{\lambda \mu}^* \right) \xi_\lambda =\xi_\lambda.$$ If $|F|=1$, say $F=\left\{\mu \right\}$, then
\begin{align*}
\left( T_\lambda T_\lambda^* - T_{\lambda \mu} T_{\lambda \mu}^* \right) \xi_\lambda = \overline{c(\lambda, s(\lambda))} T_\lambda \xi_{s(\lambda)}=c(\lambda,s(\lambda))\xi_\lambda =\xi_\lambda.
\end{align*}
Suppose $\prod_{\mu \in F} \left( T_\lambda T_\lambda^* - T_{\lambda \mu} T_{\lambda \mu}^* \right) \xi_\lambda =\xi_\lambda$ whenever $|F|=k$. Then for $|F|=k+1$ and any $\mu_0\in F$, we have
\begin{align*}
\left[\prod_{\mu \in F} \left( T_\lambda T_\lambda^* - T_{\lambda \mu} T_{\lambda \mu}^* \right)\right] \xi_\lambda &= \left[\prod_{\mu \in F\setminus \left \{ \mu_0 \right\}} \left( T_\lambda T_\lambda^* - T_{\lambda \mu} T_{\lambda \mu}^* \right) \right] \left( T_\lambda T_\lambda^* - T_{\lambda \mu_0} T_{\lambda \mu_0}^* \right)\xi_\lambda \\
&=\left[\prod_{\mu \in F\setminus \left \{ \mu_0 \right\}} \left( T_\lambda T_\lambda^* - T_{\lambda \mu} T_{\lambda \mu}^* \right)\right]\xi_\lambda\\
&=\xi_\lambda.
\end{align*}
Hence by induction $\prod_{\mu \in F} \left( T_\lambda T_\lambda^* - T_{\lambda \mu} T_{\lambda \mu}^* \right)\xi_\lambda=\xi_\lambda$ for all finite  $F\subset s(\lambda) \Lambda \setminus \lbrace s(\lambda) \rbrace$. In particular $\prod_{\mu \in F} \left( T_\lambda T_\lambda^* - T_{\lambda \mu} T_{\lambda \mu}^* \right)\neq 0$. By Theorem~\ref{toeplitzalgebra} there is a unique homomorphism $\pi^\mathcal{T}_t : \mathcal{T}C^*(\Lambda,c) \rightarrow \mathcal{B}\left( \ell^2(\Lambda)\right)$ satisfying $\pi^\mathcal{T}_t(s_\mathcal{T}(\mu))=T_\mu$ for all $\mu \in \Lambda$. Hence $\prod_{\mu \in F}\left( s_\mathcal{T}(\lambda)s_\mathcal{T}(\lambda)^*-s_\mathcal{T}(\lambda \mu)s_\mathcal{T}(\lambda \mu)^* \right)\neq 0.$
\end{proof}

The next few Lemmas involve technical calculations concerning the algebraic manipulation of gap projections. Recall from Definition~\ref{def:MCEEXT}, that for $E\subset v \Lambda$ and $\mu \in v \Lambda$, we may form the collection $\Ext(\mu;E)\subset s(\mu)\Lambda$.
\begin{lemma}\label{le:EtoExtE}
Let $(\Lambda,d)$ be a finitely aligned $k$-graph, let $c\in \underline{Z}^2(\Lambda,\mathbb{T})$ and let $\lbrace t_\lambda : \lambda \in \Lambda \rbrace$ be a Toeplitz-Cuntz-Krieger $(\Lambda,c)$-family. Let $\mathcal{E}\subset \FE (\Lambda)$. Then for $\mu \in r(E)\Lambda$, we have $$Q(t)^E t_\mu= t_\mu Q(t)^{\Ext(\mu;E)}. $$
\end{lemma}

\begin{proof}
Suppose that $\mu\in r(E)\Lambda$. Since $t_\mu$ is a partial isometry,
\begin{align*}Q(t)^E t_\mu&=\left(\prod_{\lambda \in E} \left( t_{r(E)} - t_\lambda t_\lambda^* \right)\right) t_\mu t_\mu^* t_\mu\\
&=\left( \prod_{\lambda \in E} \left( t_{r(E)}t_\mu t_\mu^* - t_\lambda t_\lambda^*t_\mu t_\mu^* \right) \right)t_\mu.
\end{align*}
Applying (TCK2) and (TCK4) gives
\begin{align*}
Q(t)^E t_\mu&=\left(\prod_{\lambda \in E} \left( c(r(\mu),\mu)t_{r(\mu)\mu} t_\mu^* - \sum_{(\alpha,\beta)\in \Lambda^{\text{min}}(\lambda,\mu)} t_{\mu \beta} t_{\mu \beta}^* \right)\right) t_\mu\\
&=\left(\prod_{\lambda \in E} \left( t_{\mu s(\mu)} t_\mu^* - \sum_{(\alpha,\beta)\in \Lambda^{\text{min}}(\lambda,\mu)}c(\mu,\beta)\overline{c(\mu,\beta)} t_\mu t_\beta t_\beta^*t_\mu^* \right) \right) t_\mu\\
&=\left(\prod_{\lambda \in E} \left(\overline{c(\mu,s(\mu)} t_{\mu}t_{s(\mu)} t_\mu^* - \sum_{(\alpha,\beta)\in \Lambda^{\text{min}}(\lambda,\mu)}t_\mu t_\beta t_\beta^*t_\mu^* \right) \right)t_\mu.
\end{align*}
Factorising gives
\begin{align*}
Q(t)^E t_\mu&=\left(\prod_{\lambda \in E} \left(  t_{\mu}\left(t_{s(\mu)}  - \sum_{(\alpha,\beta)\in \Lambda^{\text{min}}(\lambda,\mu)} t_\beta t_\beta^*\right) t_\mu^* \right) \right) t_\mu.
\end{align*}
For each $\beta\in s(\mu)\Lambda$ such that there is $\alpha\in s(\lambda)\Lambda$ with $(\alpha,\beta)\in \Lambda^{\text{min}}(\lambda,\mu)$ we have $d(\beta)=d(\lambda) \vee d(\mu)-d(\lambda)$. Lemma~\ref{handylemma}(6) implies that $\lbrace t_\beta t_\beta^*: (\alpha,\beta)\in \Lambda^{\text{min}}(\lambda,\mu) \rbrace$ is a set of mutually orthogonal projections. Hence
\begin{align*}
Q(t)^E t_\mu&= \left( \prod_{\lambda \in E} \left(  t_{\mu}\left(\prod_{\beta\in s(\mu)\Lambda \text{ s.t }  \atop \exists \alpha\in s(\lambda)\Lambda \text{ s.t } (\alpha,\beta)\in \Lambda^{\text{min}}(\lambda,\mu)}\left(t_{s(\mu)}  - t_\beta t_\beta^* \right) \right) t_\mu^* \right) \right) t_\mu \label{eq:idealstructure1}.
\end{align*}
Using the definition of $\Ext(\mu;E)$ and (TCK3) we have
\begin{flalign*}
&&Q(t)^E t_\mu&= t_{\mu}\left(\prod_{\beta \in \Ext(\mu;E)}\left( \left(t_{s(\mu)}  - t_\beta t_\beta^* \right)  t_{s(\mu)} \right) \right) \nonumber&\\
&&&= t_{\mu}\left(\prod_{\beta \in \Ext(\mu;E)} \left(t_{s(\mu)} t_{s(\mu)}   - t_\beta t_\beta^*  t_{s(\mu)} \right)  \right) \nonumber&\\
&&&= t_{\mu}\left(\prod_{\beta \in \Ext(\mu;E)} \left(t_{s(\mu)}^2   - \overline{c(s(\mu),\beta)}t_\beta t_{s(\mu)\beta}^*  \right)  \right) \nonumber&\\
&&&= t_{\mu}\left(\prod_{\beta \in \Ext(\mu;E)} \left(t_{r(\Ext(\mu;E))}   - t_\beta t_\beta^*  \right)  \right) .&\qedhere\nonumber
\end{flalign*}
\end{proof}

\begin{lemma}\label{le:neededfors4}
Let $(\Lambda,d)$ be a finitely aligned $k$-graph and let $c\in \underline{Z}^2(\Lambda,\mathbb{T})$. Let $\left \{ t_\lambda : \lambda \in \Lambda \right \}$ be a Toeplitz-Cuntz-Krieger $(\Lambda,c)$-family. Suppose that $v\in \Lambda^0$, that $\lambda \in v \Lambda$ and that $E\subset s(\lambda)\Lambda $ is finite. Then
$$t_v -t_\lambda t_\lambda^*=t_v \prod_{\nu \in E}\left( t_v-t_{\lambda \nu}t_{\lambda \nu}^* \right)-  t_\lambda Q(t)^E t_\lambda^*.$$
\end{lemma}

\begin{proof}
As $t_{\lambda \mu}t_{\lambda \mu}^* \leq t_\lambda t_\lambda^*$ for all $\mu \in s(\lambda)\Lambda$, we have
$$(t_v-t_\lambda t_\lambda^*)(t_v-t_{\lambda \nu }t_{\lambda \nu}^*)=t_v - t_\lambda t_{\lambda}^*$$
for all $\nu \in E$. The range projections pairwise commute by Lemma~\ref{handylemma}(2). Since $E$ is finite, it follows that
$$(t_v -t_\lambda t_\lambda^*) \prod_{\nu \in E}\left(t_v-t_{\lambda \nu}t_{\lambda \nu}^* \right)=t_v-t_\lambda t_\lambda^*.$$
We have
\begin{align*}
t_v-t_\lambda t_\lambda^*&=\left( t_v - t_\lambda t_\lambda^* \right)\prod_{\nu \in E}\left(t_v-t_{\lambda \nu}t_{\lambda \nu}^* \right)\\
&=t_v \prod_{\nu \in E}\left(t_v-t_{\lambda \nu}t_{\lambda \nu}^* \right)-t_\lambda t_\lambda^* \prod_{\nu \in E}\left(t_v-t_{\lambda \nu}t_{\lambda \nu}^* \right)\\
&=t_v \prod_{\nu \in E}\left(t_v-t_{\lambda \nu}t_{\lambda \nu}^* \right)-t_\lambda t_\lambda^* \prod_{\nu \in E}\left(t_v-c(\lambda,\nu)\overline{c(\lambda,\nu)}t_{\lambda} t_{\nu}t_{\nu}^*t_\lambda^* \right) \qquad \text{by (TCK2)}.
\end{align*}
Since $t_\lambda t_\lambda^*$ is a projection, we have
\begin{align*}
t_v-t_\lambda t_\lambda^*&=t_v \prod_{\nu \in E}\left(t_v-t_{\lambda \nu}t_{\lambda \nu}^* \right)- \prod_{\nu \in E}\left(t_\lambda t_\lambda^*t_v- t_\lambda t_\lambda^*t_{\lambda} t_{\nu}t_{\nu}^*t_\lambda^* \right)\\
&=t_v \prod_{\nu \in E}\left( t_v-t_{\lambda \nu}t_{\lambda \nu}^* \right)- \prod_{\nu \in E}\left( t_\lambda \left(t_{s(\lambda)}-  t_{\nu}t_{\nu}^*\right) t_\lambda^* \right)\\
&=t_v \prod_{\nu \in E}\left( t_v-t_{\lambda \nu}t_{\lambda \nu}^* \right)-  t_\lambda \left[ \prod_{\nu \in E}\left(t_{s(\lambda)}-  t_{\nu}t_{\nu}^*\right) \right]t_\lambda^*,
\end{align*}
since $t_\lambda^*t_\lambda=t_{s(\lambda)}$
\end{proof}

\begin{lemma}\label{le:gapprojectioninsatiatedsets}
Let $(\Lambda,d)$ be a finitely aligned $k$-graph, let $c\in \underline{Z}^2(\Lambda,\mathbb{T})$ and let $\mathcal{E}\subset\FE(\Lambda)$. Let $\left \{ t_\lambda : \lambda \in \Lambda \right\}$ be a Toeplitz-Cuntz-Krieger $(\Lambda,c)$-family. Then:
\begin{enumerate}
\item[(1)] if $G\in \mathcal{E}$ and $E\subset r(G)\Lambda \setminus \Lambda^0$ is finite with $G\subset E$, then $Q(t)^E=Q(t)^G Q(t)^{E\setminus G}$;
\item[(2)] if $G\in \mathcal{E}$ with $r(G)=v$ and if $\mu \in v \Lambda \setminus G \Lambda$, then $Q(t)^{\Ext(\mu;G)}=t_\mu^* Q(t)^G t_\mu$;
\item[(3)] if $G\in \mathcal{E}$, $0<n_\lambda \leq d(\lambda)$ for each $\lambda \in G$ and $E:=\left \{ \lambda(0,n_\lambda) :\lambda \in G \right \}$, then $Q(t)^E \leq Q(t)^G$.
\end{enumerate}
\end{lemma}

\begin{proof}
For (1), suppose that $G\in \mathcal{E}$ and that $E\subset r(G)\Lambda \setminus \Lambda^0$ is finite with $G\subset E$. The range projections $\left\{t_\lambda t_\lambda^*:\lambda \in \Lambda \right\}$ pairwise commute by Lemma~\ref{handylemma}(2). It follows that $Q(t)^E=Q(t)^G Q(t)^{E\setminus G}$.

For (2), suppose that $G\in \mathcal{E}$ satisfies $r(G)=v$ and that $\mu \in v \Lambda \setminus G \Lambda$. Lemma~\ref{le:EtoExtE} implies that  $t_\mu Q(t)^{\Ext(\mu;G)}= Q(t)^G t_\mu$. Multiplying both sides on the left by $t_\mu^*$ gives $$Q(t)^{\Ext(\mu;G)}=t_{s(\mu)}Q(t)^{\Ext(\mu;G)}=t_\mu^* t_\mu Q(t)^{\Ext(\mu;G)}=t_\mu^* Q(t)^G t_\mu.$$\\
For (3), suppose that $G\in \mathcal{E}$ and that $0<n_\lambda \leq d(\lambda)$ for each $\lambda \in G$. Set $E:=\left \{ \lambda(0,n_\lambda) :\lambda \in G \right \}$. Since
$$t_{r(E)}-t_{\lambda(0,n_\lambda)}t_{\lambda(0,n_\lambda)}^*\leq t_{r(G)}-t_\lambda t_\lambda^*$$
for all $\lambda \in G$. It follows that
$Q(t)^E\leq Q(t)^G.$
\end{proof}

Given a collection $\mathcal{E} \subset \FE(\Lambda)$, we recall the definition of $\overline{\mathcal{E}}$, the satiation of $\mathcal{E}$, from \cite{S20061}.

\begin{definition}[\cite{S20061}, Definition 4.1]\label{de:sat}
Let $(\Lambda,d)$ be a finitely aligned $k$-graph. We say that a subset $\mathcal{E}$ of $\FE(\Lambda)$ is \emph{satiated} if it satisfies:
\begin{enumerate}
\item[(S1)] if $G\in \mathcal{E}$ and $E\subset r(G)\Lambda \setminus \Lambda^0$ is finite with $G\subset E$ then $E\in \mathcal{E}$;
\item[(S2)]if $G\in \mathcal{E}$ with $r(G)=v$ and if $\mu \in v\Lambda \setminus G\Lambda$ then $\Ext(\mu;G)\in \mathcal{E}$;
\item[(S3)]if $G\in \mathcal{E}$ and $0< n_\lambda \leq d(\lambda)$ for each $\lambda \in G$, then $\lbrace \lambda(0,n_\lambda) : \lambda \in G \rbrace \in \mathcal{E}$; and
\item[(S4)]if $G\in \mathcal{E},G'\subset G$ and $G'_\lambda \in s(\lambda) \mathcal{E}$ for each $\lambda \in G'$, then $\left( (G\setminus G') \cup \left( \bigcup_{\lambda \in G'} \lambda G_\lambda'\right) \right) \in \mathcal{E}$.
\end{enumerate}
The smallest collection of subsets of $\Lambda$ which contains $\mathcal{E}$ and satisfies (S1)-(S4) is denoted $\overline{\mathcal{E}}$ and called the \emph{satiation} of $\mathcal{E}$.
\end{definition}

Let $\Lambda$ be a finitely aligned $k$-graph and $\mathcal{E}\subset \FE(\Lambda)$. Definition~5.2 of \cite{S20061} defines maps $\Sigma_{1}, \Sigma_{2}, \Sigma_{3} , \Sigma_{4}: \FE(\Lambda) \to \FE(\Lambda)$ given by
\begin{align*}\Sigma_{1}(\mathcal{E})&:= \left \{ F\subset \Lambda \setminus \Lambda^0: F \text{ is finite, and there exists } E\in \mathcal{E} \text{ with } E \subset F \right \},\\
\Sigma_{2}(\mathcal{E})&:=\left\{ \Ext(\mu;E): E\in \mathcal{E}, \mu \in r(E)\Lambda \setminus E \Lambda \right \},\\
\Sigma_{3}(\mathcal{E})&:=\left\{ \left\{ \lambda(0,n_\lambda) : \lambda \in E \right \} : E \in \mathcal{E}, 0<n_\lambda \leq d(\lambda) \text{ for all } \lambda \in E \right \}, \text{ and }\\
\Sigma_4(\mathcal{E})&:=\left\{ (E \setminus F) \cup\left( \displaystyle\bigcup_{\lambda \in F} \lambda F_\lambda \right): E\in \mathcal{E}, F\subset E, F_\lambda \in s(\lambda)\mathcal{E} \text{ for all } \lambda \in F \right \}.
\end{align*}

Note that Lemma~5.3 of \cite{S20061} implies that $\mathcal{E} \subset \Sigma_i(\mathcal{E})$ for each $i=1,2,3,4.$

The following summarises some properties of $\overline{\mathcal{E}}$.

\begin{lemma}\label{le:satiation1}
Let $(\Lambda,d)$ be a finitely aligned $k$-graph, let $c\in \underline{Z}^2(\Lambda,\mathbb{T})$ and let $\mathcal{E}\subset \text{FE}(\Lambda)$. Let $\lbrace t_\lambda : \lambda \in \Lambda \rbrace$ be a relative Cuntz-Krieger $(\Lambda,c;\mathcal{E})$-family. Then:
\begin{enumerate}
\item[(1)] the satiation $\overline{\mathcal{E}}$ of $\mathcal{E}$ is a subset of $\FE(\Lambda)$;
\item[(2)] if $E\in \overline{\mathcal{E}}$ then $Q(t)^E=0$; and
\item[(3)] the ideals $J_\mathcal{E}$ and $J_\mathcal{\overline{E}}$ of $\mathcal{T}C^*(\Lambda,c)$ are equal.
\end{enumerate}
Moreover, $C^*(\Lambda,c;\mathcal{E})=C^*(\Lambda,c;\overline{\mathcal{E}})$.
\end{lemma}

\begin{proof}
Part (1): Let $\Sigma=\Sigma_4\circ \Sigma_3\circ \Sigma_2\circ \Sigma_1$. Fix $\mathcal{F}\subset \FE(\Lambda)$. Lemma 5.3 of \cite{S20061} implies that $\Sigma(\mathcal{F})\subset \FE(\Lambda)$. Write $\Sigma^n(\mathcal{F})$ for $$\overbrace{\Sigma \circ \Sigma \circ \dots \Sigma}^{\text{n terms}} (\mathcal{F}).$$
We have $\bigcup_{n=1}^\infty \Sigma^n(\mathcal{F})\subset \FE(\Lambda)$. Proposition 5.5 of \cite{S20061} implies that $\overline{\mathcal{F}}=\bigcup_{n=1}^\infty \Sigma^n(\mathcal{F})\subset \FE(\Lambda)$. This gives part (1).

Part (2): Since $\overline{\mathcal{E}}=\bigcup_{n=1}^\infty \Sigma^n(\mathcal{E})$ and since $\mathcal{F}\subset \Sigma_i(\mathcal{F})$ for all $\mathcal{F}\subset \FE(\Lambda)$ and $i=1,2,3,4$, by induction it suffices to show that if $Q(t)^G=0$ for all $G\in \mathcal{F}$, where $\mathcal{F}\subset \FE(\Lambda),$ then $Q(t)^E=0$ for all $E\in \Sigma_i(\mathcal{F})$ for each $i=1,2,3,4.$

Fix $\mathcal{F}\subset \FE(\Lambda)$ such that $Q(t)^G=0$ for all $F\in \mathcal{F}$. We consider the cases $i=1,2,3,4$ separately. Suppose that $i=1$. Fix $E\in \Sigma_i(\mathcal{F})$. Then $G\subset E$ for some $G\in \mathcal{F}$. Lemma~\ref{le:gapprojectioninsatiatedsets} implies that $Q(t)^E=Q(t)^G Q(t)^{E\setminus G}=0$. Suppose that $i=2$. Fix $E\in \Sigma_i(\mathcal{F})$. Then $E=\Ext(\mu;G)$ for some $G\in \mathcal{F}$ and $\mu \in r(G)\Lambda \setminus G\Lambda$. Lemma~\ref{le:gapprojectioninsatiatedsets} implies that $Q(t)^E=t_\mu^* Q(t)^G t_\mu=0$. Suppose that $i=3$. Fix $E\in \Sigma_i(\mathcal{F})$. Then $E=\left \{ \lambda(0,n_\lambda): \lambda \in G\right\}$ for some $G\in \mathcal{F}$ where $0<n_\lambda \leq d(\lambda)$ for each $\lambda \in \Lambda$. Lemma~\ref{le:gapprojectioninsatiatedsets} implies that $Q(t)^E=Q(t)^E Q(t)^G=0$. Suppose that $i=4$. Fix $E\in \Sigma_i(\mathcal{F})$. Then $$E=\left( (G\setminus G') \cup \left( \bigcup_{\lambda \in G'} \lambda G_\lambda'\right) \right)$$ for some $G \in \mathcal{F}$, $G'\subset G$ and $G_\lambda'\in s(\lambda)\mathcal{F}$ for each $\lambda \in G'$. Lemma~\ref{le:neededfors4} implies that for each $\lambda \in G'$, we have
$$t_{r(G)}-t_\lambda t_\lambda^*=t_{r(G)}-t_\lambda t_\lambda^*+t_\lambda Q(t)^{G_\lambda'} t_\lambda^*=\prod_{\mu \in G_\lambda '} \left( t_{r(G)}-t_{\lambda \mu}t_{\lambda \mu}^* \right).$$ Hence
\begin{align*}
Q(t)^E&=Q(t)^{G \setminus G'}\prod_{\lambda \in G'} \left[ \prod_{\mu \in G_\lambda'}\left( t_{r(G)}-t_{\lambda \mu}t_{\lambda \mu}^* \right) \right]\\
&=Q(t)^{G \setminus G'}\prod_{\lambda \in G'} \left( t_{r(G)}-t_\lambda t_\lambda^* \right)\\
&=Q(t)^G\\
&=0.
\end{align*}
For part (3), since $J_\mathcal{E}\subset J_{\overline{\mathcal{E}}}$ it suffices to show the reverse containment. Suppose $E\in \overline{\mathcal{E}}$. By part (2), we have $Q(s_\mathcal{E})^E=0$. Hence $Q(s_\mathcal{T})^E\in J_{\mathcal{E}}$. So $J_{\mathcal{E}}$ contains all the generators of $J_{\overline{\mathcal{E}}}$, hence $J_{\mathcal{E}}=J_{\overline{\mathcal{E}}}$.
\end{proof}

The main goal for the rest of this chapter is to show that the converse of Lemma~\ref{le:satiation1}(2) holds in the $C^*$-algebra $C^*(\Lambda,c;\mathcal{E})$; that is, for $E\in \FE(\Lambda)$, $Q(s_\mathcal{E})^E=0$ implies that $E\in \overline{\mathcal{E}}$. This is achieved with Theorem~\ref{th:nonzerogapprojection}. First we need several technical lemmas on boundary paths and the structure of the ideals $J_\mathcal{E}$.

\begin{lemma}\label{le:idealstructure1}
Let $(\Lambda,d)$ be a finitely aligned $k$-graph, let $c\in \underline{Z}^2(\Lambda,\mathbb{T})$ and let $\lbrace t_\lambda : \lambda \in \Lambda \rbrace$ be a Toeplitz-Cuntz-Krieger $(\Lambda,c)$-family. Let $\mathcal{E}\subset \FE (\Lambda)$. Suppose $E\in \overline{\mathcal{E}}$. Then:
\begin{enumerate}\renewcommand{\theenumi}{\alph{enumi}}
\item[(1)] if $\mu \in E \Lambda \cup \left( \Lambda \setminus r(E)\Lambda \right)$ then
\begin{equation} Q(t)^E t_\mu = 0; \text{ and}
\end{equation}
\item[(2)] if $\mu \in  r(E)\Lambda \setminus E\Lambda $ then
\begin{equation} Q(t)^E t_\mu= t_\mu Q(t)^{\Ext(\mu;E)}. \end{equation}

\end{enumerate}

\end{lemma}

\begin{proof}
Suppose that $\mu \in E\Lambda$. There exists $\lambda\in E$ and $ \alpha  \in \Lambda$ such that $\mu = \lambda \alpha$. Then $(\alpha,s(\mu))\in \Lambda^{\text{min}}(\lambda,\mu)$ and hence $s(\mu)\in \Ext(\mu;E)$. As $E \Lambda \subset r(E) \Lambda$, Lemma~\ref{le:EtoExtE} implies that $$\left(\prod_{\lambda \in E} \left( t_{r(E)} - t_\lambda t_\lambda^* \right)\right) t_\mu= t_{\mu}\left(\prod_{\beta \in \Ext(\mu;E)} \left(t_{s(\mu)}   - t_\beta t_\beta^*  \right)  \right) =0,$$ since $t_{s(\mu)}-t_{s(\mu)} t_{s(\mu)}^*=0$. Suppose $\mu \in \Lambda \setminus r(E) \Lambda$. Then as $r(\mu)\neq r(E)$, using (TCK1) we have,
\begin{align}
 \left(\prod_{\lambda \in E} \left(t_{r(E)} -t_\lambda t_\lambda^* \right)\right)t_\mu&=\left(\prod_{\lambda \in E} \left(t_{r(E)} -c(r(E),\lambda)t_\lambda t_\lambda^*t_{r(E)} \right)\right)c(r(\mu),\mu)t_{r(\mu)}t_\mu \ \label{eq:gap1}\\
&=\left(\prod_{\lambda \in E} \left(t_{r(E)}t_{r(\mu)} -t_\lambda t_\lambda^*t_{r(E)}t_{r(\mu)} \right)\right) t_\mu \nonumber\\
&=0. \nonumber
\end{align}
Hence if $\mu \in E \Lambda \cup \left( \Lambda \setminus r(E)\Lambda \right)$ then $Q(t)^E t_\mu=0$, establishing part (1) of the Lemma.
Part (2) is Lemma~\ref{le:EtoExtE}.
\end{proof}

We can now write down spanning elements for the ideals $J_\mathcal{E}$.

\begin{lemma}\label{le:idealstructure2}
Let $(\Lambda,d)$ be a finitely aligned $k$-graph and let $c\in \underline{Z}^2(\Lambda,\mathbb{T})$. Let $\mathcal{E}\subset \FE (\Lambda)$. Then
$$J_\mathcal{E}=\overline{\linspan}\lbrace s_\mathcal{T}(\mu)  Q(s_\mathcal{T})^E s_\mathcal{T}(\nu)^* : E\in \overline{\mathcal{E}}, \mu,\nu \in \Lambda \text{ and } s(\mu)=s(\nu) \rbrace.$$
\end{lemma}

\begin{proof}
Let $$I:=\linspan\lbrace s_\mathcal{T}(\mu) Q(s_\mathcal{T})^E s_\mathcal{T}(\nu)^* : E\in \overline{\mathcal{E}}, \mu,\nu \in \Lambda \text{ and } s(\mu)=s(\nu) \rbrace.$$
Fix $E\in \mathcal{E}$ and $\mu,\nu,\alpha,\beta \in \Lambda$ such that $s(\mu)=s(\nu)$ and $s(\alpha)=s(\beta)$. We calculate,
\begin{align}
s_\mathcal{T}(\alpha)s_\mathcal{T}(&\beta)^* s_\mathcal{T}(\mu) Q(s_\mathcal{T})^{E} s_\mathcal{T}(\nu)^*\label{eq:idealstructure2} \\
&= s_\mathcal{T}(\alpha)\left(\sum_{(\eta,\zeta)\in \Lambda^{\text{min}}(\beta,\mu)} \overline{c(\beta,\eta)}c(\mu,\zeta)s_\mathcal{T}(\eta) s_\mathcal{T}(\zeta)^* \right) Q(s_\mathcal{T})^{E} s_\mathcal{T}(\nu)^*\nonumber\\
&\hskip8.5cm \text{ by Lemma~\ref{handylemma}(1)} \nonumber \\
&=\sum_{(\eta,\zeta)\in \Lambda^{\text{min}}(\beta,\mu)} \overline{c(\beta,\eta)}c(\mu,\zeta)c(\alpha,\eta)s_\mathcal{T}(\alpha \eta) s_\mathcal{T}(\zeta)^* Q(s_\mathcal{T})^{E} s_\mathcal{T}(\nu)^* \nonumber\\
&= \sum_{(\eta,\zeta)\in \Lambda^{\text{min}}(\beta,\mu)} \overline{c(\beta,\eta)}c(\mu,\zeta)c(\alpha,\eta)s_\mathcal{T}(\alpha \eta) \left( Q(s_\mathcal{T})^{E} s_\mathcal{T}(\zeta) \right)^* s_\mathcal{T}(\nu)^* \nonumber \\
&=  \sum_{(\eta,\zeta)\in \Lambda^{\text{min}}(\beta,\mu):\atop \zeta \in r(E)\Lambda \setminus E \Lambda} \overline{c(\beta,\eta)}c(\mu,\zeta)c(\alpha,\eta)s_\mathcal{T}(\alpha \eta) \left( s_\mathcal{T}(\zeta)  Q(s_\mathcal{T})^{\Ext(\zeta;E)}\right)^*  s_\mathcal{T}(\nu)^*\nonumber\\
&\hskip8.5cm \text{ by Lemma~\ref{le:idealstructure1}} \nonumber \\
&=\sum_{(\eta,\zeta)\in \Lambda^{\text{min}}(\beta,\mu):\atop \zeta \in r(E)\Lambda \setminus E \Lambda} \overline{c(\beta,\eta)c(\nu ,\zeta)}c(\mu,\zeta)c(\alpha,\eta)s_\mathcal{T}(\alpha \eta)  Q(s_\mathcal{T})^{\Ext(\zeta;E)}  s_\mathcal{T}(\nu \zeta )^*.\nonumber
\end{align}
Since $\overline{\mathcal{E}}$ is satiated, each $s_\mathcal{T}(\alpha \eta)  Q(s_\mathcal{T})^{\Ext(\zeta;E)}  s_\mathcal{T}(\nu \zeta )^*\in I$, and hence $$s_\mathcal{T}(\alpha)s_\mathcal{T}(\beta)^* s_\mathcal{T}(\mu) Q(s_\mathcal{T})^{E} s_\mathcal{T}(\nu)^*\in I.$$
 Suppose that $a\in \linspan \lbrace s_\mathcal{T}(\alpha) s_\mathcal{T}(\beta)^* : s(\alpha)=s(\beta) \rbrace$ and $b\in I$. Since multiplication is bilinear, equation \eqref{eq:idealstructure2} implies that $ab\in I$. Taking adjoints of both sides of \eqref{eq:idealstructure2} shows that $ba\in I$. We show that $\overline{I}$ is an ideal. Fix $a\in \mathcal{T}C^*(\Lambda,c)$ and $b\in \overline{I}$. By Lemma~\ref{handylemma}(5) there is a sequence $a_n\in \linspan \lbrace s_\mathcal{T}(\alpha) s_\mathcal{T}(\beta)^* : s(\alpha)=s(\beta) \rbrace$ such that $a_n \to a$ as $n \to \infty$. Similarly, there is a sequence $b_n\in I$ such that $b_n \to b$ as $n \to \infty$. As multiplication in a $C^*$-algebra is jointly continuous we have $a_n b_n \to ab$ and $b_n a_n \to ba$ as $n \to \infty$. The above argument shows that $a_n b_n, b_n a_n \in I$ for each $n$ and hence $ab,ba \in \overline {I}$. Hence $\overline {I}$ is a closed ideal in $\mathcal{T}C^*(\Lambda,c)$. Since $I \subset J_\mathcal{E}$ and $J_\mathcal{E}$ is closed, $\overline{I}\subset J_\mathcal{E}$. The set $I$ contains all the gap projections associated to $\mathcal{E}$.  As $J_\mathcal{E}$ is the smallest closed ideal that contains all the gap projections associated to $\mathcal{E}$, we have $\overline I = J_\mathcal{E}$.
\end{proof}

For a $k$-graph $\Lambda$, recall from Definition~\ref{def:graphmorphisms} the collection $\Lambda^*$ of graph morphisms $x: \Lambda \rightarrow \Omega_{k,m}$.

\begin{definition}[\cite{S20061}, Definition 4.3]
Let $(\Lambda,d)$ be a finitely aligned $k$-graph, and let $\mathcal{E}$ be a subset of $\FE(\Lambda)$. We say that $x\in \Lambda^*$ is an $\mathcal{E}$-\emph{relative boundary path} of $\Lambda$ if for every $n\in \mathbb{N}^k$ such that $n\leq d(x)$, and every $E\in\overline{\mathcal{E}}$ such that $r(E)=x(n)$, there exists $\lambda \in E$ such that $x(n,n+d(\lambda))=\lambda$. We denote the collection of all $\mathcal{E}$-relative boundary paths of $\Lambda$ by $\partial(\Lambda;\mathcal{E})$.
\end{definition}

Let $\left\{ T_\lambda : \lambda \in \Lambda \right \}$ be the Toeplitz-Cuntz-Krieger $(\Lambda,c)$-family of Proposition~\ref{toeplitzrepresentation}. The next few technical lemmas use $\mathcal{E}$-relative boundary paths to characterise when, for $E\in \FE(\Lambda)$, the gap projection $Q(T)^E$ belongs to the $C^*$-algebra $\pi_T^\mathcal{T}(J_{\overline{\mathcal{E}}})\subset \mathcal{B}\left( \ell^2(\Lambda)\right)$. This charactisation enables us to prove Theorem~\ref{th:nonzerogapprojection}, the main result of this section.

\begin{lemma}\label{le:gapprojection1}
Let $(\Lambda,d)$ be a finitely aligned $k$-graph, let $c\in \underline{Z}^2(\Lambda,\mathbb{T})$ and let $\mathcal{E}\subset \FE(\Lambda)$. Let $\lbrace T_\lambda : \lambda \in \Lambda \rbrace$ be the Toeplitz-Cuntz-Krieger $(\Lambda,c)$-family of Proposition~\ref{toeplitzrepresentation}. Suppose that $x\in \partial(\Lambda;\mathcal{E})$ and that $\mu,\nu \in \Lambda$ satisfy $s(\mu)=s(\nu)$. If $E\in \overline{\mathcal{E}}$ satisfies $r(E)=s(\mu)$, then there exists $p\in \mathbb{N}^k$ with $p\leq d(x)$ such that $T_\mu Q(T)^E T_\nu^* \xi_{x(0,q)}=0$ for all $q$ such that $p\leq q \leq d(x)$.
\end{lemma}

\begin{proof}

We consider two cases: $d(x)\ngeq d(\nu)$ and then $d(x)\geq d(\nu)$. Suppose first that $d(x)\ngeq d(\nu)$. Fix $q\leq d(x)$. Then $q \ngeq d(\nu)$, so $x(0,q)$ does not have the form $\nu \nu'$. Hence $T^*_\nu \xi_{x(0,q)}=0$. So $p=0$ does the job.

Suppose that $d(x) \geq d(\nu)$. First suppose that $x(0,d(\nu))\neq \nu$. Fix $d(\nu) \leq q \leq d(x)$. The factorisation property implies that $x(0,q)$ does not have the form $\nu \nu'$. So $p=d(\nu)$ does the job.

Now suppose that $d(x) \geq d(\nu)$ and $x(0,d(\nu))=\nu$. Set $n=d(\nu)$. As $x$ is an $\mathcal{E}$-relative boundary path and $r(E)=s(\nu)=x(n)$, there exists $\lambda_0\in E$ such that $x(n,n+d(\lambda_0))=\lambda_0$. Put $p=n+d(\lambda_0)$ and fix $q$ such that $p\leq q \leq d(x).$ Then
\begin{align*}
\left( T_{r(E)}-T_{\lambda_0}T_{\lambda_0}^* \right) T_\nu^* \xi_{x(0,q)} &=\overline{c(\nu,x(n,q))}\left( T_{r(E)}-T_{\lambda_0}T_{\lambda_0}^* \right)\xi_{x(n,q)}\\
&=\overline{c(\nu,x(n,q))}\left( T_{r(E)}\xi_{x(n,q)}-T_{\lambda_0}T_{\lambda_0}^* \xi_{\lambda_0 x(p,q)}\right)\\
&=\overline{c(\nu,x(n,q))}\left(\xi_{x(n,q)}- \xi_{ x(n,q)}\right)  \qquad \text{ by } \eqref{TTstar}\\
&=0.
\end{align*}
Hence
\begin{flalign*}
&&T_\mu Q(T)^E T_\nu^* \xi_{x(0,q)}&=T_\mu Q(T)^E \left( T_{r(E)}-T_{\lambda_0}T_{\lambda_0}^* \right)  T_\nu^* \xi_{x(0,q)}&\\
&&&=0.&\qedhere
\end{flalign*}
\end{proof}

\begin{lemma}\label{le:gapprojection2}
Let $(\Lambda,d)$ be a finitely aligned $k$-graph, let $c\in \underline{Z}^2(\Lambda,\mathbb{T})$ and let $\mathcal{E}\subset \FE(\Lambda)$. Let $\lbrace T_\lambda : \lambda \in \Lambda \rbrace$ be the Toeplitz-Cuntz-Krieger $(\Lambda,c)$-family of Proposition~\ref{toeplitzrepresentation}. For each $a\in \linspan \left \{ T_\mu Q(T)^E T_\nu^* : s(\mu)=s(\nu) \text{ and } E\in \overline{\mathcal{E}}\right \}$ there exists $p\in \mathbb{N}^k$ with $p\leq d(x)$ such that $a\xi_{x(0,q)}=0$ whenever $p\leq q \leq d(x)$.
\end{lemma}

\begin{proof}
Fix $a\in   \linspan \left\{ T_\mu Q(T)^E T_\nu^*  : s(\mu)=s(\nu) \text{ and } E\in \overline{\mathcal{E}}\right \}$, say, $a=\sum_{n=1}^N T_{\mu_n} Q(T)^{E_n} T_{\nu_n}^*$ where $\mu_n,\nu_n \in \Lambda$, $s(\mu_n)=s(\nu_n)$ and $E_n \in \overline{\mathcal{E}}$ for each $n$. For each $n$,  Lemma~\ref{le:gapprojection1} implies that there exists $p_n\in \mathbb{N}^k$ such that $T_{\mu_n} Q(T)^{E_n} T_{\nu_n}\xi_{x(0,q)}=0$ for all $p_n\leq q \leq d(x)$. Put $p=\max \lbrace p_1,...,p_n \rbrace$. Then $a\xi_{x(0,q)}=0$ for all $p\leq q \leq d(x)$.
\end{proof}

\begin{lemma}\label{le:boundarypath1}
Let $(\Lambda,d)$ be a finitely aligned $k$-graph, $c\in \underline{Z}^2(\Lambda,\mathbb{T})$ and $\mathcal{E}\subset \FE(\Lambda)$. Let $\lbrace T_\lambda : \lambda \in \Lambda \rbrace$ be the Toeplitz-Cuntz-Krieger $(\Lambda,c)$-family of Proposition~\ref{toeplitzrepresentation}. For all $x\in \partial(\Lambda;\mathcal{E})$, if $a\in \pi_T^\mathcal{T} (J_{\overline{\mathcal{E}}})$ then $\lim_{q\to d(x)} \Vert a \xi_{x(0,q)} \Vert =0$.
\end{lemma}

\begin{proof}
Fix $\epsilon>0$ and $a\in \pi_T^\mathcal{T}(J_{\overline{\mathcal{E}}})$. There exists $b\in J_{\overline{\mathcal{E}}}$ such that $a= \pi_T^\mathcal{T} (b)$. By Lemma~\ref{le:idealstructure2}, there exists $$c\in \linspan \lbrace s_\mathcal{T}(\mu)  Q(s_\mathcal{T})^E s_\mathcal{T}(\nu)^* : E\in \overline{\mathcal{E}}, \mu,\nu \in \Lambda \text{ and } s(\mu)=s(\nu) \rbrace$$ such that $\Vert b-c\Vert < \epsilon$. Since $\pi_T^\mathcal{T}$ is a homomorphism $$\pi_T^\mathcal{T} (c)\in \linspan \left\{ T_\mu Q(T)^E T_\nu^* : E\in \overline{\mathcal{E}}, \mu,\nu \in \Lambda \text{ and } s(\mu)=s(\nu) \right\}.$$ By Lemma~\ref{le:gapprojection2} there exists $p\in \mathbb{N}^k$ such that whenever $p\leq q \leq d(x)$ we have $ \pi_T^\mathcal{T} (c)\xi_{x(0,q)}=0$. Then whenever $p\leq q \leq d(x)$ we have
\begin{align*}
\left \Vert a\xi_{x(0,q)} \right \Vert &= \left \Vert a\xi_{x(0,q)} - \pi_T^\mathcal{T} (c)\xi_{x(0,q)} \right \Vert\\
 &= \left \Vert \pi_T^\mathcal{T} (b)\xi_{x(0,q)} - \pi_T^\mathcal{T} (c)\xi_{x(0,q)} \right \Vert \\
&=\left \Vert \left(\pi_T^\mathcal{T} (b)- \pi_T^\mathcal{T} (c)\right)\xi_{x(0,q)}  \right \Vert \\
&\leq \left \Vert \pi_T^\mathcal{T} (b)- \pi_T^\mathcal{T} (c)  \right \Vert \qquad \text{since } \left \Vert \xi_{x(0,q)}  \right \Vert=1\\
&\leq \left \Vert\pi_T^\mathcal{T} (b-c)  \right \Vert.
\end{align*}
Since homomorphisms between $C^*$-algebras are norm decreasing, we have $$\left \Vert a\xi_{x(0,q)} \right \Vert \leq \left \Vert\pi_T^\mathcal{T} (b-c)  \right \Vert \leq \left \Vert b-c \right \Vert < \epsilon.$$ Hence $\lim_{q\to d(x)} \left \Vert a \xi_{x(0,q)}  \right \Vert =0.$
\end{proof}

Recall from Lemma~4.7 of \cite{S20061} that if $(\Lambda,d)$ is a finitely aligned $k$-graph, $\mathcal{E} \subset \FE(\Lambda)$, and $v\in \Lambda^0$, then for each $F\in \FE(\Lambda) \setminus \overline{\mathcal{E}}$, both $r(F) \partial (\Lambda,\mathcal{E}) \setminus F\partial (\Lambda; \mathcal{E})$ and $v\partial(\Lambda;\mathcal{E})$ are nonempty.

\begin{lemma}\label{le:boundarypath3}
Let $(\Lambda,d)$ be a finitely aligned $k$-graph, let $c\in \underline{Z}^2(\Lambda,\mathbb{T})$ and let $\mathcal{E}\subset \FE(\Lambda)$. Let $\lbrace T_\lambda : \lambda \in \Lambda \rbrace$ be the Toeplitz-Cuntz-Krieger $(\Lambda,c)$-family of Proposition~\ref{toeplitzrepresentation}. If $F\in \FE(\Lambda)\setminus \overline{\mathcal{E}}$, there exists $x\in \partial(\Lambda;\mathcal{E})$ such that $ \Vert Q(T)^F \xi_{x(0,q)} \Vert =1$ for all $q\leq d(x)$.
\end{lemma}

\begin{proof}
By (\cite{S20061}, Lemma 4.7), there exists $x\in r(F) \partial ( \Lambda ; \mathcal{E}) \setminus F \partial (\Lambda ; \mathcal{E})$. Fix $q \leq d(x)$ and $\lambda \in F$. We have $x(0,d(\lambda)) \neq \lambda$. By the factorisation property $x(0,q) \neq \lambda \lambda ' $ for any $q \leq d(x)$. Hence
\begin{align*}
\left( T_{r(F)}-T_\lambda T_\lambda^* \right) \xi_{x(0,q)} &= \xi_{x(0,q)}.
\end{align*}
Hence $ \left \Vert Q(T)^F \xi_{x(0,q)} \right \Vert =1$.
\end{proof}

\begin{theorem}\label{th:nonzerogapprojection}
Let $(\Lambda,d)$ be a finitely aligned $k$-graph and let $c\in \underline{Z}^2(\Lambda,\mathbb{T})$. Let $\mathcal{E} \subset \FE(\Lambda)$. We have
\begin{enumerate}
\item[(1)] $s_\mathcal{E}(v)\neq 0$ for all $v\in \Lambda^0$; and
\item[(2)] for $G\in \FE(\Lambda)$, $Q(s_\mathcal{E})^G= 0$ if and only if $G\in \overline{\mathcal{E}}$.
\end{enumerate}
\end{theorem}

\begin{proof}
For (1), fix $v\in \Lambda^0$. By (\cite{S20061}, Lemma 4.7) there exists $x\in v\partial(\Lambda;\mathcal{E})$. Let $\lbrace T_\lambda : \lambda \in \Lambda \rbrace$ be the Toeplitz-Cuntz-Krieger $(\Lambda,c)$-family of Proposition~\ref{toeplitzrepresentation}. For any $q\leq d(x)$ we have $T_v\xi_{x(0,q)}=\xi_{x(0,q)}$. Lemma~\ref{le:boundarypath1} implies that $T_v \notin \pi_T^\mathcal{T}(J_{\overline{\mathcal{E}}})$. Since $J_{\overline{\mathcal{E}}}=J_\mathcal{E}$ by Lemma~\ref{le:satiation1}, we have $s_\mathcal{T}(v) \notin J_{\overline{\mathcal{E}}}=J_\mathcal{E}$ and thus $s_\mathcal{E}(v)=s_\mathcal{T}(v)+J_{\mathcal{E}} \neq 0$.
For (2), suppose that $G\in \FE(\Lambda)\setminus \overline{ \mathcal{E}}$. By Lemma~\ref{le:boundarypath3} there is $x\in \partial(\Lambda;\mathcal{E})$ such that $\lim_{p\to d(x)} \left \Vert Q(T)^G \xi_{x(0,q)}\right \Vert =1$. Lemma~\ref{le:boundarypath1} implies that $Q(T)^G \notin \pi_T^{\mathcal{T}}(J_{\overline{\mathcal{E}}})$. Hence $Q(s_\mathcal{T})^G \notin J_{\overline{\mathcal{E}}}=J_\mathcal{E}$ and thus $Q(s_\mathcal{E})^G=Q(s_\mathcal{T})^G+J_{\mathcal{E}} \neq 0$. If $G\in \overline{\mathcal{E}}$ then $Q(s_\mathcal{T})^G\in J_{\overline{\mathcal{E}}}=J_\mathcal{E}$. So $Q(s_\mathcal{E})^G=0$.
\end{proof}

\chapter{Analysis of the core}
\label{chptr:Analysis_of_the_core}
In this chapter we establish the existence of a strongly continuous group action $\gamma$ of $\mathbb{T}^k$ on $C^*(\Lambda,c;\mathcal{E})$. We call $\gamma$ the gauge action. Averaging over $\gamma$ gives a faithful conditional expectation of $C^*(\Lambda,c;\mathcal{E})$ onto the fixed point algebra for $\gamma$, which we call the core. We show that the core $C^*(\Lambda,c;\mathcal{E})^\gamma$ is the closed linear span of elements of the form $s_\mathcal{E}(\lambda)s_\mathcal{E}(\mu)^*$ where $d(\lambda)=d(\mu)$. We show that $C^*(\Lambda,c;\mathcal{E})^\gamma$ is AF. If $\left\{ t_\lambda :\lambda \in \Lambda \right \}$ is a relative Cuntz-Krieger $(\Lambda,b;\mathcal{E})$-family, we provide conditions under which
$$C^*(\Lambda,c;\mathcal{E})^\gamma \cong \overline{\linspan}\left\{ t_\lambda t_\mu^* : \lambda ,\mu \in \Lambda : d(\lambda)=d(\mu) \right \}.$$
If $b\equiv c$ and there is an action $\beta : \mathbb{T}^k \to \Aut\left(C^*\left(\left\{ t_\lambda : \lambda \in \Lambda \right \} \right)\right)$ such that $\pi_t^\mathcal{E}$ is equivariant for $\gamma$ and $\beta$, we show that $\pi_t^\mathcal{E}$ is injective. This generalisation of an Huef and Raeburn's gauge-invariant uniqueness theorem is the main result of the chapter.
\section{Group actions and faithful conditional expectations}
\label{chptr:groups_actions}
This section presents standard results associated to groups actions and faithful conditional expectations. These results may be recovered, somewhat nontrivially, from \cite{EKQR2006}. Many can also be found in \cite{W2007}. For the most part, we do not give proofs, but just recall the results we will need later.

\begin{definition}
Let $A$ be a $C^*$-algebra and let $G$ a compact abelian group. A \emph{group action} is a homomorphism
\begin{align*}
\alpha : G &\to \Aut(A), \text{ written } g \mapsto \alpha_g,
\end{align*}
that is strongly continuous; that is, whenever $g_n \to g$ as $n\to \infty$ in $G$ we have $\alpha_{g_n}(a) \to \alpha_g(a)$ as $n\to \infty$ in $A$ for each $a\in A$.
\end{definition}

\begin{definition}
Let $A$ be a $C^*$-algebra and let $G$ a compact abelian group. We write $\widehat G$ for the dual group of $G$; that is, the group of homomorphisms $\phi: G \to \mathbb{T}$ with pointwise multiplication. We define $$A^\alpha := \left \{ a\in A : \alpha_g(a) = a \text{ for all } g\in G \right \},$$ and for $\phi \in \widehat G$ we define $$A_\phi :=\left \{ a\in A : \alpha_g(a)=\phi(g)a \text{ for all } g\in G \right \}.$$
Note that $A^\alpha=A_{ \hat 1}$ where $\hat 1\in \widehat G$ is given by $\hat 1(g)=1$ for all $g\in G$.
\end{definition}

\begin{lemma}\label{le:spectralsubspaces}
Let $A$ be a $C^*$-algebra, let $G$ a compact abelian group and let $\alpha:G \to \Aut(A)$ be a group action. Then:
\begin{enumerate}
\item[(1)] $A_\phi$ is a closed subspace for each $\phi \in \widehat G$;
\item[(2)]$A^\alpha $ is a $C^*$-subalgebra of $A$;
\item[(3)] $A=\overline{ \linspan } \left \{ A_\phi : \phi\in \widehat G \right \}$; and
\item[(4)] $A_\phi \cdot A_\psi=A_{\phi \psi}$ for all $\phi,\psi \in \widehat G$.
\end{enumerate}
\end{lemma}

\begin{definition}
Let $A$ be a $C^*$-algebra and let $B$ be a $C^*$-subalgebra of $A$. A map $\Phi: A \to B$ is called a \emph{faithful conditional expectation} if:
\begin{enumerate}
\item[(1)] $\Phi$ is linear;
\item[(2)] $\Phi$ is bounded with $\Vert \Phi \Vert =1$;
\item[(3)] $\Phi^2=\Phi$; and
\item[(4)] $a\neq 0 \Longrightarrow \Phi(a^* a)>0$.
\end{enumerate}
\end{definition}

\begin{lemma}\label{le:condexpdiagram}
Let $A$ and $B$ be $C^*$-algebras and let $\pi : A \to B$ be a homomorphism. Let $A_0$ and $B_0$ be $C^*$-subalgebras of $A$ and $B$ respectively. Suppose there are a faithful conditional expectation $\Phi : A \to A_0$ and a linear map $\Psi : B \to B_0$ such that the diagram
\begin{large}
\begin{center}
\begin{tikzpicture}
\matrix(m)[matrix of math nodes,
row sep=4em, column sep=4em,
text height=1.5ex, text depth=0.25ex]
{A&B\\
A_0&B_0\\};
\path[->]
(m-1-1) edge node[auto]{$\pi$} (m-1-2)
(m-2-1) edge node[below]{$\pi|_{A_0}$}  (m-2-2);
\path[->]
(m-1-1) edge node[left]{$\Phi$}(m-2-1)
(m-1-2) edge node[auto]{$\Psi$} (m-2-2);
\end{tikzpicture}
\end{center}
\end{large}
commutes. Then $\pi$ is injective if and only if $\pi_0$ is injective.
\end{lemma}
\begin{proof}
Let $\pi_0:= \pi|_{A_0}$. If $\pi$ is injective then so is $\pi_0$. Suppose $\pi_0$ is injective. Then
\begin{align*}
\pi(a)=0 &\Longrightarrow \pi(a)^* \pi(a)=0\\
&\Longrightarrow \pi(a^* a)=0\qquad \text{ since } \pi \text{ is a homomorphism}\\
&\Longrightarrow \Psi \left( \pi(a^*a) \right)=0 \qquad\text{ since } \Psi \text{ is linear }\\
&\Longrightarrow \pi_0 \left( \Phi(a^* a) \right)=0 \qquad\text{ since the above diagram commutes}\\
&\Longrightarrow \Phi(a^* a)=0 \quad \text{ since } \pi_0 \text{ is injective }\\
&\Longrightarrow a=0 \text{ since } \Phi \text{ is a faithful conditional expectation.}
\end{align*}
As $\pi$ is linear it follows that $\pi$ is injective.
\end{proof}

\begin{lemma}\label{le:conexpexist}
Let $A$ be a $C^*$-algebra, let $G$ a compact abelian group and let $\alpha:G \to \Aut(A)$ be a group action. Let $\mu$ denote the normalised Haar measure on $G$. Then $\Phi^\alpha : A \to A$ given by
$$\Phi^\alpha(a)=\int_G \alpha_g (a) d\mu(g)$$
is a faithful conditional expectation onto $A^\alpha$. Moreover,
$$
\Phi^\alpha (a) = \begin{cases} a &\text{if } a\in A^\alpha \\
0 &\text{ if } a\in A_\phi \text{ for some } \phi \in \widehat G \setminus \left\{ \hat 1 \right \}.
\end{cases}$$
\end{lemma}

\section{The gauge action}
This section establishes the existence of the gauge action. We show that the gauge action is strongly continuous and study its fixed point algebra. We study properties of the faithful conditional expectation obtained by averaging over the gauge action.

Let $(\Lambda,d)$ be a finitely aligned $k$-graph. For $\lambda \in \Lambda$ and $z\in \mathbb{T}^k$, define  $$z^{d(\lambda)}:=\prod_{i=1}^k z_i^{d(\lambda)_i}\in \mathbb{T}.$$

\label{chptr:gauge_action}
\begin{lemma}\label{le:gaugeauto}
Let $(\Lambda,d)$ be a finitely aligned $k$-graph, let $c\in \underline{Z}^2(\Lambda,\mathbb{T})$ and let $\mathcal{E}$ be a subset of $\FE(\Lambda)$. Fix $z\in \mathbb{T}^k$. There is an automorphism $\gamma_z$ of $C^*(\Lambda,c,\mathcal{E})$ such that $\gamma_z \left( s_\mathcal{E}(\lambda) \right)=z^{d(\lambda)}s_\mathcal{E} (\lambda)$ for all $\lambda \in \Lambda.$
\end{lemma}

\begin{proof}
We show that \begin{equation}\lbrace z^{d(\lambda)}s_\mathcal{E}(\lambda) : \lambda \in \Lambda \rbrace\label{eq:buildinggaugefamily}\end{equation} is a Cuntz-Krieger $(\Lambda,c;\mathcal{E})$-family in $C^*(\Lambda,c;\mathcal{E})$. Theorem~\ref{th:relativecuntzkriegeralgebra} will then imply that there is a homomorphism $\gamma_z : C^*(\Lambda,c;\mathcal{E}) \to C^*(\Lambda,c;\mathcal{E})$ such that $\gamma_z \left( s_\mathcal{E}(\lambda) \right) = z^{d(\lambda)}s_\mathcal{E}(\lambda)$ for all $\lambda \in \Lambda$. Then we show that $\gamma_{\overline z} \circ \gamma_z = \gamma_z \circ \gamma_{\overline z} =\id$ so that $\gamma_z$ is an automorphism.

First we show (TCK1). Fix $v\in \Lambda^0$. Since $z^{d(v)}=1$, we have $z^{d(v)}s_\mathcal{E}(v)=s_\mathcal{E}(v)$. So (TCK1) for $\left\{ z^{d(v)}s_\mathcal{E}(v): v\in \Lambda^0 \right \}$ follows from (TCK1) for $\left\{s_\mathcal{E}(v): v\in \Lambda^0 \right \}$.

For each $\lambda \in \Lambda$, $\left(z^{d(\lambda)}s_\mathcal{E}(\lambda) \right)\left(z^{d(\lambda)}s_\mathcal{E}(\lambda) \right)^*=s_\mathcal{E}(\lambda)s_\mathcal{E}(\lambda)^*$ and
$\left(z^{d(\lambda)}s_\mathcal{E}(\lambda) \right)^*\left(z^{d(\lambda)}s_\mathcal{E}(\lambda) \right)=s_\mathcal{E}(\lambda)^*s_\mathcal{E}(\lambda)$. So (TCK3), (TCK4) and (CK) for $\left\{ z^{d(\lambda)}s_\mathcal{E}(\lambda): \lambda \in \Lambda \right \}$ follows from (TCK3), (TCK4) and (CK) for $\left\{ s_\mathcal{E}(\lambda): \lambda \in \Lambda \right \}$. So it remains to show (TCK2). Fix $\mu,\nu \in \Lambda$ such that $s(\mu)=r(\nu)$. Then
\begin{align*}\left( z^{d(\mu)}s_\mathcal{E}(\mu)\right) \left( z^{d(\nu)}s_\mathcal{E}(\nu)\right)&= z^{d(\mu)}z^{d(\nu)} s_\mathcal{E}(\mu)s_\mathcal{E}(\nu)\\
&=z^{d(\mu)+d(\nu)} c(\mu,\nu)s_\mathcal{E}(\mu \nu)\\
&=c(\mu,\nu) \left( z^{d(\mu\nu)} s_\mathcal{E}(\mu \nu)\right),
\end{align*}
which establishes (TCK2). So $\lbrace z^{d(\lambda)}s_\mathcal{E}(\lambda) : \lambda \in \Lambda \rbrace$ is a Cuntz-Krieger $(\Lambda,c;\mathcal{E})$-family in $C^*(\Lambda,c;\mathcal{E})$.

By Theorem~\ref{th:relativecuntzkriegeralgebra} there is a homomorphism $\gamma_z : C^*(\Lambda,c;\mathcal{E}) \to C^*(\Lambda,c;\mathcal{E})$ such that $\gamma_z \left( s_\mathcal{E}(\lambda) \right) = z^{d(\lambda)}s_\mathcal{E}(\lambda)$ for all $\lambda \in \Lambda$. Fix $\mu \in \Lambda$. We have
\begin{align*}
(\gamma_{\overline{z}} \circ \gamma_{z}) (s_\mathcal{E}(\mu))\\
&=\gamma_{\overline{z}}( z^{d(\mu)} s_\mathcal{E}(\mu)\\
&=z^{d(\mu)}{\overline{z}}^{d(\mu)} s_\mathcal{E}(\mu)\\
&=s_\mathcal{E}(\mu).
\end{align*}
By Theorem~\ref{th:relativecuntzkriegeralgebra} the collection $\lbrace s_\mathcal{E}(\lambda) : \lambda \in \Lambda \rbrace$ generates $C^*(\Lambda,c;\mathcal{E})$. Since $\gamma_{\overline{z}}\circ \gamma_z$ is a homomorphism, we have $\gamma_{\overline z} \circ \gamma_z =\id$. The same calculation but with $z$ replaced with $\overline z$ gives $\gamma_{\overline z} \circ \gamma_z =\gamma_z \circ  \gamma_{\overline z} =\id$. So $\gamma_z$ is an automorphism.
\end{proof}

\begin{lemma}\label{le:stronglycontinuous}
Let $(\Lambda,d)$ be a finitely aligned $k$-graph, let $c\in \underline{Z}^2(\Lambda,\mathbb{T})$ and let $\mathcal{E}$ be a subset of $\FE(\Lambda)$. Then $z \mapsto \gamma_z$ is a homomorphism from $\mathbb{T}^k$ to $\Aut(C^*(\Lambda,c;\mathcal{E}))$. This homomorphism is strongly continuous in the sense that if $z_n \to z$ as $n \to \infty$ in $\mathbb{T}^k$ then $\gamma_{z_n}(a) \to \gamma_z(a)$ as $n \to \infty$ in $C^*(\Lambda,c;\mathcal{E})$ for every $a\in C^*(\Lambda,c;\mathcal{E})$. The map $\gamma:z\mapsto \gamma_z$ is called the gauge action.
\end{lemma}

\begin{proof}
First we show that $z \mapsto \gamma_z$ is a homomorphism from $\mathbb{T}^k$ to $\Aut(C^*(\Lambda,c;\mathcal{E}))$. For each $z\in \mathbb{T}$, Lemma~\ref{le:gaugeauto} implies that $\gamma_z$ is an automorphism. So it suffices to show that $z\mapsto \gamma_z$ is a homomorphism. To see this, fix $z,\omega \in \mathbb{T}^k$ and $\mu \in \Lambda$. Since $\gamma_z$ and $\gamma_\omega$ are homomorphisms we have
\begin{align*}
\gamma_{z \omega} \left( s_\mathcal{E} (\mu) \right) &= \left( z \omega\right)^{d(\mu)} s_\mathcal{E}(\mu)\\
&=  z^{d(\mu)} \omega^{d(\mu)} s_\mathcal{E}(\mu)\\
&=   \omega^{d(\mu)} \gamma_z\left( s_\mathcal{E}(\mu) \right) \\
&=   \gamma_z \left( \omega^{d(\mu)}  s_\mathcal{E}(\mu)\right) \\
&=   \gamma_z \left(\gamma_\omega \left( s_\mathcal{E}(\mu) \right)\right)\\
&= \left( \gamma_z \circ \gamma_{\omega}\right)( s_\mathcal{E} (\mu)).
\end{align*}
Since $\gamma_{z\omega}$ and $\gamma_z \circ \gamma_\omega$ are homomorphisms, we deduce that $z \mapsto \gamma_z$ is a homomorphism.

We will now show that this homomorphism is strongly continuous. Fix $\epsilon>0$, $a\in C^*(\Lambda,c;\mathcal{E})$, $z\in \mathbb{T}^k$ and a sequence $(z_n)_{n=1}^\infty$ in $\mathbb{T}$ such that $z_n \to z$ as $n \to \infty$. The case where $a=0$ is trivial since $\gamma_{z_n}(0)=0$ for all $n$, so we assume that $a\neq 0$. Since $$C^*(\Lambda,c;\mathcal{E}) = \overline{\linspan} \lbrace s_\mathcal{E}(\mu)s_\mathcal{E}(\nu)^* :\mu,\nu\in \Lambda, s(\mu)=s(\nu) \rbrace,$$ there is nonempty and finite $F\subset \Lambda *_s \Lambda$ and $b=\sum_{(\mu,\nu)\in F} b_{\mu,\nu} s_\mathcal{E}(\mu) s_\mathcal{E}(\nu)^*$ with $\Vert a-b \Vert < \frac{\epsilon}{3}$ where $b_{\mu,\nu}\neq 0$ for all $(\mu,\nu)\in F$. Now
\begin{align*}
\left \Vert \gamma_{z_n}(b)-\gamma_z (b) \right \Vert &= \left \Vert \sum_{(\mu,\nu)\in F} b_{\mu,\nu} \gamma_{z_n} \left( s_\mathcal{E}(\mu) s_\mathcal{E}(\nu)^* \right) - \sum_{(\mu,\nu)\in F} b_{\mu,\nu} \gamma_z \left( s_\mathcal{E} (\mu) s_\mathcal{E} (\nu)^* \right) \right \Vert\\
&= \left \Vert \sum_{(\mu,\nu)\in F} b_{\mu,\nu} z_n^{d(\mu)} \overline{z_n^{d(\nu)}} s_\mathcal{E}(\mu) s_\mathcal{E}(\nu)^*- \sum_{(\mu,\nu)\in F} b_{\mu,\nu}  z^{d(\mu)}\overline{z^{d(\nu)}}s_\mathcal{E} (\mu) s_\mathcal{E} (\nu)^*  \right \Vert\\
&= \left \Vert \sum_{(\mu,\nu)\in F} b_{\mu,\nu} z_n^{d(\mu)-d(\nu)} s_\mathcal{E}(\mu) s_\mathcal{E}(\nu)^*- \sum_{(\mu,\nu)\in F} b_{\mu,\nu}  z^{d(\mu)-d(\nu)}s_\mathcal{E} (\mu) s_\mathcal{E} (\nu)^*  \right \Vert\\
&= \left \Vert \sum_{(\mu,\nu)\in F} b_{\mu,\nu} \left( z_n^{d(\mu)-d(\nu)} -z^{d(\mu)-d(\nu)} \right) s_\mathcal{E}(\mu) s_\mathcal{E}(\nu)^*  \right \Vert\\
&\leq \sum_{(\mu,\nu)\in F}|b_{\mu,\nu}|\left|z_n^{d(\mu)-d(\nu)} -z^{d(\mu)-d(\nu)} \right|\Vert s_\mathcal{E}(\mu)\Vert \Vert s_\mathcal{E}(\nu)\Vert\\
&\leq \sum_{(\mu,\nu)\in F}|b_{\mu,\nu}|\left|z_n^{d(\mu)-d(\nu)} -z^{d(\mu)-d(\nu)} \right| \qquad \text{ since } \Vert s_\mathcal{E}(\mu) \Vert , \Vert s_\mathcal{E}(\nu) \Vert \leq 1.
\end{align*}
For $\omega \in \mathbb{T}^k$, the map $\omega \mapsto \omega^{d(\mu)-d(\nu)}$  is continuous for each $(\mu,\nu)\in F$. Hence $z_n^{d(\mu)-d(\nu)}\to z^{d(\mu)-d(\nu)}$ as $n \to \infty$ in $\mathbb{T}^k$. So for each $(\mu,\nu)\in F$, there is $N_{\mu,\nu} \in \mathbb{N}$ such that if $n \geq N_{\mu,\nu}$ then $\left|z_n^{d(\mu)-d(\nu)} -z^{d(\mu)-d(\nu)} \right|< \frac{\epsilon}{3|b_{\mu,\nu}||F|}$. So whenever $n \geq N := \max \left \{ N_{\mu,\nu}: (\mu,\nu)\in F \right \}$, the above calculation shows that
\begin{align*}
\left \Vert \gamma_{z_n}(b)-\gamma_z (b) \right \Vert &\leq \sum_{(\mu,\nu)\in F}|b_{\mu,\nu}|\left|z_n^{d(\mu)-d(\nu)} -z^{d(\mu)-d(\nu)} \right|\\
&< \sum_{(\mu,\nu)\in F}|b_{\mu,\nu}|\frac{\epsilon}{3|b_{\mu,\nu}||F|}\\
&=\frac{\epsilon}{3}.
\end{align*}
Since homomorphisms of $C^*$-algebras are norm decreasing, when we have $n\geq N$
\begin{align*}
\Vert \gamma_{z}(a)-\gamma_{z_n}(a) \Vert &\leq \Vert \gamma_{z}(a)-\gamma_{z}(b) \Vert +\Vert \gamma_{z}(b)-\gamma_{z_n}(b) \Vert +\Vert \gamma_{z_n}(b)-\gamma_{z_n}(a) \Vert \\
&\leq \Vert a-b\Vert +\Vert \gamma_{z}(b)-\gamma_{z_n}(b) \Vert +\Vert b-a \Vert \\
&<\frac{\epsilon}{3} +\frac{\epsilon}{3} +\frac{\epsilon}{3} \\
&=\epsilon.
\end{align*}
So $\gamma_{z_n}(a) \to \gamma_z(a)$ as $n \to \infty$ in $C^*(\Lambda,c;\mathcal{E})$ for every $a\in C^*(\Lambda,c;\mathcal{E})$.
\end{proof}

\begin{lemma}\label{le:corestructure}
Let $(\Lambda,d)$ be a finitely aligned $k$-graph, let $c\in \underline{Z}^2(\Lambda,\mathbb{T})$ and let $\mathcal{E}$ be a subset of $\FE(\Lambda)$. Let $\left\{ t_\lambda : \lambda \in \Lambda \right\}$ be a relative Cuntz-Krieger $(\Lambda,c;\mathcal{E})$-family generating a $C^*$-algebra $B$. Suppose that there is an action $\beta: \mathbb{T}^k \to \Aut(B)$ such that $\beta_z(t_\lambda)=z^{d(\lambda)}t_\lambda$ for all $\lambda \in \Lambda$. We have
 $$B^\beta= \overline { \linspan }\left \{ t_\mu t_\nu^*: \mu,\nu \in \Lambda \text { and }d(\mu)=d(\nu)\right \}.$$
\end{lemma}

\begin{proof}
We claim that
\begin{equation} \Phi^\beta(t_\mu t_\nu^*)=\delta_{d(\mu),d(\nu)} t_\mu t_\nu^* \label{eq:faithcondontothecorsubalge}\end{equation}
for all $\mu,\nu \in \Lambda$, where $\delta$ denotes the Kronecker delta function. Recall that each $m\in \mathbb{Z}^k$ determines a character $\omega_m \in \widehat{\mathbb{T}}^k$ by $\omega_m(z)=z^m$ for each $z\in \mathbb{T}^k$: in fact, $m\mapsto \omega_m$ is isomorphism, although we do not need this. Fix $\mu,\nu\in \Lambda$. If $d(\mu)=d(\nu)$, then $\omega_{d(\mu)-d(\nu)}=\hat 1$. Now suppose $d(\mu)\neq d(\nu)$. There exists $i$ such that $d(\mu)_i \neq d(\nu)_i$. Fix $\zeta \in \mathbb{T}$ such that $\zeta^{d(\mu)_i-d(\nu)_i}\neq 1$, and define $z\in \mathbb{T}^k$ by $z_i=\zeta$ and $z_j=1$ for $j\neq i$. Then $\omega_{d(\mu)-d(\nu)}(z)_i=\zeta^{d(\mu)-d(\nu)}\neq 1$, demonstrating that $\omega_{d(\mu)-d(\nu)}\neq \hat 1$. To show~\eqref{eq:faithcondontothecorsubalge}, it now suffices by Lemma~\ref{le:conexpexist} to show that $t_\mu t_\nu^*\in B_{\omega_{d(\mu)-d(\nu)}}$ for each $\mu,\nu \in \Lambda$. For this, we calculate
\begin{align*}
\gamma_z \left (t_\mu t_\nu^* \right) &= \gamma_z \left( t_\mu \right) \gamma_z \left( t_\nu \right)^*\\
&=\left( z^{d(\mu)} t_\mu \right) \left( z^{d(\nu)}t_\nu \right)^* \\
&=z^{d(\mu)} \overline{z^{d(\nu)}} t_\mu t_\nu^*\\
&=z^{d(\mu)-d(\nu)}t_\mu t_\nu^*\\
&=\omega_{d(\mu)-d(\nu)}(z)t_\mu t_\nu^*,
\end{align*}
establishing~\eqref{eq:faithcondontothecorsubalge}.
Lemma~\ref{le:conexpexist} implies that $\Phi^\beta\left(B\right)=B^\beta$. This, together with~\eqref{eq:faithcondontothecorsubalge}, gives
\begin{flalign*}
&&B^\beta &= \Phi^\beta(B )&\\
&&&=\overline{ \linspan} \left \{ \Phi^\beta \left( t_\mu t_\nu^* \right) : \mu,\nu \in \Lambda \right \} \qquad \text{ since } \Phi^\beta \text{ is linear and continuous }&\\
&&&=\overline { \linspan }\left \{ t_\mu t_\nu^* : \mu,\nu \in \Lambda \text { and }d(\mu)=d(\nu)\right \}.&\qedhere
\end{flalign*}
\end{proof}

\begin{definition}
Let $(\Lambda,d)$ be a finitely aligned $k$-graph, let $c\in \underline{Z}^2(\Lambda,\mathbb{T})$ and let $\mathcal{E}\subset\FE(\Lambda)$. Applying Lemma~\ref{le:corestructure} to the gauge-action $\gamma$ on $C^*(\Lambda,c;\mathcal{E})$ we obtain
$$C^*(\Lambda,c;\mathcal{E})^\gamma=\overline{\linspan}\left\{ s_\mathcal{E}(\mu)s_\mathcal{E}(\nu)^*: \mu,\nu \in \Lambda \text{ and } d(\mu)=d(\nu) \right \},$$
we call this the \emph{core} of $C^*(\Lambda,c;\mathcal{E}).$
\end{definition}

\begin{proposition}\label{prop:actiontheninjectiveoncore}
Let $(\Lambda,d)$ be a finitely aligned $k$-graph, let $c\in \underline{Z}^2(\Lambda,\mathbb{T})$ and let $\mathcal{E}\subset\FE(\Lambda)$. Let $\gamma : \mathbb{T}^k \to \Aut(C^*(\Lambda,c;\mathcal{E}))$ be the gauge action. Suppose that $\left \{ t_\lambda : \lambda \in \Lambda \right \}$ is a relative Cuntz-Krieger $(\Lambda,c;\mathcal{E})$-family in a $C^*$-algebra $B$ and that there exists a group action $\beta: \mathbb{T}^k \to \Aut(B)$ such that $\beta_z (t_\lambda) =z^{d(\lambda)}t_\lambda$ for all $\lambda \in \Lambda$ and $z\in \mathbb{T}^k$. Write $\pi:=\pi_{t}^{\mathcal{E}}$ for the homomorphism $C^*(\Lambda,c;\mathcal{E}) \to B$ given by the universal property. If $\pi |_{C^*(\Lambda,c;\mathcal E)^\gamma}$ is injective then $\pi$ is injective.
\end{proposition}
\begin{proof}
By Lemma~\ref{le:conexpexist} there are faithful conditional expectations $\Phi^\gamma: C^*(\Lambda,c;\mathcal{E}) \to C^*(\Lambda,c;\mathcal{E})^\gamma$ and $\Phi^\beta : B \to B^\beta$. By Lemma~\ref{le:corestructure}, $\Phi^\gamma$ and $\Phi^\beta$ satisfy
$$\Phi^\gamma(s_\mathcal{E}(\mu)s_\mathcal{E}(\nu)^*)= \begin{cases} s_\mathcal{E}(\mu)s_\mathcal{E}(\nu)^* &\text{ if } d(\mu)=d(\nu) \\
0 &\text{ otherwise,} \end{cases}$$
and
$$\Phi^\beta(t_\mu t_\nu^*)= \begin{cases} t_\mu t_\nu^* &\text{if } d(\mu)=d(\nu) \\
0 &\text{otherwise.} \end{cases}$$

If $\mu,\nu \in \Lambda$ satisfy $d(\mu)=d(\nu)$, then
\begin{align*}
\Phi^\beta \left (\pi(s_\mathcal{E}(\mu)s_\mathcal{E}(\nu)^*) \right)=\Phi^\beta(t_\mu t_\nu^*)=t_\mu t_\nu^*.
\end{align*}
So $\pi$ restricts to a homomorphism from  $C^*(\Lambda,c;\mathcal{E})^\gamma$ to $B^\beta$ by linearity and continuity.
We have
\begin{align*}\left(\pi |_{C^*(\Lambda,c;\mathcal{E})^\gamma}\circ  \Phi^\gamma \right)( s_\mathcal{E}(\mu)s_\mathcal{E}(\nu)^*)&= \begin{cases} t_\mu t_\nu^* &\text{ if } d(\mu)=d(\nu) \\
0 &\text{ otherwise} \end{cases}\\
&=\left(  \Phi^\beta \circ \pi \right) \left( s_\mathcal{E}(\mu) s_\mathcal{E}(\nu)^* \right).
\end{align*}
So $\Phi^\gamma \circ \pi|_{C^*(\Lambda,c;\mathcal{E})^\gamma)} =\pi \circ \Phi^\beta$ on the spanning elements of $C^*(\Lambda,c;\mathcal{E})$ and hence on all of $C^*(\Lambda,c;\mathcal{E})$. Thus, the diagram
\begin{center}
\begin{tikzpicture}
\matrix(m)[matrix of math nodes,
row sep=6em, column sep=6em,
text height=1.5ex, text depth=0.25ex]
{C^*(\Lambda,c;\mathcal{E})&B\\
C^*(\Lambda,c;\mathcal{E})^\gamma&B^\beta\\};
\path[->]
(m-1-1) edge node[auto]{$\pi$} (m-1-2)
(m-2-1) edge node[below]{$\pi |_{C^*(\Lambda,c;\mathcal{E})^\gamma}$}  (m-2-2);
\path[->]
(m-1-1) edge node[left]{$\Phi^\gamma$}(m-2-1)
(m-1-2) edge node[auto]{$\Phi^\beta$} (m-2-2);
\end{tikzpicture}\label{dia:commute2}
\end{center}

commutes. Lemma~\ref{le:condexpdiagram} implies that if $\pi |_{C^*(\Lambda,c;\mathcal E)^\gamma}$ is injective then $\pi$ is injective.
\end{proof}

\section{Orthogonalising range projections}
\label{chptr:orthogonalising_range}
We want to establish conditions under which the canonical homomorphism $\pi_t^\mathcal{E}$ determined by a relative Cuntz-Krieger $(\Lambda,c;\mathcal{E})$-family $\left\{ t_\lambda :\lambda \in \Lambda \right \}$ restricts to an isomorphism of the core. Our strategy is to show that $\overline{\linspan}\left\{ t_\lambda t_\mu^* : \lambda,\mu \in \Lambda , d(\lambda)=d(\mu)\right\}$ is an increasing union of finite-dimensional subalgebras. To find matrix units for these finite dimensional subalgebras, we first ``orthogonalise" the generators of $\overline{\linspan}\left\{ t_\lambda t_\mu^* : \lambda,\mu \in \Lambda , d(\lambda)=d(\mu)\right\}$. We begin by concentrating on the ``diagonal" subalgebra $\overline{\linspan}\left\{ t_\mu t_\mu^* : \mu \in \Lambda \right \}$.

\begin{definition}
Let $(\Lambda,d)$ be a finitely aligned $k$-graph, let $c\in \underline{Z}^2(\Lambda,\mathbb{T})$ and let $\left \{ t_\lambda : \lambda \in \Lambda \right \}$ be a Toeplitz-Cuntz-Krieger $(\Lambda,c)$-family. Suppose that $E$ is a finite subset of $\Lambda$. For $\lambda \in \Lambda$, we define
$$Q(t)_\lambda^E:= t_\lambda t_\lambda^* \prod_{\lambda \alpha \in E \atop d(\alpha)>0} \left( t_\lambda t_\lambda^* - t_{\lambda \alpha}t_{\lambda \alpha}^* \right).$$
\end{definition}

\begin{remark}\label{re:proj1}
\begin{enumerate}
 \item[(1)] Lemma~\ref{handylemma}(2) shows that for a finite set $E\subset \Lambda$, the product $Q(t)_\lambda^{E}$ is both well-defined and a projection for each $\lambda \in E$.
\item[(2)] By convention, when a formal product in a $C^*$-algebra $A$ is indexed by the empty set, it is taken to be equal to the unit of the multiplier algebra of $A$. In particular, if $\lambda \in E$ is such that there exists no $\nu \in \Lambda \setminus \Lambda^0$ with $\lambda \nu \in E$, then $Q(t)_{\lambda}^E=t_\lambda t_\lambda^* \cdot 1_{\mathcal{M}(C^*(\left \{ t_\lambda : \lambda \in \Lambda \right \}))}=t_\lambda t_\lambda^*.$
\end{enumerate}
\end{remark}

\begin{definition}
Let $(\Lambda,d)$ be a finitely aligned $k$-graph. A subset $E$ of $\Lambda$ is said to be \emph{closed under minimal common extensions} if
$$ \lambda ,\mu \in E \Longrightarrow \MCE(\lambda,\mu) \subset E.$$
\end{definition}

The main result of this section is the following Proposition.

\begin{proposition}[cf. \cite{RSY2004}, Proposition 3.5]\label{prop:orth1}
Let $(\Lambda,d)$ be a finitely aligned $k$-graph, let $c\in \underline{Z}^2(\Lambda,\mathbb{T})$ and let $\left \{ t_\lambda : \lambda \in \Lambda \right \}$ be a Toeplitz-Cuntz-Krieger $(\Lambda,c)$-family. Let $E$ be a finite subset of $\Lambda$ that is closed under minimal common extensions. Then $\left \{ Q(t)_\lambda^E : \lambda \in E \right \}$ is a collection of mutually orthogonal (possibly zero) projections such that
\begin{equation} t_v \prod_{\lambda \in v E} \left ( t_v - t_\lambda t_\lambda^* \right)+\sum_{\lambda \in vE}Q(t)_\lambda^E =t_v\label{eq:orth1}\end{equation}
for each $v\in r(E).$
\end{proposition}

In order to prove Proposition~\ref{prop:orth1} we reduce to the case where $E\subset v \Lambda$ for some $v\in \Lambda^0$, and $E$ contains $v$.

\begin{lemma}\label{le:orth2}
Let $(\Lambda,d)$ be a finitely aligned $k$-graph, let $c\in \underline{Z}^2(\Lambda,\mathbb{T})$ and let $\left \{ t_\lambda : \lambda \in \Lambda \right \}$ be a Toeplitz-Cuntz-Krieger $(\Lambda,c)$-family. Let $v\in \Lambda^0$. Suppose that $E$ is a finite subset of $v\Lambda$, that $v\in E$ and that $E$ is closed under minimal common extensions. Then $\left \{ Q(t)_\lambda^E : \lambda \in E \right \}$ is a collection of mutually orthogonal (possibly zero) projections such that
\begin{equation} \sum_{\lambda \in E} Q(t)_\lambda^E=t_v \label{eq:orth2}. \end{equation}
\end{lemma}

Using Lemma~\ref{le:orth2} we prove Proposition~\ref{prop:orth1}.

\begin{proof}[Proof of  Proposition~\ref{prop:orth1}]
The following argument is taken from the work of Raeburn, Sims and Yeend in \cite{RSY2004}. The difference here being that we must account for the 2-cocycle, but it turns out that this plays little r\^ole.

Remark~\ref{re:proj1} shows that $Q(t)_\lambda^E$ is projection for each $\lambda \in E$. Then since each $t_v$ is a projection for each $v\in r(E)$, equation \eqref{eq:orth2} shows that they are mutually orthogonal. So we are left to prove \eqref{eq:orth1}.

First, suppose that $F\subset v \Lambda$ is closed under minimal common extensions. The case where $v\in F$ is exactly Lemma~\ref{le:orth2}, so suppose that $v\notin F$. We must show that $F$ satisfies \eqref{eq:orth1}. Let $E:= F \cup \left \{ v \right \}$. Since $\MCE(v,\lambda)=\left \{ \lambda \right \}$ for each $\lambda\in F$, the set $E$ is closed under minimal common extensions, contains $v\in E$ and is a subset of $v\Lambda$. Hence Lemma~\ref{le:orth2} gives
\begin{align*} t_v = \sum_{\lambda \in E} Q(t)_\lambda^E=Q(t)_v^E+\sum_{\lambda \in E \setminus \left \{ v \right \}}Q(t)_\lambda^E&=t_v\prod_{\lambda \in F} \left( t_v -t_\lambda t_\lambda^* \right) + \sum_{\lambda \in F} Q(t)_\lambda^F\\
&=t_v\prod_{\lambda \in vF} \left( t_v -t_\lambda t_\lambda^* \right) + \sum_{\lambda \in vF} Q(t)_\lambda^F.
\end{align*}

Now suppose $E\subset\Lambda$ is finite and closed under minimal common extensions. Fix $v\in r(E)$. We will show that $E$ and $v$ satisfy \eqref{eq:orth1}. Notice that for all $\lambda \in E$ with $r(\lambda)=v$,
\begin{equation}
Q(t)_\lambda^E=t_\lambda t_\lambda^* \prod_{\lambda \nu \in E \atop d(\nu)>0} \left( t_\lambda t_\lambda^*-t_{\lambda \nu} t_{\lambda \nu}^* \right)=t_\lambda t_\lambda^* \prod_{\lambda \nu \in vE \atop d(\nu)>0} \left( t_\lambda t_\lambda^*-t_{\lambda \nu} t_{\lambda \nu}^* \right)=Q(t)_\lambda^{v E} \label{eq:orth3}.
\end{equation}
Let $F:=vE$. Then $F=vE \subset v \Lambda$ is finite and if $\lambda,\mu \in F$ then $\MCE(\lambda,\mu) \subset v E=F$; that is, $F$ is closed under minimal common extensions. It follows from \eqref{eq:orth3} and the previously considered special case that
\begin{flalign*}&&t_v&= t_v \prod_{\lambda \in v F} \left ( t_v - t_\lambda t_\lambda^* \right)+\sum_{\lambda \in vF}Q(t)_\lambda^F&\\
&&&= t_v \prod_{\lambda \in v E} \left ( t_v - t_\lambda t_\lambda^* \right)+\sum_{\lambda \in vE}Q(t)_\lambda^{E}.&\qedhere
\end{flalign*}
\end{proof}

The rest of this section is devoted to establishing Lemma~\ref{le:orth2}. We therefore fix, for the duration of this section, a finitely aligned $k$-graph $(\Lambda,d)$; a 2-cocycle $c\in \underline{Z}^2(\Lambda,\mathbb{T})$; a Toeplitz-Cuntz-Krieger $(\Lambda,c)$-family $\left \{ t_\lambda : \lambda \in \Lambda \right \}$; $v\in \Lambda^0$; and a finite subset $E$ of $v\Lambda$.

\begin{definition}\label{def:closedundercommonextensions}
Given $F\subset \Lambda$, we define
\begin{equation} \MCE(F) := \left \{ \lambda \in \Lambda : d(\lambda)= \bigvee_{\alpha \in F}d(\alpha) \text{ and } \lambda(0,d(\alpha))=\alpha \text{ for all } \alpha \in F \right \},\nonumber \end{equation}
and $\bigvee F:= \bigcup_{G\subset F}\MCE(G)$.
\end{definition}
The next two Lemmas are from \cite{RS2005}.
\begin{lemma}[\cite{RS2005}, Lemma 8.4]\label{le:orth3}If $v\in E$ then:
\begin{enumerate}
\item[(1)] $E \subset \vee E$;
\item[(2)] $\vee E$ is finite;
\item[(3)] $G\subset \vee E$ implies that $\MCE(G) \subset \vee E$; and
\item[(4)] $\lambda \in \vee E$ implies that $d(\lambda) \leq \vee_{\mu\in E} d(\mu)$.
\end{enumerate}
\end{lemma}

\begin{lemma}[\cite{RS2005}, Lemma 8.7] \label{le:orth4}
Suppose that $v\in E$, $\lambda \in E \setminus \left \{ v \right \}$ and $G:=E \setminus \left \{ \lambda \right \}$. Then for each $\mu \in \vee E \setminus \vee G$, there exists a unique $\xi_\mu \in \vee G$ such that
\begin{enumerate}
\item[(1)] $d(\mu) \geq d(\xi_\mu)$ and $\mu(0,d(\xi_\mu))=\xi_\mu$; and
\item[(2)] $\xi \in \vee G$ and $\mu(0,d(\xi))=\xi$ imply $d(\xi)\leq d(\xi_\mu)$.
\end{enumerate}
Furthermore, $\mu\in \MCE(\xi_\mu,\lambda)$ for each $\mu \in \vee E$.
\end{lemma}

\begin{lemma}[\cite{RS2005}, Lemma 8.8] \label{le:orth5}
Suppose that $v\in E$, $\lambda \in E \setminus \left \{ v \right \}$ and $G:=E \setminus \left \{ \lambda \right \}$. Let $\mu \in \vee E \setminus \vee G$ and $\xi_\mu\in \vee G$ be the maximal subpath of $\mu$ provided by Lemma~\ref{le:orth4}. Then
\begin{equation}
Q(t)_\mu^{\vee E} = Q(t)_{\xi_\mu}^{\vee G}t_\mu t_\mu^*.
\end{equation}
\end{lemma}

\begin{proof}
First we show that $Q(t)_\mu^{\vee E} \leq Q(t)_{\xi_\mu}^{\vee G}.$ Since $\mu(0,d(\xi_\mu))=\xi_\mu$, we have $t_\mu t_\mu^* \leq t_{\xi_\mu}t_{\xi_\mu}^*$. Hence
$$Q(t)_{\xi_\mu}^{\vee G}Q(t)_\mu^{\vee E}=t_\mu t_\mu^* \left( \prod_{\xi_\mu \nu \in \vee G \atop d(\nu)>0} \left(t_{\xi_\mu}t_{\xi_\mu}^*-t_{\xi_\mu \nu} t_{\xi_\mu \nu}^* \right) \right) Q(t)_\mu^{\vee E}.$$
We will show that $t_\mu t_\mu^* \left(t_{\xi_\mu}t_{\xi_\mu}^*-t_{\xi_\mu \nu} t_{\xi_\mu \nu}^* \right) Q(t)_{\mu}^{\vee E}=Q(t)_\mu^{\vee E}$ whenever $\xi_\mu \in \vee G$ with $d(\nu)>0$. Fix $\xi_\mu \in \vee G$ with $d(\nu)>0$. Then
\begin{align}
t_\mu t_\mu^* \left(t_{\xi_\mu}t_{\xi_\mu}^*-t_{\xi_\mu \nu} t_{\xi_\mu \nu}^* \right)&=t_\mu t_\mu^* t_{\xi_\mu}t_{\xi_\mu}^*-t_\mu t_\mu^*t_{\xi_\mu \nu} t_{\xi_\mu \nu}^* \nonumber\\
&=t_\mu t_\mu^* -\sum_{\sigma\in \MCE(\mu,\xi_\mu \nu)} t_\sigma t_\sigma^* \qquad \text{since } t_\mu t_\mu^* \leq t_{\xi_\mu}t_{\xi_\mu}^* \nonumber\\
&=t_\mu t_\mu^* -\sum_{(\alpha,\beta) \in \Lambda^{\text{min}}(\mu,\xi_\mu \nu)} t_{\mu \alpha} t_{\mu \alpha}^* \qquad \text{ by Proposition }~\ref{prop:MCE}\nonumber\\
&=\prod_{(\alpha,\beta)\in \Lambda^{\text{min}}(\mu, \xi_\mu \nu)} \left( t_\mu t_\mu^* - t_{\mu \alpha} t_{\mu \alpha}^* \right)\qquad\text{ by Lemma }  ~\ref{handylemma}(6). \label{eq:orth4}
\end{align}

Now suppose that $(\alpha, \beta) \in \Lambda^{\text{min}}(\mu, \xi_\mu \nu)$. We claim $\mu \alpha \in \vee E$ and $d(\alpha)>0$. We have $$(\mu \alpha)(0,d(\xi_\mu \nu))=\xi_\mu \nu \in \vee G,$$ and so Lemma~\ref{le:orth4}(2) shows that $d(\xi_{\mu \alpha}) \geq d(\xi_\mu \nu ) > d(\xi_\mu).$ Hence $\mu \alpha \neq \mu$, so $d(\alpha)>0$. We have $\xi_\mu \nu \in \vee G \subset \vee E$ and $\mu \in \vee E$ by assumption, so $\mu \alpha \in \MCE(\mu , \xi_\mu \nu) \subset \vee E$ by Lemma~\ref{le:orth3}(4). Establishing our claim. Thus, each term in \eqref{eq:orth4} is a factor in $Q(t)_\mu^{\vee E}$, so
$$t_\mu t_\mu^* \left(t_{\xi_\mu}t_{\xi_\mu}^*-t_{\xi_\mu \nu} t_{\xi_\mu \nu}^* \right) Q(t)_\mu^{\vee E}=Q(t)_\mu^{\vee E}.$$

Next we show that if $\mu\nu \in \vee E$ with $d(\nu)>0$, then $Q(t)_{\xi_\mu}^{\vee G}t_{\mu \nu}t_{\mu \nu}^*=0.$ Fix $\mu \nu \in \vee E$ with $d(\nu)>0$. The last statement in Lemma~\ref{le:orth4} implies that $d(\mu \nu)=d(\xi_{\mu \nu}) \vee d(\lambda)$ and $d(\mu)=d(\xi_\mu)\vee d(\lambda)$. Since $d(\nu)>0$ we have $d(\xi_\mu)\neq d(\xi_{\mu \nu})$, and so $\xi_\mu\neq \xi_{\mu \nu}$. Since $(\mu \nu)(0,d(\xi_\mu))=\mu(0,d(\xi_\mu))=\xi_\mu \in \vee G$, Lemma~\ref{le:orth4}(2) implies that $d(\xi_{\mu \nu})>d(\xi_\mu)$. Hence $\xi_{\mu \nu}=\xi_\mu \tau$ for some $\tau$ with $d(\tau)>0$. Since $\xi_{\mu \nu} \in \vee G$, we have
\begin{align*}
Q(t)_{\xi_\mu}^{\vee G} t_{\mu \nu} t_{\mu \nu}^*&=t_{\xi_\mu} t_{\xi_\mu}^* \prod_{\xi_\mu \rho \in \vee G \atop d(\rho)>0}\left( t_{\xi_\mu}t_{\xi_\mu}^*-t_{\xi_\mu \rho}t_{\xi_\mu \rho}^* \right) t_{\mu \nu}t_{\mu \nu}^*\\
&=t_{\xi_\mu} t_{\xi_\mu}^* \prod_{\xi_\mu \rho \in \vee G \atop d(\rho)>0}\left( t_{\xi_\mu}t_{\xi_\mu}^*-t_{\xi_\mu \rho}t_{\xi_\mu \rho}^* \right)\left(t_{\xi_\mu}t_{\xi_\mu}^*-t_{\xi_\mu \tau}t_{\xi_\mu \tau}^*\right) t_{\mu \nu}t_{\mu \nu}^*\\
&=t_{\xi_\mu} t_{\xi_\mu}^* \prod_{\xi_\mu \rho \in \vee G \atop d(\rho)>0}\left( t_{\xi_\mu}t_{\xi_\mu}^*-t_{\xi_\mu \rho}t_{\xi_\mu \rho}^* \right)\left(t_{\xi_\mu}t_{\xi_\mu}^*-t_{\xi_{\mu \nu}}t_{\xi_{\mu \nu}}^*\right) t_{\mu \nu}t_{\mu \nu}^*
\end{align*}
which vanishes since $t_{\xi_\mu}t_{\xi_\mu}^*, t_{\xi_{\mu \nu}}t_{\xi_{\mu\nu}}^*\geq t_{\mu \nu} t_{\mu \nu}^*$.

Since $Q(t)_\mu^{\vee E}=Q(t)_{\xi_\mu}^{\vee G} Q(t)_\mu^{\vee E} $ and $\mu \nu \in \vee E$ with $d(\nu)>0$ implies that $Q(t)_{\xi_\mu}^{\vee G}t_{\mu \nu}t_{\mu \nu}^*=0$, we have
\begin{flalign*}&&Q(t)_\mu^{\vee E}&=Q(t)_{\xi_\mu}^{\vee G}t_\mu t_\mu^* \prod_{\mu \nu \in \vee E \atop d(\nu)>0} \left( t_\mu t_\mu^* - t_{\mu \nu}t_{\mu \nu}^*\right)&\\
&&&=Q(t)_{\xi_\mu}^{\vee G} \prod_{\mu \nu \in \vee E \atop d(\nu)>0} \left( t_\mu t_\mu^* - t_{\mu \nu}t_{\mu \nu}^*\right)&\\
&&&=Q(t)_{\xi_\mu}^{\vee G} t_\mu t_\mu^*.&\qedhere
\end{flalign*}
\end{proof}

\begin{proof}[Proof of Lemma~\ref{le:orth2}]
Fix $v\in \Lambda^0$ and finite $E\subset v\Lambda$ that is closed under minimal common extensions. Suppose that $v\in E$. As $E$ is closed under minimal common extensions, induction on $|G|$ shows that $G\subset E$ implies that $\MCE(G)\subset E$. Hence $\vee E \subset E$. Lemma~\ref{le:orth3}(1) implies that $E=\vee E$. Therefore it suffices to show that $Q(t)_\lambda^{\vee E} Q(t)_\mu^{\vee E} =\delta_{\lambda,\mu}Q(t)_\lambda^{\vee E}$ for all $\lambda,\mu \in \vee E$, and  that $t_v=\sum_{\lambda \in \vee E} Q(t)_\lambda^{\vee E}$.

Fix $\lambda,\mu \in \vee E$ with $\lambda \neq \mu$. Suppose that $d(\lambda)=d(\mu)$. We $Q(t)_\lambda^{\vee E}\leq t_\lambda t_\lambda^*$ and $Q(t)_\mu^{\vee E}\leq t_\mu t_\mu^*$. Lemma~\ref{handylemma}(2) implies that $Q(t)_\lambda^{\vee E}Q(t)_\mu^{\vee E}\leq t_\lambda t_\lambda^*t_\mu t_\mu^*$. Lemma~\ref{handylemma}(6) implies that $t_\lambda t_\lambda^*t_\mu t_\mu^*=0$, giving $Q(t)_\lambda^{\vee E}Q(t)_\mu^{\vee E}=0$.

Now suppose that $d(\lambda)\neq d(\mu)$. Then $d(\lambda)\vee d(\mu)$ is strictly larger than one of $d(\lambda)$ or $d(\mu)$; say $d(\lambda)\vee d(\mu) > d(\lambda)$. Hence $(\alpha,\beta)\in \Lambda^{\text{min}}(\lambda,\mu)$ implies that $d(\alpha)>0$ and $\lambda \alpha \in \vee E$. Then
\begin{align*}
Q(t)_\lambda^{\vee E}& Q(t)_\mu^{\vee E}\\
&=t_\lambda t_\lambda^* t_\mu t_\mu^*Q(t)_\lambda^{\vee E} Q(t)_\mu^{\vee E}\\
&=\left( \sum_{(\alpha,\beta)\in \Lambda^{\text{min}}(\lambda,\mu)} t_{\lambda \alpha} t_{\lambda\alpha}^* \right) \left( \prod_{\lambda \nu \in \vee E \atop d(\nu)>0} \left( t_\lambda t_\lambda^*-t_{\lambda \nu}t_{\lambda \nu}^* \right ) \right) Q(t)_\mu^{\vee E}\\
&=\sum_{(\alpha,\beta)\in \Lambda^{\text{min}}(\lambda,\mu)}\left( t_{\lambda \alpha} t_{\lambda \alpha}^*\left(\left(t_\lambda t_\lambda^*-t_{\lambda \alpha}t_{\lambda \alpha}^* \right) \prod_{\lambda \nu \in \vee E \atop d(\nu)>0} \left( t_\lambda t_\lambda^*-t_{\lambda \nu}t_{\lambda \nu}^* \right) \right) \right) Q(t)_\mu^{\vee E}\\
&=0.
\end{align*}

It remains to show that $\sum_{\lambda \in \vee E} Q(t)_\lambda^{\vee E} = t_v$. We proceed by induction on $|E|$. Suppose that $|E|=1$. Since $v\in E$, we have $\vee E=E=\left\{ v \right \}$. We have $\sum_{\lambda \in \vee E} Q(t)_\lambda^{\vee E} = Q(t)_v^{\left\{ v \right \}}=t_v$, giving the base case.

Now suppose that $|E|=n\geq 2$ and that the result holds whenever $|E|\leq n-1.$ Since $|E|>1$, there exists $\lambda \in E \setminus \left\{ v \right \}$. Set $G:= E\setminus \left\{ \lambda \right \}$. Fix $\mu \in \vee G$. We have
\begin{align}
Q(t)_\mu^{\vee E} &= t_\mu t_\mu^* \left( \prod_{\mu \nu \in \vee E \atop d(\nu)>0} \left( t_\mu t_\mu^*-t_{\mu \nu }t_{\mu \nu}^* \right) \right)\nonumber\\
&=t_\mu t_\mu^* \left( \prod_{\mu \nu \in \vee G \atop d(\nu)>0} \left( t_\mu t_\mu^* -t_{\mu\nu}t_{\mu \nu}^* \right)\right) \left( \prod_{\mu \sigma\in \vee E \setminus \vee G}\left( t_\mu t_\mu^*-t_{\mu \sigma } t_{\mu \sigma}^* \right)\right)\label{eq:ximusigma}.
\end{align}
Suppose that $\mu \sigma \in \vee E \setminus \vee G$ and $\xi_{\mu \sigma}\neq \mu$. Lemma~\ref{le:orth4}(2) ensures that $\xi_{\mu \sigma}=\mu \alpha$ for some $\alpha$ with $d(\alpha)>0$. Since $\mu \alpha =\xi_{\mu \sigma}\in \vee G$, the projection $(t_\mu t_\mu^*-t_{\xi_{\mu \sigma}} t_{\xi_{\mu \sigma}}^*)$ is a factor in $Q(t)_\mu^{\vee G}$. Since $t_\mu t_\mu^*-t_{\xi_{\mu \sigma}}t_{\xi_{\mu\sigma}}^*\leq t_\mu t_\mu^*-t_{\mu \sigma}t_{\mu \sigma}^*$, we have $(t_\mu t_\mu^*-t_{\mu \sigma}t_{\mu \sigma}^*)Q(t)_\mu^{\vee G}=Q(t)_\mu^{\vee G}.$ Equation~\eqref{eq:ximusigma} then implies that
\begin{equation}
Q(t)_\mu^{\vee E}=Q(t)_\mu^{\vee G} \left( \prod_{\mu \sigma\in \vee E \setminus \vee G \atop \xi_{\mu\sigma}=\mu}\left( t_\mu t_\mu^*-t_{\mu \sigma } t_{\mu \sigma}^* \right)\right).
\label{eq:veeEveeGrelation}\end{equation}
Lemma~\ref{le:orth4} ensures that if $\mu \sigma \in \vee E \setminus \vee G$, then $\mu \sigma \in \MCE(\xi_{\mu \sigma},\lambda)$; in particular, when $\xi_{\mu\sigma}=\mu$, we have $d(\mu\sigma)=d(\mu) \vee d(\sigma)$. Hence if $\mu \sigma ,\mu \sigma' \in \vee E \setminus \vee G$ satisfy $\xi_{\mu \sigma}=\mu =\xi_{\mu \sigma'}$, then $d(\mu \sigma)=d(\mu \sigma')$. Lemma~\ref{handylemma}(6) implies that $t_{\mu \sigma}t_{\mu \sigma}^* t_{\mu\sigma'}t_{\mu \sigma'}^*=\delta_{\sigma, \sigma'} t_{\mu \sigma} t_{\mu \sigma}^*$. So
$$\prod_{\mu \sigma \in \vee E \setminus \vee G \atop \xi_{\mu \sigma}=\mu}\left( t_\mu t_\mu^*-t_{\mu \sigma} t_{\mu\sigma}^* \right)=t_\mu t_\mu^*-\sum_{\mu \sigma \in \vee E \setminus \vee G \atop \xi_{\mu \sigma}=\mu } t_{\mu \sigma} t_{\mu \sigma}^*.$$
Equation~\eqref{eq:veeEveeGrelation} implies that for all $\mu \in \vee G$, we have
\begin{equation}
Q(t)_\mu^{\vee E}=Q(t)_\mu^{\vee G} \left( t_\mu t_\mu^* - \sum_{\mu \sigma \in \vee E \setminus \vee G \atop \xi_{\mu \sigma}=\mu} t_{\mu \sigma} t_{\mu \sigma}^* \right). \label{eq:veeEveeGrelation2}
\end{equation}
Substituting~\eqref{eq:veeEveeGrelation2} for those terms in $\sum_{\sigma\in \vee E} Q(t)_{\sigma}^{\vee E}$ for which $\sigma $ belongs to $\vee G$, we obtain
\begin{align}
\sum_{\sigma\in \vee E} Q(t)_{\sigma}^{\vee E}&=\sum_{\mu\in \vee G} Q(t)_{\mu}^{\vee E}+\sum_{\tau\in \vee E\setminus \vee G} Q(t)_{\tau}^{\vee E}\nonumber\\
&=\sum_{\mu \in \vee G}\left( Q(t)_\mu^{\vee G}\left( t_\mu t_\mu^*-\sum_{\mu \sigma\in \vee E \setminus \vee G \atop \xi_{\mu \sigma}=\mu} t_{\mu \sigma}t_{\mu \sigma}^* \right) \right)+\sum_{\tau\in \vee E\setminus \vee G} Q(t)_{\tau}^{\vee E} \label{eq:veeEveeGrelation3}.
\end{align}
Lemma~\ref{le:orth4} ensures that for $\tau \in \vee E \setminus \vee G$, the path $\xi_\tau\in \vee G$ is uniquely determined by $\tau$. Hence we can rewrite ~\eqref{eq:veeEveeGrelation3} as
$$\sum_{\sigma\in \vee E} Q(t)_{\sigma}^{\vee E}=\sum_{\mu \in \vee G}\left( Q(t)_\mu^{\vee G}\left( t_\mu t_\mu^*-\sum_{\mu \sigma\in \vee E \setminus \vee G \atop \xi_{\mu \sigma}=\mu} t_{\mu \sigma}t_{\mu \sigma}^* \right)+\sum_{\tau\in \vee E\setminus \vee G \atop \xi_\tau=\mu} Q(t)_{\tau}^{\vee E} \right).$$
Lemma~\ref{le:orth5} allows us to replace each $Q(t)_\tau^{\vee E}$ with $Q(t)_{\xi_\tau}^{\vee G}t_\tau t_\tau^*$, yielding
\begin{align*}
\sum_{\sigma\in \vee E} Q(t)_{\sigma}^{\vee E}&=\sum_{\mu \in \vee G}\left( Q(t)_\mu^{\vee G}\left( t_\mu t_\mu^*-\sum_{\mu \sigma\in \vee E \setminus \vee G \atop \xi_{\mu \sigma}=\mu} t_{\mu \sigma}t_{\mu \sigma}^* \right)+\sum_{\tau\in \vee E\setminus \vee G \atop \xi_\tau=\mu} Q(t)_{\xi_{\tau}}^{\vee E} t_\tau t_\tau^*\right)\\
&=\sum_{\mu \in \vee G}\left( Q(t)_\mu^{\vee G}\left( t_\mu t_\mu^*-\sum_{\mu \sigma\in \vee E \setminus \vee G \atop \xi_{\mu \sigma}=\mu} t_{\mu \sigma}t_{\mu \sigma}^* \right)+\sum_{\tau\in \vee E\setminus \vee G \atop \xi_\tau=\mu} Q(t)_{\mu}^{\vee G} t_\tau t_\tau^*\right)\\
&=\sum_{\mu \in \vee G} Q(t)_\mu^{\vee G} \left(\left( t_\mu t_\mu^*-\sum_{\mu \sigma\in \vee E \setminus \vee G \atop \xi_{\mu \sigma}=\mu} t_{\mu \sigma}t_{\mu \sigma}^* \right)+\sum_{\tau\in \vee E\setminus \vee G \atop \xi_\tau=\mu} t_\tau t_\tau^*\right)\\
&=\sum_{\mu \in \vee G} Q(t)_\mu^{\vee G}.
\end{align*}
By the inductive hypothesis, we have $\sum_{\mu \in \vee G} Q(t)_\mu^{\vee G}=t_v$.
\end{proof}

\begin{corollary}[\cite{RSY2004}, Corollary 3.7] \label{co:orth5}
If $E$ is closed under minimal common extensions then
\begin{equation}
t_\mu t_\mu^* = \sum_{\mu \nu \in E} Q(t)_{\mu \nu}^E \nonumber
\end{equation}
\end{corollary}

\begin{proof}
Let $v:=r(\mu)$. By Proposition~\ref{prop:orth1}, we have
$$t_\mu t_\mu^*=t_\mu t_\mu^*\left( \prod_{\lambda \in v E} (t_v-t_\lambda t_\lambda^*)+\sum_{\lambda \in v E} Q(t)_\lambda^E \right).$$
Since $\mu \nu \in E$ implies that $t_\mu t_\mu^* \geq Q(t)_{\mu \nu}^E$, we need only show that:
\begin{enumerate}
\item[(1)] $t_\mu t_\mu^* \prod_{\lambda \in v E} (t_v-t_\lambda t_\lambda^*)=0$; and
\item[(2)] $t_\mu t_\mu^* Q(t)_\lambda^E=0$ for all $ \lambda \in E \setminus \mu \Lambda$.
\end{enumerate}
Since $\mu \in v E$, we have
$$t_\mu t_\mu^* \prod_{\lambda \in v E} (t_v-t_\lambda t_\lambda^*)=t_\mu t_\mu^* (t_v -t_\mu t_\mu^*)\prod_{\lambda \in v E} (t_v-t_\lambda t_\lambda^*)=0.$$
This establishes (1).

For (2), fix $\sigma \in E \setminus \mu E$. If $\MCE(\mu,\sigma)=\emptyset$, then $t_\mu t_\mu^*Q(t)_\sigma^E \leq t_\mu t_\mu^* t_\sigma t_\sigma^*=0$, by (TCK4). If $\MCE(\mu,\sigma)\neq \emptyset$, then $r(\sigma)=r(\mu)=v$, and
\begin{align}
t_\mu t_\mu^* Q(t)_\sigma^{E}&=t_\mu t_\mu^* t_\sigma t_\sigma^*Q(t)_\sigma^{E}\nonumber\\
&=\sum_{(\alpha,\beta)\in \Lambda^{\text{min}}}\left( t_{\sigma \beta} t_{\sigma \beta}^* \prod_{\sigma \nu \in E \atop d(\nu)>0} ( t_\sigma t_\sigma^*-t_{\sigma \nu } t_{\sigma \nu}^*) \right).\label{eq:coroeq1}
\end{align}
Fix $(\alpha,\beta)\in \Lambda^{\text{min}}(\mu,\sigma)$. Since $E$ is closed under minimal common extensions, we have $\sigma \beta \in E$ and, since $\sigma \notin \mu \Lambda$, $d(\beta)>0$. Hence
$$t_{\sigma \beta} t_{\sigma \beta}^* \prod_{\sigma \nu \in E \atop d(\nu)>0} ( t_\sigma t_\sigma^*-t_{\sigma \nu } t_{\sigma \nu}^*) =t_{\sigma \beta} t_{\sigma \beta}^* (t_\sigma t_\sigma^* - t_{\sigma \beta } t_{\sigma \beta}^*)\prod_{\sigma \nu \in E \atop d(\nu)>0} ( t_\sigma t_\sigma^*-t_{\sigma \nu } t_{\sigma \nu}^*)=0.$$
Applying this to each term in~\eqref{eq:coroeq1} gives (2).
\end{proof}

\section{Matrix units for the core}
Let $\left\{ t_\lambda :\lambda \in \Lambda \right \}$ be a relative Cuntz-Krieger $(\Lambda,c;\mathcal{E})$-family. For each finite $E\subset \Lambda$, we identify a finite-dimensional subalgebra $M^t_{\Pi E}$ of $C^*\left( \left\{ t_\lambda t_\mu^* : \lambda ,\mu \in \Lambda, d(\lambda)=d(\mu) \right \} \right).$ We describe matrix units for $M^t_{\Pi E}$.
\label{chptr:matrix_units}

Lemma~\ref{le:matrixalgebra} and Corollary~\ref{co:matrixalgebrauni} can be deduced from Appendix A of \cite{R2005}.

\begin{lemma}\label{le:matrixalgebra}
Let $A$ be a $C^*$-algebra. Let $I$ a finite set. Suppose that $E:=\left \{ e_{i,j} : i,j\in I \right \}\subset A$ satisfies:
\begin{enumerate}
\item[(M1)] $e_{i,j}^*=e_{j,i}$ for all $i,j \in I $; and
\item[(M2)] $e_{i,j}e_{k,l}=\delta_{j,k}e_{i,l}$ for all $i,j,k,l\in I$.
\end{enumerate}
Suppose that $e_{i,j}\neq 0$ for each $i,j\in I$. Then $\linspan{E}$ is a $C^*$-subalgebra of $A$ that is universal in the sense that if $B$ is a $C^*$-algebra and $F=\left \{ f_{i,j} : i,j\in I \right \}\subset B$ satisfies $(M1)$ and $(M2)$ then there exists a homomorphism $\phi: \linspan E \rightarrow B$ such that $\phi(e_{i,j})=f_{i,j}$.
\end{lemma}

If $\left \{ e_{i,j} : i,j\in I \right \}$ satisfy (M1) and (M2), we call the $e_{i,j}$ a matrix units.

\begin{corollary}\label{co:matrixalgebrauni}
Let $A$ and $B$ be $C^*$-algebras. Let $I$ be a finite set. Suppose that $E:=\left \{ e_{i,j}: i,j \in I \right \}\subset A$ and $ F:=\left \{ f_{i,j}: i,j \in I \right \}\subset B$ both satisfy $(M1)$ and $(M2)$. If $e_{i,j}\neq 0$ and $f_{i,j}\neq 0$ for all $i,j\in I$, then there is an isomorphism $$\linspan\left \{ e_{i,j} :i,j\in I\right \} \rightarrow \linspan \left \{ f_{i,j} : i,j\in I\right \}$$ taking $e_{i,j}$ to $f_{i,j}$ for each $i,j\in I$.
\end{corollary}

Let $I$ be a finite set. Let $M_I(\mathbb{C}):= \left \{ f: I \times I \rightarrow \mathbb{C} \right \}$. For $f,g\in M_I(\mathbb{C})$, define $f^*,fg\in M_I(\mathbb{C})$ by $f^*(i,j)=\overline{f(j,i)}$ and $(fg)(i,j) =\sum_{k\in I}f(i,k)g(k,j)$. These operations mimic the usual adjoint and multiplication of matrices. Each $f\in M_I(\mathbb{C})$ determines a linear map $T_f : \mathbb{C}^I \rightarrow \mathbb{C}^I$ by
$$(T_fv)(i) =\sum_{j\in I} f(i,j)v(j).$$
Define $\Vert f \Vert$ to be the operator norm of $T_f$ with respect to the $\ell^2$ norm on $\mathbb{C}^I$. With this norm $M_I(\mathbb{C})$ becomes a $C^*$-algebra. For each $i,j\in I$, define $\xi_{i,j} \in M_I(\mathbb{C})$ by $\xi_{i,j}(k,l)=\delta_{i,k}\delta_{j,l}$ for all $k,l\in I$, where $\delta$ denotes the Kronecker delta. Then $\left\{ \xi_{i,j} : i,j\in I \right \}$ is a collection of nonzero matrix units and $M_I(\mathbb{C})=\linspan\left\{ \xi_{i,j}:i,j\in I \right \}$. Corollary~\ref{co:matrixalgebrauni} implies that if $\left \{ e_{i,j}: i,j \in I \right \}\subset A$ is a collection of nonzero matrix units in a $C^*$-algebra $A$, then $M_I(\mathbb{C}) \cong \linspan \left\{ e_{i,j}: i,j\in I \right \}$.
%

\begin{definition}
Let $(\Lambda,d)$ be a finitely aligned $k$-graph. For a subset $E$ of $\Lambda$ we define
$$E\times_{d,s} E:= \left \{ (\lambda,\mu)\in E \times E : d(\lambda)=d(\mu) \text{ and } s(\lambda)= s(\mu) \right \}.$$
\end{definition}

Let $(\Lambda,d)$ be a finitely aligned $k$-graph, let $c\in \underline{Z}^2(\Lambda,\mathbb{T})$ and let $\left\{ t_\lambda: \lambda \in \Lambda \right \}$ be a Toeplitz-Cuntz-Krieger $(\Lambda,c)$-family. Given a finite $E\subset \Lambda$, we show there is a finite $\Pi E \subset \Lambda$ such that
$$\linspan\left\{ t_\lambda t_\mu^* : (\lambda, \mu) \in \Pi E\times_{d,s} \Pi E\right \}$$
is closed under multiplication, and hence is a finite-dimensional $C^*$-subalgebra of $$C^*\left( \left\{ t_\lambda : \lambda \in \Lambda \right \} \right).$$

\begin{lemma}[\cite{RSY2004}, Lemma 3.2]\label{le:piEsets} Let $(\Lambda,d)$ be a finitely aligned $k$-graph, let $c\in \underline{Z}^2(\Lambda,\mathbb{T})$ and let $\left\{ t_\lambda: \lambda \in \Lambda \right \}$ be a Toeplitz-Cuntz-Krieger $(\Lambda,c)$-family. Let $E \subset \Lambda$ be finite.
There exists a finite set $F\subset \Lambda$ that contains $E$ and satisfies \begin{equation} (\lambda,\mu),(\sigma,\tau) \in F \times_{d,s} F \Longrightarrow \left \{ \lambda\alpha,\tau\beta: (\alpha, \beta)\in \Lambda^{\text{min}}(\mu,\sigma) \right \} \subset F.\label{eq:piE}\end{equation}
Moreover, for any finite $F\subset \Lambda$ that contains $E$ and satisfies \eqref{eq:piE},
$$M_F^t := \linspan \left \{ t_\lambda t_\mu^* : (\lambda,\mu) \in F\times_{d,s} F \right \}$$
is a finite-dimensional $C^*$-subalgebra of $C^* \left ( \left \{ t_\lambda t_\mu^* : \lambda,\mu \in \Lambda \right \} \right)$.
\end{lemma}

\begin{proof}
The first statement is the first statement of Lemma 3.2 in \cite{RSY2004}. For the second statement, suppose that $F\subset \Lambda$ is finite contains $E$ and satisfies \eqref{eq:piE}. Then $M_F^t$ is a finite-dimensional, and hence norm-closed, subspace of $C^*\left( \left\{ t_\lambda : \lambda \in \Lambda \right \} \right)$, which is closed under adjoints. It remains to show that $M_F^t$ is closed under multiplication. Let $t_\lambda t_\mu^*$ and $t_\sigma t_\tau^*$ be generators of $M_F^t$. Then $d(\lambda)=d(\mu),  d(\sigma)=d(\tau),  s(\lambda)=s(\mu)$ and $s(\sigma)=s(\tau)$. By Lemma~\ref{handylemma} part (4) we have
\begin{equation} t_\lambda t_\mu^* t_\sigma t_\tau^* = \sum_{(\alpha,\beta)\in \Lambda^{\text{min}}(\mu,\sigma)}c(\sigma,\beta)c(\lambda,\alpha)\overline{c(\mu,\alpha)c(\tau,\beta)}t_{\lambda \alpha}t_{\tau \beta}^*. \nonumber \end{equation}

 As $F$ satisfies \eqref{eq:piE} we have $\lambda \alpha, \tau \beta \in F$ for each $(\alpha,\beta)\in \Lambda^{\text{min}}(\mu,\sigma) $ and hence $ t_\lambda t_\mu^* t_\sigma t_\tau^* \in M_F^t$. The result then follows since multiplication is bilinear.
\end{proof}

The intersection of all sets containing $E$ and satisfying \eqref{eq:piE} also contains $E$ and satisfies \eqref{eq:piE}, and so we make the following definition.

\begin{definition}
Let $(\Lambda,d)$ be a finitely aligned $k$-graph. Let $E \subset \Lambda$ be finite. Define
\begin{equation}
\Pi E := \bigcap \left \{ F \subset \Lambda : E \subset F \text{ and } F \text{ satisfies } \eqref{eq:piE} \right \} \nonumber.
\end{equation}
\end{definition}

The following properties of $\Pi E$ will be important.
\begin{lemma}[\cite{RSY2004}, Remark 3.4]\label{re:matrixunits0}
Let $(\Lambda,d)$ be a finitely aligned $k$-graph. Let $E\subset \Lambda$ be finite. Then
\begin{enumerate}
\item[(1)] $\Pi E$ is finite;
\item[(2)] for $\rho,\zeta \in \Pi E$ with $d(\rho)=d(\zeta)$ and $s(\rho)=s(\zeta)$, and for all $\nu \in s(\rho) \Lambda$,
$$\rho \nu \in \Pi E \text{ if and only if } \zeta \nu \in \Pi E; \text{ and }$$
\item[(3)] $\Pi E$ is closed under minimal common extensions.
\end{enumerate}
\end{lemma}

\begin{proof}
For (1), the set $\Pi E$ is an intersection of finite sets and hence is finite itself. For (2), the ``if" direction follows from \eqref{eq:piE} with $\lambda=\rho$, $\mu=\zeta$, and $\sigma=\tau=\zeta\nu$. The ``only if" directions follows from \eqref{eq:piE} with $\lambda=\mu=\rho \nu$, $\sigma=\rho$, and $\tau=\zeta$. For (3), suppose that $\rho,\zeta\in \Pi E$ and that $(\alpha,\beta)\in \Lambda^{\text{min}}(\rho,\zeta)$. Then \eqref{eq:piE} with $\lambda=\mu=\rho$ and $\sigma=\tau=\zeta$ gives $\rho \alpha=\zeta \beta\in \Pi E$; that is, $\Pi E$ is closed under minimal common extensions.
\end{proof}

\begin{definition}
Let $(\Lambda,d)$ be a finitely aligned $k$-graph, let $c\in \underline{Z}^2(\Lambda,\mathbb{T})$ and let $\left\{ t_\lambda : \lambda  \in \Lambda \right \}$ be a Toeplitz-Cuntz-Krieger $(\Lambda,c)$-family. Let $E$ be a finite subset of $\Lambda$. For $(\lambda,\mu) \in \Pi E \times_{d,s} \Pi E$, define $$\Theta(t)_{\lambda,\mu}^{\Pi E}:= Q(t)_{\lambda}^{\Pi E} t_\lambda t_\mu^*.$$
\end{definition}

\begin{proposition}[cf. \cite{RSY2004}, Proposition 3.9]\label{prop:matrixunits2}
Let $(\Lambda,d)$ be a finitely aligned $k$-graph, let $c\in \underline{Z}^2(\Lambda,\mathbb{T})$ and let $\left\{ t_\lambda : \lambda  \in \Lambda \right \}$ be a Toeplitz-Cuntz-Krieger $(\Lambda,c)$-family. Let $E$ be a finite subset of $\Lambda$. The set
$$\left \{ \Theta(t)_{\lambda,\mu}^{\Pi E}: (\lambda , \mu) \in \Pi E \times_{d,s} \Pi E \right \}$$
is a collection of matrix units for $M_{\Pi E}^t$.
\end{proposition}

Before we prove Proposition~\ref{prop:matrixunits2} we need the following Lemma.

\begin{lemma}[cf. \cite{RSY2004}, Lemma 3.10 and 3.11]\label{le:matrixunits3}
Let $(\Lambda,d)$ be a finitely aligned $k$-graph, let $c\in \underline{Z}^2(\Lambda,\mathbb{T})$ and let $\left\{ t_\lambda : \lambda  \in \Lambda \right \}$ be a Toeplitz-Cuntz-Krieger $(\Lambda,c)$-family. Let $E$ be a finite subset of $\Lambda$. If $(\lambda,\mu)\in \Pi E \times_{d,s} \Pi E$ then:
\begin{align}
\Theta(t)_{\lambda,\mu}^{\Pi E} &=t_\lambda \left( \prod_{\lambda \nu \in \Pi E \atop d(\nu)>0} \left ( t_{s(\lambda)}-t_\nu t_\nu^* \right)\right) t_\mu^* = t_\lambda t_\mu^* Q(t)_{\mu}^{\Pi E}; \text{ and }\label{eq:matrixunits1} \\
t_\lambda t_\mu^* &= \sum_{\lambda \nu \in \Pi E} c(\lambda,\nu) \overline{c(\mu,\nu)}\Theta(t)_{\lambda \nu, \mu \nu}^{\Pi E}\label{eq:matrixunits2} .
\end{align}
\end{lemma}

\begin{proof}
For equation \eqref{eq:matrixunits1} we calculate
\begin{align*}
\Theta(t)_{\lambda,\mu}^{\Pi E}&=Q(t)_{\lambda}^{\Pi E} t_\lambda t_\mu^*\\
&=t_\lambda t_\lambda^* \left( \prod_{\lambda \nu \in \Pi E \atop d(\nu)>0 } \left( t_\lambda t_\lambda^* -t_{\lambda \nu } t_{\lambda \nu}^* \right) \right) t_\lambda t_\mu^*\\
&=t_\lambda t_\lambda^* \left( \prod_{\lambda \nu \in \Pi E \atop d(\nu)>0 } \left( c(\lambda,s(\lambda)) t_\lambda t_{s(\lambda)} t_\lambda^* -c(\lambda,\nu)\overline{c(\lambda,\nu)} t_{\lambda}t_{\nu }  t_{\nu }^*t_{\lambda}^* \right) \right) t_\lambda t_\mu^*\\
&=t_\lambda t_\lambda^* \left( \prod_{\lambda \nu \in \Pi E \atop d(\nu)>0 } \left(  t_\lambda \left(t_{s(\lambda)}-t_{\nu }  t_{\nu }^* \right)t_{\lambda}^* \right) \right) t_\lambda t_\mu^*\\
&=t_\lambda \left( \prod_{\lambda \nu \in \Pi E \atop d(\nu)>0} \left ( t_{s(\lambda)}-t_\nu t_\nu^* \right)\right) t_\mu^* \qquad \text{since }t_\lambda^* t_\lambda = t_{s(\lambda)}.
\end{align*}
This establishes the first equality. For the second equality in equation \eqref{eq:matrixunits1}, we continue to calculate
\begin{align*}
\Theta(t)_{\lambda,\mu}^{\Pi E}&=t_\lambda \left( \prod_{\mu \nu \in \Pi E \atop d(\nu)>0} \left ( t_{s(\lambda)}-t_\nu t_\nu^* \right)\right) t_\mu^* \qquad \text{by Lemma }\ref{re:matrixunits0}(2)\\
&=\left(t_\mu \left( \prod_{\mu \nu \in \Pi E \atop d(\nu)>0} \left ( t_{s(\mu)}-t_\nu t_\nu^* \right)\right) t_\lambda^*\right)^*\\
&=\left(Q(t)_\mu^{\Pi E} t_\mu t_\lambda^* \right)^* \qquad \text{ by the first equality}\\
&= t_\lambda t_\mu^* Q(t)_{\mu}^{\Pi E}.
\end{align*}
For equation \eqref{eq:matrixunits2} we calculate
\begin{align*}
t_\lambda t_\mu^*&=t_\lambda t_\mu^*t_\mu t_\mu^*\\
&=t_\lambda t_\mu^*\left( \sum_{\mu \nu \in \Pi E} Q(t)_{\mu \nu}^{\Pi E} \right) \qquad \text{by Corollary~\ref{co:orth5} and Lemma~\ref{re:matrixunits0}(4)}.
\end{align*}
For each $\zeta \in \Pi E$ we have $\Theta(t)_{\zeta, \zeta}^{\Pi E}=Q(t)_{\zeta}^{\Pi E}t_{\zeta}t_{\zeta}^*=Q(t)_{\zeta}^{\Pi E}$, since $Q(t)_{\zeta}^{\Pi E} \leq t_{\zeta}t_{\zeta}^*$. So
\begin{flalign*}
&&t_\lambda t_\mu^*&=\sum_{\mu \nu \in \Pi E} t_\lambda t_\mu^* \Theta(t)_{\mu \nu,\mu \nu }^{\Pi E} &\\
&&&=\sum_{\mu \nu \in \Pi E} t_\lambda t_\mu^* t_{\mu \nu} \left( \prod_{\mu \nu \alpha \in \Pi E \atop d(\alpha)>0}(t_{s(\mu \nu)}-t_{\alpha}t_{\alpha}^*) \right) t_{\mu \nu}^*\quad \text{by } \eqref{eq:matrixunits1}& \\
&&&=\sum_{\mu \nu \in \Pi E} \overline{c(\mu,\nu)} t_\lambda t_{s(\mu)} t_{\nu} \left( \prod_{\mu \nu \alpha \in \Pi E \atop d(\alpha)>0}(t_{s(\mu \nu)}-t_{\alpha}t_{\alpha}^*) \right) t_{\mu \nu}^*&\\
&&&=\sum_{\mu \nu \in \Pi E} c(\lambda,s(\mu))\overline{c(\mu,\nu)}t_{\lambda s(\mu)} t_{\nu} \left( \prod_{\mu \nu \alpha \in \Pi E \atop d(\alpha)>0}(t_{s(\mu \nu)}-t_{\alpha}t_{\alpha}^*) \right) t_{\mu \nu}^*&\\
&&&=\sum_{\mu \nu \in \Pi E}c(\lambda,\nu) \overline{c(\mu,\nu)}t_{\lambda \nu} \left( \prod_{\mu \nu \alpha \in \Pi E \atop d(\alpha)>0}(t_{s(\mu \nu)}-t_{\alpha}t_{\alpha}^*) \right) t_{\mu \nu}^*&\\
&&&=\sum_{\mu \nu \in \Pi E}c(\lambda,\nu) \overline{c(\mu,\nu)}\Theta(t)_{\lambda \nu, \mu \nu}^{\Pi E} \quad \text{by } \eqref{eq:matrixunits1}.&\qedhere
\end{flalign*}
\end{proof}

\begin{proof}[Proof of Proposition~\ref{prop:matrixunits2}]
By equation \eqref{eq:matrixunits2} the $\Theta(t)_{\lambda,\mu}^{\Pi E}$ span $M_{\Pi E}^t$. Fix $(\lambda, \mu), (\sigma, \tau) \in \Pi E \times_{d,s} \Pi E$. We need to show that
$$\left( \Theta(t)_{\lambda,\mu}^{\Pi E} \right )^* =\Theta(t)_{\mu,\lambda}^{\Pi E} $$ and
$$\Theta(t)_{\lambda,\mu}^{\Pi E} \Theta(t)_{\sigma,\tau}^{\Pi E}=\delta_{\mu,\sigma}\Theta(t)_{\lambda,\tau}^{\Pi E}.$$
Since $Q(t)_{\lambda}^{\Pi E}$ is a projection we have
\begin{align*}
\left( \Theta(t)_{\lambda,\mu}^{\Pi E} \right )^*=\left(  Q(t)_{\lambda}^{\Pi E} t_\lambda t_\mu^* \right)^*=t_\mu t_\lambda^* Q(t)_{\lambda}^{\Pi E}=\Theta(t)_{\mu,\lambda}^{\Pi E} \qquad \text{ by } \eqref{eq:matrixunits1}.
\end{align*}
\begin{flalign*}
&&\Theta(t)_{\lambda,\mu}^{\Pi E} \Theta(t)_{\sigma,\tau}^{\Pi E}&=t_\lambda t_\mu^* Q(t)_\mu^{\Pi E} Q(t)_\sigma^{\Pi E} t_\sigma t_\tau^* \qquad \text{by } \eqref{eq:matrixunits1}&\\
&&&=\delta_{\mu,\sigma} t_\lambda t_\mu^* Q(t)_\mu^{\Pi E} t_\sigma t_\tau^* \qquad \text{by Proposition }~\ref{prop:orth1}&\\
&&&=\delta_{\mu,\sigma} Q(t)_{\lambda}^{\Pi E} t_\lambda t_\mu^* t_\mu t_\tau^* \qquad \text{by } \eqref{eq:matrixunits1}&\\
&&&=c(\lambda,s(\mu))\delta_{\mu,\sigma} Q(t)_{\lambda}^{\Pi E} t_{\lambda s(\mu)} t_\tau^* &\\
&&&=\delta_{\mu,\sigma} Q(t)_{\lambda}^{\Pi E} t_{\lambda} t_\tau^*&\\
&&&=\delta_{\mu,\sigma}\Theta(t)_{\lambda,\tau}^{\Pi E}.&\qedhere
\end{flalign*}
\end{proof}

\begin{proposition}[cf. \cite{RSY2004}, Proposition 3.13]\label{prop:matrixunits4}
Let $(\Lambda,d)$ be a finitely aligned $k$-graph, let $c\in \underline{Z}^2(\Lambda,\mathbb{T})$ and let $\left\{ t_\lambda : \lambda  \in \Lambda \right \}$ be a Toeplitz-Cuntz-Krieger $(\Lambda,c)$-family. Let $E\subset \Lambda$ be finite. Suppose that $t_v \neq 0$ for every $v\in \Lambda^0$ and suppose that $(\lambda,\mu) \in \Pi E \times_{d,s} \Pi E$. If $T(\lambda):=\left \{ v \in s(\lambda) \Lambda \setminus \Lambda^0 : \lambda \nu \in \Pi E \right \}$ is not exhaustive, then $\Theta(t)_{\lambda,\mu}^{\Pi E} \neq 0$. Otherwise, if $T(\lambda)$ is exhaustive, then $\Theta(t)_{\lambda,\mu}^{\Pi E}=0$ if and only if
$$\prod_{\lambda \nu \in \Pi E, d(\nu)>0 } \left( t_{s(\lambda)}-t_\nu t_\nu^* \right ) =0.$$
\end{proposition}

\begin{proof}
First suppose that $t_v \neq 0$ for every $v\in \Lambda^0$, $(\lambda,\mu) \in \Pi E \times_{d,s} \Pi E$ and that $T(\lambda)$ is not exhaustive. There exists $\xi\in s(\lambda) \Lambda$ such that $\Lambda^{\text{min}}(\xi,\nu)= \emptyset$ for all $\nu \in T(\lambda)$. So for each $\nu \in T(\lambda)$ we have,
\begin{align*}
t_{\lambda \xi}t_{\lambda \xi}^* \left( t_\lambda t_\lambda^*-t_{\lambda \nu }t_{\lambda \nu}^* \right)&=t_{\lambda \xi}t_{\lambda \xi}^*  t_\lambda t_\lambda^*-t_{\lambda \xi}t_{\lambda \xi}^* t_{\lambda \nu }t_{\lambda \nu}^*\\
&=c(\lambda, \xi) t_{\lambda \xi}t_{\xi}^*t_\lambda^*  t_\lambda t_\lambda^*-\sum_{\sigma\in \MCE(\lambda \xi , \lambda \nu)} t_\sigma t_\sigma^* \\
&=c(\lambda, \xi) t_{\lambda \xi}t_{\xi}^*t_{s(\lambda)}^* t_\lambda^*\qquad \text{  since } \Lambda^{\text{min}}(\lambda  \xi, \lambda \nu)=\Lambda^{\text{min}}(  \xi,  \nu)=\emptyset \\
&=c(\lambda, \xi)\overline{c(s(\lambda),\xi)} t_{\lambda \xi}t_{s(\lambda)\xi}^* t_\lambda^* \\
&= t_{\lambda \xi}t_{\lambda \xi}^* .
\end{align*}

Hence
\begin{align*}
t_{\lambda \xi} t_{\lambda \xi}^* \Theta(t)_{\lambda,\mu}^{\Pi E} &=t_{\lambda \xi}t_{\lambda \xi}^* \left ( \prod_{\lambda \nu \in \Pi E\atop d(\nu)>0} \left( t_\lambda t_\lambda^*-t_{\lambda \nu } t_{\lambda \nu}^* \right) \right) t_\lambda t_\mu^*\\
&=t_{\lambda \xi} t_{\lambda \xi}^* t_\lambda t_\mu^*\\
&=c(\lambda,\xi)t_{\lambda \xi} t_{\xi}^* t_\lambda^* t_\lambda t_\mu^*\\
&=c(\lambda,\xi)t_{\lambda \xi} t_{\xi}^* t_{s(\lambda)}^* t_\mu^*\\
&=c(\lambda,\xi)\overline{c(s(\lambda),\xi)}t_{\lambda \xi} t_{s(\lambda)\xi}^*  t_\mu^*\\
&=c(\lambda,\xi)\overline{c(\nu,\xi)}t_{\lambda \xi}t_{\mu \xi}^*.
\end{align*}
Since $t_{\lambda \xi} t_{\mu \xi}^* t_{\mu \xi}=t_{\lambda \xi} \neq 0$ by Lemma~\ref{handylemma}(7), we have $t_{\lambda \xi}t_{\mu \xi}^*\neq 0$. Since the 2-cocycle takes values in the unit circle, we have $\Theta(t)_{\lambda,\mu}^{\Pi E}\neq 0$.

Now suppose that $T(\lambda)$ is exhaustive. By Lemma~\ref{le:matrixunits3} we have
$$\Theta(t)_{\lambda,\mu}^{\Pi E} =t_\lambda \left( \prod_{\lambda \nu \in \Pi E \atop d(\nu)>0} \left ( t_{s(\lambda)}-t_\nu t_\nu^* \right)\right) t_\mu^*,$$
and
$$\prod_{\lambda \nu \in \Pi E \atop d(\nu)>0} \left ( t_{s(\lambda)}-t_\nu t_\nu^* \right)=t_\lambda^* \Theta(t)_{\lambda,\mu}^{\Pi E} t_\mu=0,$$
and the final assertion follows.
\end{proof}

\section{Isomorphisms of the core}
In this section we show how each finite-dimensional subalgebra $M^t_{\Pi E}$ decomposes into a direct sum of matrix algebras. For $E\subset G$, we write down the inclusion map $M^t_{\Pi E} \hookrightarrow M^t_{\Pi G}$. If $\left \{ t_\lambda : \lambda \in \Lambda \right \}$ is a relative Cuntz-Krieger $(\Lambda,b;\mathcal{E})$-family and $ \left \{ t_\lambda' : \lambda \in \Lambda \right \}$ is a relative Cuntz-Krieger $(\Lambda,c;\mathcal{E})$-family, we provide conditions under which
$$C^*\left( \left\{ t_\lambda t_\mu^* : \lambda ,\mu\in \Lambda, d(\lambda)=d(\mu) \right \} \right ) \cong C^*\left( \left\{ t_\lambda' t_\mu^{\prime*}: \lambda,\mu \in \Lambda, d(\lambda)=d(\mu) \right \} \right ).$$

\label{chpt:isoofthecore}
\begin{definition}[\cite{SPhD}, Definition 3.6.1]
Let $(\Lambda,d)$ be a finitely aligned $k$-graph, let $c\in \underline{Z}^2(\Lambda,\mathbb{T})$ and let $\left\{ t_\lambda : \lambda  \in \Lambda \right \}$ be a Toeplitz-Cuntz-Krieger $(\Lambda,c)$-family. Let $E\subset \Lambda$ be finite and suppose that $n\in \mathbb{N}^k$ and that $v\in \Lambda^0$ are such that $(\Pi E)v \cap \Lambda^n = \emptyset$. Define $M_{\Pi E}^t (n,v)$ to be the subalgebra
$$M_{\Pi E}^t (n,v):= \linspan \left \{ \Theta(t)_{\lambda,\mu}^{\Pi E} : \lambda, \mu \in (\Pi E ) v \cap \Lambda^n \right \}$$
of $M_{\Pi E}^t$. Define $T^{\Pi E}(n,v)$ to be the set
$$T^{\Pi E}(n,v):=\left \{ \nu\in v \Lambda \setminus \Lambda^0 : \lambda \nu \in \Pi E \text{ for } \lambda \in (\Pi E)v\cap \Lambda^n \right \}$$
of nontrivial tails which extend paths in $(\Pi E )v \cap \Lambda^n $ to larger elements of $\Pi E$.
\end{definition}

\begin{lemma}[\cite{SPhD}, Lemma 3.6.2]\label{le:coredirectsum}
Let $(\Lambda,d)$ be a finitely aligned $k$-graph, let $c\in \underline{Z}^2(\Lambda,\mathbb{T})$ and let $\left\{ t_\lambda : \lambda  \in \Lambda \right \}$ be a Toeplitz-Cuntz-Krieger $(\Lambda,c)$-family. Let $E\subset \Lambda$ be finite. Then
$$M_{\Pi E}^t = \bigoplus_{v\in s( \Pi E) \atop n\in d((\Pi E)v)} M_{\Pi E}^t(n,v).$$
For fixed $v\in s(\Pi E)$ and $n\in d( ( \Pi E)v)$,
\begin{equation}
M_{\Pi E}^t(n,v) \cong \begin{cases} \left \{ 0 \right \} & \text{if } \prod_{v\in T^{\Pi E}(n,v) }\left( t_v-t_\nu t_\nu^* \right)=0\\ M_{(\Pi E)v \cap \Lambda^n} ( \mathbb{C}) &\text{otherwise.} \end{cases} \label{eq:matrixinclu1}
\end{equation}
\end{lemma}

\begin{proof}
Suppose that $(\lambda,\mu) \in \Pi E \times_{d,s} \Pi E$. Then $d(\lambda)=d(\mu)$ and $s(\lambda)=s(\mu)$ and so $\Theta(t)_{\lambda,\mu}^{\Pi E} \in M_{\Pi E}^t(d(\lambda),s(\lambda))$. Also $\Theta(t)_{\lambda,\mu}^{\Pi E} \notin M_{\Pi E}^t(m,w)$ whenever $d(\lambda) \neq m$ or $s(\lambda) \neq w$. Hence each of the spanning elements $\Theta(t)_{\lambda,\mu}^{\Pi E}$ of $M_{\Pi E}^t$ belong to exactly one of the matrix algebras $M_{\Pi E}^t(n,v)$. Suppose that $(\lambda,\mu),(\nu,\tau) \in \Pi E \times_{d,s} \Pi E$, say $s(\lambda)=v$, $d(\lambda)=n$ and $s(\nu)=w$, $d(\nu)=m$. If
$$\Theta(t)_{\lambda,\mu}^{\Pi E} \Theta(t)_{\nu,\tau}^{\Pi E}\neq 0,$$
then $\mu=\nu$ by Proposition~\ref{prop:matrixunits2}, and hence $v=w$ and $n=m$. So if $n\neq m$ or $v\neq w$ then $M_{\Pi E}^t(n,v) M_{\Pi E}^t(m,w)=\left \{0 \right \}.$ Hence $$M_{\Pi E}^t = \bigoplus_{v\in s( \Pi E) \atop n\in d((\Pi E)v)} M_{\Pi E}^t(n,v).$$

We will now prove \eqref{eq:matrixinclu1}. Fix $v\in s(\Pi E)$ and $n\in d( ( \Pi E)v)$. Suppose that $$\prod_{\nu \in T^{\Pi E}(n,v)} \left( t_v -t_\nu t_\nu^* \right)=0.$$ Then Proposition~\ref{prop:matrixunits4} implies that $\Theta(t)_{\lambda,\mu}^{\Pi E}=0$ for all $\lambda,\mu \in (\Pi E) v \cap \Lambda^n$. Hence $M_{\Pi E}^t(n,v)=\left\{ 0 \right \}$. Now suppose that $$\prod_{\nu \in T^{\Pi E}(n,v)} \left( t_v -t_\nu t_\nu^* \right)\neq0.$$ Then Proposition~\ref{prop:matrixunits2} and Proposition~\ref{prop:matrixunits4} imply that $$\left \{ \Theta(t)_{\lambda,\mu}^{\Pi E}: \lambda,\mu \in (\Pi E)v \cap \Lambda^n \right \}$$ is family of nonzero matrix units which span $M_{\Pi E}^t(n,v)$. Lemma~\ref{le:matrixalgebra} therefore implies that $M_{\Pi E}^t(n,v)\cong M_{(\Pi E)v \cap \Lambda^n}(\mathbb{C})$.
\end{proof}

\begin{lemma}[cf. Lemma 3.6.3, \cite{SPhD}] \label{le:iota1} Let $(\Lambda,d)$ be a finitely aligned $k$-graph, and let $E\subset \Lambda$ be finite. Suppose that $\lambda \in \Lambda \setminus \Pi E$ and that there exists $n \leq d(\lambda)$ such that $\lambda(0,n)\in \Pi E$. Then there is a unique path $\iota_\lambda^E \in \Pi E$ such that:
\begin{enumerate}
\item[(1)] $\lambda(0,d(\iota^E_\lambda))=\iota^E_\lambda$; and
\item[(2)] if $\mu \in \Pi E$ and $\lambda(0,d(\mu))=\mu$, then $d(\mu)\leq d(\iota^E_\lambda)$.
\end{enumerate}

\end{lemma}

\begin{proof}
Let $N:=\vee \left\{ n \leq d(\lambda):\lambda(0,n)\in \Pi E\right \}$ and let $\iota_\lambda^E=\lambda(0,N)$. Then $\iota_\lambda^E\in \Pi E$ because $\Pi E$ is closed under minimal common extensions. We have $\lambda(0,d(\iota_\lambda^E))=\iota_\lambda^E$ by definition. If $\mu \in \Pi E$ and $\lambda(0,d(\mu))=\mu$, then $d(\iota_\lambda^E)=N\geq d(\mu)$ by definition.
\end{proof}

\begin{definition} [cf. \cite{SPhD}, Definition 3.6.4]\label{def:iotaandkappa}
Let $(\Lambda,d)$ be a finitely aligned $k$-graph, and let $E$ be a finite subset of $\Lambda$. For those $\lambda \in \Lambda \setminus \Pi E$ such there exists $n \leq d(\lambda)$ for which $\lambda(0,n)\in \Pi E$, the path $\iota^E_\lambda$ is defined by Lemma~\ref{le:iota1}. For all other $\lambda \in \Lambda$, we define $\iota^E_\lambda:=\lambda$. For $\lambda\in \Lambda$, we define $\kappa^E_\lambda=\lambda(d(\iota^E_\lambda),d(\lambda)).$
\end{definition}

Lemma~\ref{le:iota1} and Definition~\ref{def:iotaandkappa} are similar to Lemma 3.6.3 and Definition 3.6.4 in \cite{SPhD}. The difference is that here $\iota_\lambda^E$ is defined for all $\lambda \in \Lambda$, whereas in \cite{SPhD} it is only defined for $\lambda \in G$ where $G$ is finite and $E\subset G$. That $\iota_\lambda^E$ is defined for all $\lambda \in \Lambda$ is crucial in the statement of Proposition~\ref{co:twistedcoreiso3} and in the proof of Lemma~\ref{le:theladderlemma}.

\begin{lemma}[cf. \cite{SPhD}, Lemma 3.6.5] \label{le:iota2}  Let $(\Lambda,d)$ be a finitely aligned $k$-graph, and let $E,G\subset \Lambda$ be finite with $E\subset G$. Suppose that $(\lambda,\mu) \in \Pi E \times_{d,s} \Pi E$, that $\lambda \nu \in \Pi G$ and that $\iota^E_{\lambda \nu}=\lambda$. Then $\mu \nu \in \Pi G$ and $\iota^E_{\mu \nu }=\mu$.
\end{lemma}

\begin{proof}
Since $\Pi E \subset \Pi G$, we have $\mu \nu \in \Pi G$ by Lemma~\ref{re:matrixunits0}(2). Suppose for contradiction that $\iota_{\mu \nu}\neq \mu$. Since $(\mu\nu)(0,d(\mu))=\mu$ and $\mu \in \Pi E$, Lemma~\ref{le:iota1}(2) ensures that $\iota_{\mu \nu}=\mu \mu'$ for some $\mu' \in \Lambda \setminus \Lambda^0$. Lemma~\ref{re:matrixunits0}(2) implies that $\lambda \mu' \in \Pi E$. Since $d(\lambda)=d(\mu)$, we have $d(\lambda \mu')>d(\lambda)=d(\iota_{\lambda \nu}).$ Since $\iota_{\mu \nu}=\mu \mu'$, we have $(\mu \nu)(0,d(\mu \mu'))=\mu \mu'$, and the factorisation property implies that $\nu =\mu'\nu'$ for some $\nu'\in \Lambda$. Another application of the factorisation property gives $(\lambda \nu)(0,d(\lambda \mu'))=\lambda \mu'$. We have shown that $\lambda \mu' \in \Pi E$, that $d(\lambda \mu')>d(\iota_{\lambda \nu})$, and that $(\lambda \nu)(0,d(\lambda \mu'))=\lambda \mu'$. This contradicts Lemma~\ref{le:iota1}(2).
\end{proof}

\begin{lemma}[cf. \cite{SPhD}, Lemma 3.6.6]\label{le:inclusion1}
Let $(\Lambda,d)$ be a finitely aligned $k$-graph, let $c\in \underline{Z}^2(\Lambda,\mathbb{T})$ and let $\left\{ t_\lambda : \lambda  \in \Lambda \right \}$ be a Toeplitz-Cuntz-Krieger $(\Lambda,c)$-family. Let $E,G\subset \Lambda$ be finite with $E\subset G$. Then for each $(\lambda,\mu)\in \Pi E \times_{d,s} \Pi E$,
$$\Theta(t)_{\lambda,\mu}^{\Pi E}=\sum_{\lambda \nu \in \Pi G \atop \iota^E_{\lambda\nu}=\lambda}c(\lambda,\nu)\overline{c(\mu,\nu)} \Theta(t)_{\lambda \nu,\mu \nu}^{\Pi G}.$$
\end{lemma}

\begin{proof}
We first show that
\begin{equation}\lambda \in \Pi G \text{ and } \mu \in \Pi E \text{ implies that } Q(t)_{\mu}^{\Pi E} Q(t)_{\lambda}^{\Pi G} = \delta_{\iota_\lambda,\mu}Q(t)_\lambda^{\Pi G}.\label{eq:inclusion1ourclaim}\end{equation}
First suppose that $\iota_\lambda \neq \mu$. We consider two cases:

Case 1: suppose that $\lambda \notin \mu \Lambda$. Then $(\alpha,\beta)\in \Lambda^{\text{min}}(\lambda,\mu)$ implies that $d(\alpha)>0$. For each $(\alpha,\beta)\in \Lambda^{\text{min}}(\lambda,\mu)$, $(t_\lambda t_\lambda^*-t_{\lambda \alpha} t_{\lambda\alpha}^*)$ is a factor in $Q(t)_\lambda^{\Pi G}$. Hence
\begin{align*}
Q(t)_\mu^{\Pi E} Q(t)_\lambda^{\Pi G} &\leq t_\mu t_\mu^*Q(t)_\lambda^{\Pi G} \\
&=t_\mu t_\mu^* t_\lambda t_\lambda^* Q(t)_\lambda^{\Pi G}\\
&\leq \sum_{(\alpha,\beta)\in \Lambda^{\text{min}}(\lambda,\mu)} t_{\lambda \alpha}t_{\lambda \alpha}^* Q(t)_\lambda^{\Pi G}\\
&=0.\end{align*}

Case 2: suppose that $\lambda \in \mu \Lambda$. Since $\iota_\lambda \neq \mu$, Lemma~\ref{le:iota1}(2) implies that $\iota_\lambda = \mu \alpha$ where $d(\alpha)>0$. In particular $\mu \alpha \in \Pi E$, and so $$ Q(t)_\mu^{\Pi E} Q(t)_\lambda^{\Pi G} \leq (t_\mu t_\mu^*-t_{\mu \alpha}t_{\mu \alpha}^*)t_\lambda t_\lambda^*.$$ As $\mu\alpha$ is an initial segment of $\lambda$, we have $\MCE(\lambda,\mu \alpha)=\left\{ \lambda \right \}= \MCE(\lambda,\mu)$. So (TCK4) gives $$Q(t)_\mu^{\Pi E} Q(t)_\lambda^{\Pi G}\leq t_\lambda t_\lambda^* -t_\lambda t_\lambda^*=0.$$
We have established that $\iota_\lambda \neq \mu $ implies that $Q(t)_\mu^{\Pi E} Q(t)_\lambda^{\Pi G}=0$. Now suppose that $\iota_\lambda =\mu$. Then
\begin{equation}Q(t)_\mu^{\Pi E} Q(t)_\lambda^{\Pi G} =t_\mu t_\mu^* \prod_{\mu \nu \in \Pi E \atop d(\nu)>0}\left( t_\mu t_\mu^*-t_{\mu\nu}t_{\mu \nu}^*\right) Q(t)_\lambda^{\Pi G}.\label{eq:inclusion1def}\end{equation}

Suppose that $\mu \nu \in \Pi E$ with $d(\nu)>0$. Since $\iota_\lambda =\mu$, the maximality of $\iota_\lambda$ implies that $\lambda \notin \mu \nu \Lambda$. So $(\alpha,\beta)\in \Lambda^\text{min}(\lambda,\mu\nu)$ implies that $d(\alpha)>0$. Fix $(\alpha,\beta)\in \Lambda^\text{min}(\lambda,\mu \nu)$. Then $(t_\lambda t_\lambda^*-t_{\lambda \alpha}t_{\lambda \alpha}^*)$ is a factor $Q(t)_\lambda^{\Pi G}$. Hence $t_{\lambda \alpha}t_{\lambda \alpha}^* Q(t)^{\Pi G}_\lambda=0$. Since $t_\lambda t_\lambda^* \geq Q(t)_\lambda^{\Pi E}$, we have
\begin{align*}
(t_\mu t_\mu^*&-t_{\mu \nu}t_{\mu \nu}^*) Q(t)_\lambda^{\Pi E}\\
&=t_\mu t_\mu^*t_\lambda t_\lambda^* Q(t)_\lambda^{\Pi E} -t_{\mu \nu}t_{\mu \nu}^* t_\lambda t_\lambda^* Q(t)_\lambda^{\Pi E}\\
&=t_\mu t_\mu^*t_\lambda t_\lambda^* Q(t)_\lambda^{\Pi E}-\sum_{(\alpha,\beta)\in\Lambda^\text{min}(\lambda,\mu \nu)} t_{\lambda \alpha} t_{\lambda \alpha}^* Q(t)_\lambda^{\Pi E} \quad \text{ by (TCK4)}\\
&=t_\mu t_\mu^* t_\lambda t_\lambda^* Q(t)_\lambda^{\Pi E}.
\end{align*}
Since $t_\lambda t_\lambda^* \leq t_\mu t_\mu^*$, we conclude that $(t_\mu t_\mu^*-t_{\mu\nu}t_{\mu\nu}^*)Q(t)_\lambda^{\Pi E}=Q(t)_\lambda^{\Pi E}$. Applying this to each factor on the right-hand side of \eqref{eq:inclusion1def} gives $Q(t)_\mu^{\Pi E} Q(t)_\lambda^{\Pi G}=Q(t)_\lambda^{\Pi G}$. We have established~\eqref{eq:inclusion1ourclaim}.

Fix $(\lambda,\mu)\in \Pi E \times_{d,s}\Pi E$. We have $\Theta(t)_{\lambda,\mu}^{\Pi E}=t_\lambda t_\lambda^* Q(t)_\lambda^{\Pi E} t_\lambda t_\mu^*$. An application of Corollary~\ref{co:orth5} gives
$$\Theta(t)_{\lambda,\mu}^{\Pi E}=\sum_{\lambda \nu \in\Pi G} Q(t)_{\lambda \nu}^{\Pi G} Q(t)_\lambda^{\Pi E} t_\lambda t_\mu^*.$$ Applying \eqref{eq:inclusion1ourclaim} to each term in this sum gives
\begin{equation} \Theta(t)_{\lambda,\mu}^{\Pi E} = \sum_{\lambda \nu \in \Pi G \atop \iota_{\lambda \nu } =\lambda}Q(t)_{\lambda \nu}^{\Pi G} t_\lambda t_\mu^*.\label{eq:matunitinclusion}\end{equation}
Fix $\lambda \nu \in \Pi G$ with $\iota_{\lambda \nu }=\lambda$. Since $Q(t)_{\lambda \nu}^{\Pi G}\leq t_{\lambda \nu}t_{\lambda \nu}^*$, \begin{flalign*}&&Q(t)_{\lambda \nu}^{\Pi G}t_\lambda t_\mu^*&=Q(t)_{\lambda \nu}^{\Pi G} t_{\lambda \nu} t_{\lambda \nu}^* t_\lambda t_\mu^*&\\&&&=c(\lambda,\nu)Q(t)_{\lambda \nu}^{\Pi G} t_{\lambda \nu}t_\nu^* t_\lambda^*t_\lambda t_\mu^*&\\&&&=c(\lambda,\nu)\overline{c(\mu,\nu)}Q(t)_{\lambda \nu}^{\Pi G} t_{\lambda \nu}t_{\mu \nu}^*&\\
&&&=c(\lambda,\nu)\overline{c(\mu,\nu)}\Theta(t)_{\lambda \nu,\mu \nu}^{\Pi G}.&\qedhere\end{flalign*}
\end{proof}

\begin{definition}[\cite{SPhD}, Definition 3.6.7]
Let $(\Lambda,d)$ be a finitely aligned $k$-graph, let $c\in \underline{Z}^2(\Lambda,\mathbb{T})$ and let $\left\{ t_\lambda : \lambda  \in \Lambda \right \}$ be a Toeplitz-Cuntz-Krieger $(\Lambda,c)$-family. Let $E,G\subset \Lambda$ be finite with $E\subset G$. Suppose that $n\in \mathbb{N}^k$ and that $v\in \Lambda^0$ are such that $M_{\Pi E}^t(n,v)$ is nontrivial. We define
$$I^G_E (n,v):= \left \{ \nu \in \Lambda : \lambda \nu \in \Pi G \text{ and } \iota^E_{\lambda \nu}=\lambda \text{ for } \lambda \in (\Pi E) v \cap \Lambda^n \right \}.$$
For $\lambda \in \Pi E$, we write $I_E^G(\lambda)$ for $I_E^G(d(\lambda),s(\lambda))$ for convenience.
\end{definition}

\begin{corollary}[cf. \cite{SPhD}, Corollary 3.6.8]\label{co:inclusion1}
Let $(\Lambda,d)$ be a finitely aligned $k$-graph, let $c\in \underline{Z}^2(\Lambda,\mathbb{T})$ and let $\left\{ t_\lambda : \lambda  \in \Lambda \right \}$ be a Toeplitz-Cuntz-Krieger $(\Lambda,c)$-family. Let $E,G\subset \Lambda$ be finite with $E\subset G$. The inclusion map $M_{\Pi E}^t \hookrightarrow M_{\Pi G}^t$ is given by
$$\Theta(t)_{\lambda,\mu}^{\Pi E} \mapsto \bigoplus_{v\in s(I_E^G(\lambda)) \atop n\in d(I_E^G(\lambda )v)} \sum_{\nu\in I_E^G(\lambda)v \cap \Lambda^n} c(\lambda,\nu)\overline{c(\mu,\nu)}\Theta(t)_{\lambda \nu, \mu \nu}^{\Pi G}.$$
\end{corollary}

\begin{proof}
The result is a direct consequence of Lemma~\ref{le:coredirectsum} and Lemma~\ref{le:inclusion1}.
\end{proof}

\begin{theorem}\label{th:coreiso}
Let $(\Lambda,d)$ be a finitely aligned $k$-graph and let $b,c\in\underline{Z}^2(\Lambda,\mathbb{T})$. Let $\mathcal{E}$ be a subset of $\FE(\Lambda)$. Suppose that $\left \{ t_\lambda : \lambda \in \Lambda \right \}$ is a relative Cuntz-Krieger $(\Lambda,b;\mathcal{E})$-family and that $ \left \{ t_\lambda' : \lambda \in \Lambda \right \}$ is a relative Cuntz-Krieger $(\Lambda,c;\mathcal{E})$-family. Suppose that
\begin{enumerate}
\item[(1)] $t_v\neq 0$ and $t_v'\neq 0$ for all $v\in \Lambda^0$, and that
\item[(2)] $\left\{E\in \FE(\Lambda):Q(t)^E=0\right\}=\overline{\mathcal{E}}=\left\{E\in \FE(\Lambda):Q(t')^E= 0\right\}$.
\end{enumerate}
Then there is an isomorphism $\pi_{t,t'}: C^*\left(\left \{ t_\lambda t_\mu^*:d(\lambda)=d(\mu) \right \}\right) \rightarrow C^*\left(\left \{ t_\lambda' t_\mu^{\prime*}:d(\lambda)=d(\mu) \right \}\right) $ satisfying $\pi_{t,t'}(t_\lambda t_\lambda^*)= t_\lambda't_\lambda^{\prime*}$ for all $\lambda \in \Lambda$. If $b\equiv c$, then $\pi_{t,t'}(t_\lambda t_\mu^*)= t_\lambda't_\mu^{\prime*}$ for all $\lambda,\mu \in \Lambda$ with $d(\lambda)=d(\mu)$. We have
$$C^*(\Lambda,b;\mathcal{E})^\gamma\cong C^*(\Lambda,c;\mathcal{E})^\gamma.$$
\end{theorem}

Before proving Theorem~\ref{th:coreiso}, we need to establish several technical results.

\begin{lemma}\label{le:twistedcoreiso2}
Let $I$ be a finite set. Let $\left \{ e_{i,j} : i,j\in I \right \}$ be a set of matrix units in a $C^*$-algebra $A$. Let $z: I \times I \rightarrow \mathbb{T}$ satisfy:
\begin{enumerate}
\item[(T1)] $z(i,j)=\overline{z(j,i)}$ for all $i,j\in I$; and
\item[(T2)] $z(i,j)z(j,k)=z(i,k)$ for all $i,j,k\in I$.
\end{enumerate}
Then $\left \{ z(i,j)e_{i,j} : i,j \in I\right \}$ is a collection of matrix units and $$\linspan \left  \{ e_{i,j} : i,j\in I \right \} = \linspan \left \{ z(i,j)e_{i,j} : i,j \in I\right \}.$$ Moreover, if $w: I \times I \rightarrow \mathbb{T}$ satisfies properties $(T1)$ and $(T2)$, then $zw: I \times I \rightarrow \mathbb{T}$ defined by $(zw)(i,j)=z(i,j)w(i,j)$ for all $i,j\in I$ also satisfies $(T1)$ and $(T2)$.
\end{lemma}

\begin{proof}
Fix $i,j,k,l\in I$. We show that $\left( z(i,j)e_{i,j} \right)^*=z(j,i)e_{j,i}$ and $\left( z(i,j)e_{i,j} \right) \left( z(k,l)e_{k,l} \right)=\delta_{j,k}z(i,l)e_{i,l}$. This shows $(M1)$ and $(M2)$. Firstly, (T1) and then (M1) imply that
\begin{align*}
\left( z(i,j)e_{i,j} \right)^*&=\overline{z(i,j)} e_{i,j}^* =z(j,i)e_{i,j}^* =z(j,i)e_{j,i},
\end{align*}
and secondly, (M2) and then (T2) give
\begin{align*}
\left( z(i,j)e_{i,j} \right) \left( z(k,l)e_{k,l} \right)=\delta_{j,k}z(i,j)z(k,l)e_{i,l}
=\delta_{j,k}z(i,j)z(j,l)e_{i,l}
&=\delta_{j,k}z(i,l)e_{i,l}.
\end{align*}
So $\left \{ z(i,j)e_{i,j} : i,j \in I\right \}$ is a collection of matrix units. It is clear that
$$\linspan \left  \{ e_{i,j} : i,j\in I \right \} = \linspan \left \{ z(i,j)e_{i,j} : i,j \in I\right \}$$
since $z$ takes values in the circle. For the final statement, let $w: I \times I \rightarrow \mathbb{T}$ satisfy $(T1)$ and $(T2)$. Then for $i,j,k,\in I$ we have,
\begin{align*}
\overline{(zw)(i,j)}&=\overline{z(i,j)w(i,j)}\\
&=z(j,i)w(j,i) \qquad \text{ by (T1)}\\
&=(zw)(j,i),
\end{align*}
and
\begin{align*}
(zw)(i,j)(zw)(j,k)&=z(i,j)z(j,k)w(i,j)w(j,k)\\
&=z(i,k)w(i,k) \qquad \text{ by (T2)}\\
&=(zw)(i,k).
\end{align*}
So $zw$ satisfies (T1) and (T2).
\end{proof}

\begin{proposition}\label{co:twistedcoreiso3}
Let $(\Lambda,d)$ be a finitely aligned $k$-graph, let $b,c\in\underline{Z}^2(\Lambda,\mathbb{T})$ and let $\left\{ t_\lambda : \lambda  \in \Lambda \right \}$ be a Toeplitz-Cuntz-Krieger $(\Lambda,c)$-family. Let $E, G \subset \Lambda$ be finite with $E\subset G$, $v\in \Lambda^0$ and let $n\in \mathbb{N}^k$. Then $\omega_b,\omega_c : \Lambda \times \Lambda \rightarrow \mathbb{T}$ defined by
$$\omega_b(\lambda,\mu) = b(\iota^E_\mu,\kappa^E_\mu)\overline{b(\iota^E_\lambda,\kappa^E_\lambda)} \text{ and } \omega_c(\lambda,\mu) = c(\iota^E_\lambda,\kappa^E_\lambda) \overline{c(\iota^E_\mu,\kappa^E_\mu)}$$
satisfy $(T1)$ and $(T2)$. If $z: \Lambda \times \Lambda \rightarrow \mathbb{T}$ satisfies $(T1)$ and $(T2)$ then
$$\left \{ z(\lambda,\mu)\omega_b(\lambda,\mu)\omega_c(\lambda,\mu) \Theta(t)_{\lambda,\mu}^{\Pi G}: (\lambda,\mu)\in \Pi G \times_{d,s} \Pi G \right\}$$
is a collection of matrix units which span $M_{\Pi G}^t$.
\end{proposition}

\begin{proof}
For $(\lambda,\mu), (\mu,\nu) \in \Lambda \times \Lambda$, the calculations
\begin{align*}
\omega_b(\lambda,\mu)=b(\iota^E_\mu,\kappa^E_\mu)\overline{b(\iota^E_\lambda,\kappa^E_\lambda)}=\overline{\overline{b(\iota^E_\mu,\kappa^E_\mu)}b(\iota^E_\lambda,\kappa^E_\lambda) }=\overline{b(\iota^E_\lambda,\kappa^E_\lambda) \overline{b(\iota^E_\mu,\kappa^E_\mu)}}=\overline{\omega_b(\mu,\lambda)}
\end{align*}
and
\begin{align*}
\omega_b(\lambda,\mu)\omega_b(\mu,\nu)=b(\iota^E_\mu,\kappa^E_\mu)\overline{b(\iota^E_\lambda,\kappa^E_\lambda)}b(\iota^E_\nu,\kappa^E_\nu)\overline{b(\iota^E_\mu,\kappa^E_\mu)}=b(\iota^E_\nu,\kappa^E_\nu)\overline{b(\iota^E_\lambda,\kappa^E_\lambda)}=\omega_b(\lambda,\nu)
\end{align*}
show that $\omega_b$ satisfies (T1) and (T2). A similar calculation shows that $\omega_c$ also satisfies (T1) and (T2). By Proposition~\ref{prop:matrixunits2}
$$\left \{\Theta(t)_{\lambda,\mu}^{\Pi G}: (\lambda,\mu)\in \Pi G \times_{d,s} \Pi G \right\}$$
is a collection of matrix units which span $M_{\Pi G}^t$. The last statement is then a direct consequence of Lemma~\ref{le:twistedcoreiso2}.
\end{proof}

Since $\Lambda$ is countable, we can write $\Lambda =\left \{ \lambda_0 , \lambda_1, \lambda_2,... \right \}$. Put $E_i:=\left \{ \lambda_j : 0\leq j\leq i \right \}$. Then $E_i \subset E_{i+1}$ for each $i\in \mathbb{N}$ and $\bigcup_{i=0}^\infty E_i= \Lambda$. So there exists a family $\left \{ E_i : i\in \mathbb{N} \right \}$ of finite subsets of $\Lambda$ such that $E_i \subset E_{i+1}$ for each $i\in \mathbb{N}$ and $\bigcup_{i=0}^\infty E_i= \Lambda$.

\begin{lemma}\label{le:theladderlemma}
Let $(\Lambda,d)$ be a finitely aligned $k$-graph and let $b,c\in\underline{Z}^2(\Lambda,\mathbb{T})$. Let $\mathcal{E}$ be a subset of $\FE(\Lambda)$. Let $\left\{t_\lambda : \lambda \in \Lambda \right \}$ be a relative Cuntz-Krieger $(\Lambda,b;\mathcal{E})$-family and let $\left \{ t_\lambda': \lambda \in \Lambda \right \}$ be a relative Cuntz-Krieger $(\Lambda,c;\mathcal{E})$-family such that:
\begin{enumerate}
\item[(1)] $t_v\neq0$ and $t'_v\neq 0$ for all $v\in \Lambda^0$; and
\item[(2)] $\left\{E\in \FE(\Lambda):Q(t)^E=0\right\}=\overline{\mathcal{E}}=\left\{E\in \FE(\Lambda):Q(t')^E= 0\right\}$.
\end{enumerate}Fix a family $\left \{ E_i : i\in \mathbb{N} \right \}$ of finite subsets of $\Lambda$ such that $E_i \subset E_{i+1}$ for each $i\in \mathbb{N}$ and $\bigcup_{i=0}^\infty E_i= \Lambda$. For each $i\in \mathbb{N}$ and each $\lambda \in \Lambda$, let $\iota^{i}_\lambda:=\iota_\lambda^{E_i}$ and let $\kappa^{i}_\lambda:=\kappa^{E_i}_\lambda$ as in Definition~\ref{def:iotaandkappa}. Define $\omega_0: \Lambda \times \Lambda \rightarrow \mathbb{T}$ by \begin{equation} \omega_0(\lambda,\mu)=1\label{eq:interchangingcocylefunction1} \end{equation}
for all $\lambda ,\mu \in \Lambda$. For each $i\in \mathbb{N}$ define $\omega_{i+1}: \Lambda \times \Lambda \rightarrow \mathbb{T}$ by
\begin{equation}\omega_{i+1}(\lambda,\mu)=\omega_i(\iota^i_\lambda,\iota^i_\mu)b(\iota^{i}_\mu,\kappa^{i}_\mu)c(\iota^{i}_\lambda,\kappa^{i}_\lambda)\overline{b(\iota^{i}_\lambda,\kappa^{i}_\lambda)c(\iota^{i}_\mu,\kappa^{i}_\mu)}\label{eq:interchangingcocylefunction2} \end{equation}
for all $\lambda,\mu\in \Lambda$.
Then for each $i\in \mathbb{N}$ there is a unique isomorphism $\psi_i : M_{\Pi E_i}^{t} \rightarrow M_{\Pi E_i}^{t'}$ satisfying
\begin{equation}\psi_i \left( \Theta(t)_{\lambda,\mu}^{\Pi E_i}\right)=\omega_{i}(\lambda,\mu)\Theta(t')_{\lambda,\mu}^{\Pi E_i}  \label{eq:thecoreiso}\end{equation} for each $(\lambda,\mu) \in \Pi E_i \times_{d,s} \Pi E_i$. The diagram
\begin{equation}
\begin{tikzpicture}
\matrix(m)[matrix of math nodes,
row sep=4em, column sep=2.5em,
text height=1.5ex, text depth=0.25ex]
{M_{\Pi E_0}^{t}&M_{\Pi E_1}^{t}&M_{\Pi E_2}^{t}&\cdots\\
M_{\Pi E_0}^{t'}&M_{\Pi E_1}^{t'}&M_{\Pi E_2}^{t'}&\cdots\\};
\path[right hook->]
(m-1-1) edge node[auto]{$\phi_0^b$} (m-1-2)
(m-1-2) edge node[auto]{$\phi_1^b$}(m-1-3)
(m-2-1) edge node[auto]{$\phi_0^c$} (m-2-2)
(m-2-2) edge node[auto]{$\phi_1^c$} (m-2-3)
(m-1-3) edge node[auto]{$\phi_2^b$} (m-1-4)
(m-2-3) edge node[auto]{$\phi_2^c$}(m-2-4);
\path[->]
(m-1-1) edge node[auto]{$\psi_0$}(m-2-1)
(m-1-2) edge node[auto]{$\psi_1$} (m-2-2)
(m-1-3) edge node[auto]{$\psi_2$} (m-2-3);

\end{tikzpicture}
\label{dia:ladderofmatrixalgebras}\end{equation}

commutes, where the inclusion maps $\phi_i^b$ and $\phi_i^c$ for $i\in \mathbb{N}$ are given by Lemma~\ref{co:inclusion1}.
\end{lemma}

\begin{proof}

We claim that $\omega_i$ satisfies (T1) and (T2) for each $i\in \mathbb{N}$. We proceed by induction. The map $w_0$ satisfies (T1) and (T2) because $w_0(\lambda,\mu)=1$ for each $\lambda,\mu \in \Lambda$. Suppose that $\omega_k$ satisfies (T1) and (T2). Then
$$w_{k+1}(\lambda,\mu) =w_k(\iota_\lambda^i,\iota_\mu^i)b(\iota^k_\mu,\kappa^k_\mu)c(\iota^k_\lambda,\kappa^k_\lambda)\overline{b(\iota^k_\lambda,\kappa^k_\lambda)c(\iota^k_\mu,\kappa^k_\mu)}.$$ Proposition~\ref{co:twistedcoreiso3} implies that $\omega_{k+1}$ also satisfies (T1) and (T2).

We will now construct the maps $\psi_i :M_{\Pi E_i}^{t} \rightarrow M_{\Pi E_i}^{t'}$. Fix $i\in \mathbb{N}$. By Lemma~\ref{le:coredirectsum} we have
$$M_{\Pi E_i}^{t} = \bigoplus_{v\in s( \Pi E_i) \atop n\in d((\Pi E_i)v)} M_{\Pi E_i}^{t}(n,v),$$
where each $M_{\Pi E_i}^{t}(n,v)$ is spanned by the elements $$\left \{ \Theta(t)_{\lambda,\mu}^{\Pi E_i} : \lambda,\mu \in (\Pi E_i)v\cap \Lambda^n \right \}.$$ The same holds for the family $\left \{ t'_\lambda:\lambda \in \Lambda \right \}$. Hence there is at most one linear map $\psi_i : M_{\Pi E_i}^t \rightarrow M_{\Pi E_i}^{t'} $ satisfying \eqref{eq:thecoreiso}.

Fix $n\in \mathbb{N}^k$ and $v\in \Lambda^0$. We claim $Q(t)^{T^{\Pi E_i}(n,v)}=0$ if and only if $Q(t')^{T^{\Pi E_i}(n,v)}=0$. We consider the three cases:
\begin{enumerate}
\item[(1)] $T^{\Pi E_i}(n,v)\in\overline{\mathcal{E}}$
\item[(2)] $T^{\Pi E_i}(n,v)\in \FE(\Lambda)\setminus \overline{\mathcal{E}}$; and
\item[(3)] $T^{\Pi E_i}(n,v) \notin \FE(\Lambda)$.
\end{enumerate}
For case (1), suppose that $T^{\Pi E_i}(n,v)\in\overline{\mathcal{E}}$. Then $Q(t)^{T^{\Pi E_i}(n,v)}=0$ and $Q(t')^{T^{\Pi E_i}(n,v)}=0$ by (CK) and Lemma~\ref{le:satiation1}. For (2), suppose that $T^{\Pi E_i}(n,v)\in\FE(\Lambda)\setminus \overline{\mathcal{E}}$. Then $Q(t)^{T^{\Pi E_i}(n,v)}\neq0$ and $Q(t')^{T^{\Pi E_i}(n,v)}\neq0$ by hypothesis. For case (3), suppose that $T^{\Pi E_i}(n,v)\notin \FE(\Lambda)$. Since $T^{\Pi E_i}(n,v)\subset v\Lambda \setminus \Lambda^0$ by definition, and since $t_v,t_v'\neq 0$, Lemma~\ref{le:nonzerogapfornonextsets} implies that both $Q(t)^{T^{\Pi E_i}(n,v)}$ and $Q(t')^{T^{\Pi E_i}(n,v)}$ are nonzero.

Lemma~\ref{le:coredirectsum} implies that $M_{\Pi_i E}^t(n,v)= \left \{ 0 \right \} = M_{\Pi E_i}^{t'}$ if $T^{\Pi E_i}(n,v)\in \overline{\mathcal{E}},$ and
$M^t_{\Pi E_i}(n,v)\cong M_{v(\Pi E_i)\cap \Lambda^n}(\mathbb{C})\cong M^{t'}_{\Pi E_i}(n,v)$ if $T^{\Pi E_i}(n,v) \notin \overline{\mathcal{E}}$.
It follows from Proposition~\ref{co:twistedcoreiso3} that whenever $T^{\Pi E_i}(n,v)\notin \overline{\mathcal{E}}$, $$\left \{ \Theta(t)_{\lambda,\mu}^{\Pi E_i} : \lambda,\mu \in (\Pi E_i)v\cap \Lambda^n \right \}$$ and $$\left \{ \omega_i(\lambda,\mu)\Theta(t')_{\lambda,\mu}^{\Pi E_i} : \lambda,\mu \in (\Pi E_i)v\cap \Lambda^n \right \}$$
are collections of nonzero matrix units which span $M^t_{\Pi E_i}(n,v)$ and $M^{t'}_{\Pi E_i}(n,v)$ respectively. By Corollary~\ref{co:matrixalgebrauni} there is an isomorphism $\psi_i^{n,v}:M_{\Pi E_i}^{t}(n,v) \rightarrow M_{\Pi E_i}^{t'}(n,v)$ taking $\Theta(t)_{\lambda,\mu}^{\Pi E_0} \mapsto w_i(\lambda,\mu)\Theta(t')_{\lambda,\mu}^{\Pi E_i} $ for each $\lambda,\mu \in (\Pi E_i)v \cap \Lambda^n$ such that $T^{\Pi E_i} (n,v)\notin \overline{\mathcal{E}}$. For each $(n,v)$ with $T^{\Pi E_i}(n,v) \in \overline{\mathcal{E}}$, $\psi_i^{n,v}=0:\left \{0 \right \} \rightarrow \left \{ 0 \right \}$ is an isomorphism. Define \begin{equation}\psi_i=\bigoplus_{v\in s( \Pi E_i) \atop n\in d((\Pi E_i)v)}\psi_i^{n,v}.\label{eq:psidef}\end{equation}
We now check that for each $i\in \mathbb{N}$ the diagram
\begin{equation}
\begin{tikzpicture}
\matrix(m)[matrix of math nodes,
row sep=4em, column sep=2.5em,
text height=1.5ex, text depth=0.25ex]
{M_{\Pi E_i}^{t}&M_{\Pi E_{i+1}}^{t}\\
M_{\Pi E_i}^{t'}&M_{\Pi E_{i+1}}^{t'}\\};
\path[right hook->]
(m-1-1) edge node[auto]{$\phi_i^b$} (m-1-2)
(m-2-1) edge node[auto]{$\phi_i^c$} (m-2-2);
\path[->]
(m-1-1) edge node[left]{$\psi_i$}(m-2-1)
(m-1-2) edge node[auto]{$\psi_{i+1}$} (m-2-2);
\end{tikzpicture}
\nonumber\end{equation}
commutes so that the diagram \eqref{dia:ladderofmatrixalgebras} commutes. Fix $i\in \mathbb{N}$ and $(\lambda,\mu) \in \Pi E_i \times_{d,s} \Pi E_i$. Applying the inclusion map given by Lemma~\ref{le:inclusion1} we have
\begin{align*}
(\psi_{i+1} \circ \phi_i^b)\left(\Theta(t)_{\lambda,\mu}^{\Pi E_i} \right)=\psi_{i+1} \Big( \bigoplus_{v\in s(I_{E_i}^{E_{i+1}}(\lambda)) \atop n\in d(I_{E_i}^{E_{i+1}}(\lambda )v)} \sum_{\nu\in I_{E_i}^{E_{i+1}}(\lambda)v \cap \Lambda^n} b(\lambda,\nu)\overline{b(\mu,\nu)}\Theta(t)_{\lambda \nu, \mu \nu}^{\Pi E_{i+1}}\Big).
\end{align*}
By definition of $I_{E_i}^{E_{i+1}}(\lambda)$, we have $\nu\in I_{E_i}^{E_{i+1}}(\lambda)v \cap \Lambda^n$ implies that $\iota^i_{\lambda \nu}=\lambda$ and $\kappa^i_{\lambda \nu}=\nu$. Then Lemma~\ref{le:iota2} implies that $\iota^i_{\mu \nu}=\mu$ and $\kappa^i_{\mu \nu}=\nu$. Using \eqref{eq:psidef}, we have
\begin{align*}
(\psi_{i+1} \circ \phi_i^b)\left(\Theta(t)_{\lambda,\mu}^{\Pi E_i} \right)= \bigoplus_{v\in s(I_{E_i}^{E_{i+1}}(\lambda)) \atop n\in d(I_{E_i}^{E_{i+1}}(\lambda )v)} \sum_{\nu\in I_{E_i}^{E_{i+1}}(\lambda)v \cap \Lambda^n} b(\iota^i_{\lambda \nu},\kappa^i_{\lambda \nu})\overline{b(\iota^i_{\mu \nu},\kappa^i_{\mu \nu})}\psi_{i+1}^{n,v}\left( \Theta(t)_{\lambda \nu, \mu \nu}^{\Pi E_{i+1}}\right).
\end{align*}
Since each $\psi_{i+1}^{n,v}$ takes $\Theta(t)_{\eta,\zeta}^{\Pi E_{i+1}}$ to $w_i(\eta,\zeta)\Theta(t')_{\eta,\zeta}^{\Pi E_{i+1}} $ whenever $\eta,\zeta \in (\Pi E_{i+1})v \cap \Lambda^n$, we have
\begin{align*}
(\psi_{i+1} \circ \phi_i^b)\left(\Theta(t)_{\lambda,\mu}^{\Pi E_i} \right)= \bigoplus_{v\in s(I_{E_i}^{E_{i+1}}(\lambda)) \atop n\in d(I_{E_i}^{E_{i+1}}(\lambda )v)} \sum_{\nu\in I_{E_i}^{E_{i+1}}(\lambda)v \cap \Lambda^n} b(\iota^i_{\lambda \nu},\kappa^i_{\lambda \nu})\overline{b(\iota^i_{\mu \nu},\kappa^i_{\mu \nu})}\omega_{i+1}(\lambda \nu,\mu \nu) \Theta(t')_{\lambda \nu, \mu \nu}^{\Pi E_{i+1}}.
\end{align*}
Applying \eqref{eq:interchangingcocylefunction2} gives
\begin{align*}
(\psi_{i+1} \circ \phi_i^b)\left(\Theta(t)_{\lambda,\mu}^{\Pi E_i} \right)= \bigoplus_{v\in s(I_{E_i}^{E_{i+1}}(\lambda)) \atop n\in d(I_{E_i}^{E_{i+1}}(\lambda )v)} \sum_{\nu\in I_{E_i}^{E_{i+1}}(\lambda)v \cap \Lambda^n}&b(\iota^i_{\lambda \nu},\kappa^i_{\lambda \nu})\overline{b(\iota^i_{\mu \nu},\kappa^i_{\mu \nu})}\omega_i(\iota^i_{\lambda \nu},\kappa^i_{\mu \nu}) b(\iota^i_{\mu \nu},\kappa^i_{\mu \nu}) \\
&\cdot c(\iota^i_{\lambda \nu},\kappa^i_{\lambda \nu})\overline{b(\iota^i_{\lambda \nu},\kappa^i_{\lambda \nu})c(\iota^i_{\mu \nu},\kappa^i_{\mu \nu})}\Theta(t')_{\lambda \nu,\mu \nu}^{\Pi E_{i+1}}.
\end{align*}
We continue,
\begin{flalign*}
&&(\psi_{i+1} \circ \phi_i^b)\left(\Theta(t)_{\lambda,\mu}^{\Pi E_i} \right)&= \bigoplus_{v\in s(I_{E_i}^{E_{i+1}}(\lambda)) \atop n\in d(I_{E_i}^{E_{i+1}}(\lambda )v)} \sum_{\nu\in I_{E_i}^{E_{i+1}}(\lambda)v \cap \Lambda^n} c(\iota^i_{\lambda \nu},\kappa^i_{\lambda \nu})\overline{c(\iota^i_{\mu \nu},\kappa^i_{\mu \nu})} \omega_i(\lambda,\mu)\Theta(t')_{\lambda \nu,\mu \nu}^{\Pi E_{i+1}}&\\
&&&=\phi_i^c\left( \omega_i(\lambda,\mu)\Theta(t')_{\lambda ,\mu }^{\Pi E_i}\right)&\\
&&&=( \phi_i^c \circ \psi_{i})\left( \Theta(t')_{\lambda ,\mu }^{\Pi E_i}\right).&\qedhere
\end{flalign*}
\end{proof}

\begin{proof}[Proof of Theorem~\ref{th:coreiso}]
Pick a family $\left \{ E_i : i\in \mathbb{N} \right \}$ of finite subsets of $\Lambda$ such that $E_i \subset E_{i+1}$ for each $i\in \mathbb{N}$ and $\bigcup_{i=0}^\infty E_i = \Lambda$.
By Lemma~\ref{le:piEsets} we have
$$\overline{\linspan}\left\{ t_\lambda t_\mu^* : \lambda,\mu\in \Lambda, d(\lambda)=d(\mu) \right \}=\overline{\bigcup_{i=0}^\infty M_{\Pi E_i }^t}$$
and
$$\overline{\linspan}\left\{ t_\lambda' t_\mu^{\prime*} : \lambda,\mu\in \Lambda,  d(\lambda)=d(\mu) \right \}=\overline{\bigcup_{i=0}^\infty M_{\Pi E_i }^{t'}}$$
For each $i\in \mathbb{N}$, let $\psi_i : M_{\Pi E_i}^t \rightarrow M_{\Pi E_i}^{t'}$ be the isomorphism given by Lemma~\ref{le:theladderlemma}, which satisfies $\psi_i \left( \Theta(t)_{\lambda,\mu}^{\Pi E_i}\right)=\omega_{i}(\lambda,\mu)\Theta(t')_{\lambda,\mu}^{\Pi E_i}$ for each $(\lambda,\mu) \in \Pi E_i \times_{d,s} \Pi E_i$.  Since the diagram~\eqref{dia:ladderofmatrixalgebras} commutes, there is a well-defined homomorphism $\pi_{t,t'}: \bigcup_{i=0}^\infty M_{\Pi E_i }^t \rightarrow \bigcup_{i=0}^\infty M_{\Pi E_i }^{t'}$ such that $\pi_{t,t'}|_{M_{\Pi E_i}^t}=\psi_i$. As each $\psi_i$ is injective on the $C^*$-algebra $M_{\Pi E_i}^t$, each $\psi_i$ isometric. Hence $\pi_{t,t'}$ is isometric on $\bigcup_{i=0}^\infty M_{\Pi E_i }^t$. By continuity $\pi_{t,t'}$ extends to an isometric homomorphism on $\overline{\bigcup_{i=0}^\infty M_{\Pi E_i }^t}$. Since $\pi_{t,t'}$ is a homomorphism between the $C^*$-algebras $\overline{\bigcup_{i=0}^\infty M_{\Pi E_i }^{t}}$ and $\overline{\bigcup_{i=0}^\infty M_{\Pi E_i }^{t'}}$, the image $\pi_{t,t'}\left(\overline{\bigcup_{i=0}^\infty M_{\Pi E_i }^{t}}\right)$ is closed. Since $\pi_{t,t'}\left(\bigcup_{i=0}^\infty M_{\Pi E_i }^{t}\right)$ is dense in $\overline{\bigcup_{i=0}^\infty M_{\Pi E_i }^{t'}}$, the homomorphism $\pi_{t,t'}$ surjective and hence an isomorphism.

We claim that for all $(\lambda,\mu)\in \Lambda \times_{d,s}\Lambda$ there is $i\in \mathbb{N}$ such that
\begin{equation}
\pi_{t,t'}(t_\lambda t_\mu^*)=\sum_{\lambda \nu \in \Pi E_i} c(\lambda,\nu) \overline{c(\mu,\nu)}\omega_i(\lambda \nu, \mu \nu)\Theta(t')_{\lambda \nu, \mu \nu}^{\Pi E_i}.\label{eq:thecoreiso4}
\end{equation}
To see this, fix $(\lambda,\mu)\in \Lambda \times_{d,s}\Lambda$. Choose $i\in \mathbb{N}$ such that $(\lambda,\mu)\in \Pi E_i \times_{d,s} \Pi E_i$. By~\eqref{eq:matrixunits2} we have
$$t_\lambda t_\mu^* = \sum_{\lambda \nu \in \Pi E_i} c(\lambda,\nu) \overline{c(\mu,\nu)}\Theta(t)_{\lambda \nu, \mu \nu}^{\Pi E_i}.$$
Since $\pi_{t,t'}|_{M^t_{\Pi E_i}}=\psi_i$, we have
\begin{align*}
\pi_{t,t'}(t_\lambda t_\mu^*)&=\psi_i\left(\sum_{\lambda \nu \in \Pi E_i} c(\lambda,\nu) \overline{c(\mu,\nu)}\Theta(t)_{\lambda \nu, \mu \nu}^{\Pi E_i}\right)\\
&=\sum_{\lambda \nu \in \Pi E_i} c(\lambda,\nu) \overline{c(\mu,\nu)}\omega_i(\lambda \nu, \mu \nu)\Theta(t')_{\lambda \nu, \mu \nu}^{\Pi E_i},
\end{align*}
as claimed. Now fix $\lambda \in \Lambda$. Condition (T1) of Lemma~\ref{le:twistedcoreiso2} implies that $w_i(\alpha,\alpha)=1$ for all $\alpha \in \Lambda$. Hence equation~\eqref{eq:thecoreiso4} implies that
\begin{align*}\pi_{t,t'}(t_\lambda t_\lambda^*)&=\sum_{\lambda \nu \in \Pi E_i} c(\lambda,\nu) \overline{c(\lambda,\nu)} \omega_i(\lambda \nu , \lambda \nu)\Theta(t')_{\lambda \nu, \lambda \nu}^{\Pi E_i}\\
&=\sum_{\lambda \nu \in \Pi E_i} \Theta(t')_{\lambda \nu, \lambda \nu}^{\Pi E_i}\\
&=  t_\lambda' t_\lambda^{\prime*}\qquad \text{by } \eqref{eq:matrixunits2}.
\end{align*}
Fix $\lambda,\mu \in \Lambda$ and suppose that $b\equiv c$. Then $\omega_i(\alpha,\beta)=1$ for all $\alpha,\beta \in \Lambda$. Hence~\eqref{eq:thecoreiso4} gives
\begin{align*}\pi_{t,t'}(t_\lambda t_\mu^*)&=\sum_{\lambda \nu \in \Pi E_i} c(\lambda,\nu) \overline{c(\mu,\nu)}\Theta(t')_{\lambda \nu, \mu \nu}^{\Pi E_i}\\
&=\sum_{\lambda \nu \in \Pi E_i} b(\lambda,\nu) \overline{b(\mu,\nu)}\Theta(t')_{\lambda \nu, \mu \nu}^{\Pi E_i}\\
&=t_\lambda' t_{\mu}^{\prime*}\qquad \text{by } \eqref{eq:matrixunits2}.\end{align*}
By Theorem~\ref{th:nonzerogapprojection}, we have $C^*(\Lambda,b;\mathcal{E})^\gamma \cong C^*(\Lambda,c;\mathcal{E})^\gamma.$
\end{proof}

\section{The gauge-invariant uniqueness theorem}
We state the main result of this chapter.
\label{chptr:gaugeinvariantuniqueness}
\begin{theorem}[an Huef and Raeburn's gauge-invariant uniqueness theorem]\label{th:gaugeinvariant}
Let $(\Lambda,d)$ be a finitely aligned $k$-graph, let $c\in \underline{Z}^2(\Lambda,\mathbb{T})$ and let $\mathcal{E}\subset\FE(\Lambda)$. Suppose that $\left \{ t_\lambda : \lambda \in \Lambda \right \}$ is a relative Cuntz-Krieger $(\Lambda,c;\mathcal{E})$-family. The homomorphism $$\pi_t^\mathcal{E} : C^*(\Lambda,c;\mathcal{E}) \rightarrow C^*(\left \{ t_\lambda :\lambda \in \Lambda \right \} )$$ satisfying $\pi_t^\mathcal{E}(s_\mathcal{E}(\lambda))=t_\lambda$ for all $\lambda\in \Lambda$ is injective on $C^*(\Lambda,c;\mathcal{E})^\gamma$ if and only if
\begin{enumerate}
\item[(1)] $t_v \neq 0$ for each $v\in \Lambda^0$; and
\item[(2)] $Q(t)^E\neq 0$ for all $E\in \FE(\Lambda) \setminus \mathcal{\overline{E}}$.
\end{enumerate}
Furthermore, $\pi_t^\mathcal{E}$ is injective on $C^*(\Lambda,c;\mathcal{E})$ if and only if both (1) and (2) hold and additionally,
\begin{enumerate}
\item[(3)] there exists a group action $\theta:\mathbb{T}^k \rightarrow \Aut\left( C^*\left\{t_\lambda:\lambda \in \Lambda \right \} \right)$ satisfying $\theta_z(t_\lambda)=z^{d(\lambda)}t_\lambda$ for all $\lambda \in \Lambda$.
\end{enumerate}
\end{theorem}

\begin{proof}
Suppose that $\pi_t^\mathcal{E}$ is injective on $C^*(\Lambda,c;\mathcal{E})^\gamma$. Then (1) and (2) follow from Theorem~\ref{th:nonzerogapprojection}. Now suppose that (1) and (2) hold. Then Lemma~\ref{le:satiation1} shows that for $E\in \FE(\Lambda)$ we have $Q(t)^E=0$ if and only in $E\in \overline{\mathcal{E}}$. So by Theorem~\ref{th:nonzerogapprojection} we have
$$Q(t)^E=0 \text{ if and only if } Q(s_\mathcal{E})^E=0.$$
Theorem~\ref{th:nonzerogapprojection} implies that $s_\mathcal{E}(v)\neq 0$ for each $v\in \Lambda^0$, and so by Theorem~\ref{th:coreiso}, $\pi_t^\mathcal{E}$ is injective on $C^*(\Lambda,c;\mathcal{E})^\gamma$.

Now suppose that $\pi_t^\mathcal{E}$ is injective on $C^*(\Lambda,c;\mathcal{E})$. Then setting $\theta:= \pi_t^\mathcal{E}\circ \gamma$ establishes (3) and Theorem~\ref{th:nonzerogapprojection} implies both (1) and (2). Finally, suppose that (1), (2) and (3) hold. Then $\pi_t^\mathcal{E}$ is injective on $C^*(\Lambda,c;\mathcal{E})^\gamma$ by the preceding paragraph. It then follows from Proposition~\ref{prop:actiontheninjectiveoncore} and (3) that $\pi_\mathcal{E}^t$ is injective on $C^*(\Lambda,c;\mathcal{E})$.
\end{proof}

\begin{corollary}
Let $(\Lambda,d)$ be a finitely aligned $k$-graph. Let $\lbrace T_\lambda : \lambda \in \Lambda \rbrace$ be the Toeplitz-Cuntz-Krieger $(\Lambda,c)$-family of Proposition~\ref{toeplitzrepresentation}. Then $\pi_T^\mathcal{T}$ is injective on $\mathcal{T}C^*(\Lambda,c)^\gamma$.
\end{corollary}

\begin{proof}
Fix $v\in \Lambda^0$. The calculation $T_v \xi_v=\xi_v$ shows that $T_v \neq 0$. We have $Q(T)^E\neq 0$ for all $E\in \FE(\Lambda)$ by Lemma~\ref{le:allgapprojectionsarenonzero}. The result then follows from Theorem~\ref{th:gaugeinvariant}.
\end{proof}

\chapter{Gauge-invariant ideals} %
\label{chptr:Gauge-invariant_ideals}
In this chapter we provide a graph-theoretic description of the gauge-invariant ideal structure of $C^*(\Lambda,c;\mathcal{E})$.
\section{From ideals to hereditary sets}
For each ideal $I\subset C^*(\Lambda,c;\mathcal{E})$, we construct a hereditary and relatively saturated subset $H_I$ of vertices and a satiated subset $\mathcal{B}_I$ of a subgraph $\Lambda \setminus \Lambda H$.
\label{chptr:idealstohereditarysets}
\begin{definition}
Let $(\Lambda,d)$ be a finitely aligned $k$-graph. Let $\mathcal{E}\subset \FE(\Lambda)$. Define a relation $\leq$ on $\Lambda^0$ by $v \leq w$ if and only if $v \Lambda w \neq \emptyset$.
\begin{enumerate}
\item[(1)] A set $H \subset \Lambda^0$ is \emph{hereditary} if $v\in H$, $w\in \Lambda^0$ and $v \leq w$ imply $w\in H$.
\item[(2)] A set $H\subset \Lambda^0$ is \emph{saturated relative to} $\mathcal{E}$ if $E\in\overline{ \mathcal{E}}$ and $s(E)\subset H$ implies that $r(E)\in H$.
\end{enumerate}
\end{definition}

\begin{lemma}[cf. \cite{S20062}, Lemma 3.3]\label{le:idealstohereditarysets}
Let $(\Lambda,d)$ be a finitely aligned $k$-graph, let $c\in \underline{Z}^2(\Lambda,\mathbb{T})$ and let $\mathcal{E}\subset\FE(\Lambda)$. Suppose that $I\subset C^*(\Lambda,c;\mathcal{E})$ is an ideal. Define $$H_I:=\left \{ v\in \Lambda^0: s_\mathcal{E}(v)\in I \right \}.$$ Then $H_I$ is hereditary and saturated relative to $\mathcal{E}$.
\end{lemma}

\begin{proof}
To see that $H_I$ is hereditary suppose that $v\in H_I$ and $w\in \Lambda^0$ with $v\leq w$. There is $\lambda \in \Lambda$ such that $s(\lambda)=w$ and $r(\lambda)=v$. As $s_\mathcal{E}(v)\in I$, we have $$s_\mathcal{E}(w)=s_\mathcal{E}(\lambda)^*s_\mathcal{E}(\lambda)=s_\mathcal{E}(\lambda)^*s_\mathcal{E}(r(\lambda))s_\mathcal{E}(\lambda)=s_\mathcal{E}(\lambda)^*s_\mathcal{E}(v)s_\mathcal{E}(\lambda)\in I.$$ So $H_I$ is hereditary.

To see that $H_I$ is saturated relative to $\mathcal{E}$, suppose that $E\in \overline{\mathcal{E}}$ satisfies $s(E)\subset H_I$. We will show that $r(E)\in H_I$. The set $\vee E\subset r(E)\Lambda \setminus \Lambda^0$ is finite and satisfies $E\subset \vee E$ by Lemma~\ref{le:orth3}. So $\vee E \in \overline{\mathcal{E}}$ by (S1). By Lemma~\ref{le:satiation1} we have $Q(s_\mathcal{E})^{\vee E}=0$. As $\vee E$ is closed under minimal common extensions, by Definition~\ref{def:closedundercommonextensions}, it follows from Proposition~\ref{prop:orth1} that
\begin{equation} s_\mathcal{E}(r(E))=s_\mathcal{E}(r(\vee E))=Q(s_\mathcal{E})^{\vee E}+\sum_{\lambda \in \vee E} Q(s_\mathcal{E})^{\vee E}_\lambda=\sum_{\lambda \in \vee E} Q(s_\mathcal{E})_\lambda^{\vee E}.\label{eq:saturatedrelative1}\end{equation}
Since $I$ is an ideal, it suffices to show that $\lambda \in \vee E$ implies $Q(s_\mathcal{E})_\lambda^{\vee E} \in I$. Fix $\lambda \in \vee E$. Then $\lambda = \mu \mu'$ for some $\mu \in E$ and hence $s(\mu)\leq s(\lambda)$. Since $H_I$ is hereditary and $s(\mu)\in s(E)\subset H_I$, it follows that $s(\lambda)\in H_I$. So $s_\mathcal{E}(s(\lambda))\in I.$ Then
$$Q(s_\mathcal{E})_\lambda^{\vee E}=s_\mathcal{E}(\lambda) s_\mathcal{E}(s(\lambda))s_\mathcal{E}(\lambda)^* \prod_{\lambda \lambda'\in \vee E \atop d(\lambda')>0} \left( s_\mathcal{E}(\lambda)s_\mathcal{E}(\lambda)^*-s_\mathcal{E}(\lambda\lambda')s_\mathcal{E}(\lambda \lambda')^* \right)$$
belongs to $I$. Now \eqref{eq:saturatedrelative1} implies that $r(E)\in H_I$.
\end{proof}

\begin{lemma}[\cite{S20062}, Lemma 4.1]\label{le:subkgraphs}
Let $(\Lambda,d)$ be a finitely aligned $k$-graph and let $\mathcal{E}\subset\FE(\Lambda)$. Let $H \subset \Lambda^0$ be hereditary. Then $(\Lambda \setminus \Lambda H,d|_{\Lambda \setminus \Lambda H})$ is a finitely aligned $k$-graph.
\end{lemma}

\begin{proof}
First we check the factorisation property for  $(\Lambda \setminus \Lambda H,d|_{\Lambda \setminus \Lambda H})$. Fix $\lambda \in \Lambda \setminus \Lambda H$ and suppose that $d(\lambda)=m+n$. By the factorisation property for $\Lambda$, there are unique $\mu,\nu \in \Lambda$ such that  $d(\mu)=m$, $d(\nu)=n$ and $\lambda =\mu\nu$. Since $s(\nu)=s(\lambda)\notin H$ we have $\nu \in \Lambda \setminus \Lambda H$. As $H$ is hereditary and $r(\nu) \leq s(\nu)\notin H$ we have $r(\nu)\notin H$. But $s(\mu)=r(\nu) \notin H$, which implies that $\mu \in \Lambda \setminus \Lambda H$. Hence $\Lambda \setminus \Lambda H$ has the factorisation property. To see that the $k$-graph $\Lambda \setminus \Lambda H$ is finitely aligned note that $(\Lambda \setminus \Lambda H)^{\text{min}}(\lambda,\mu) \subset \Lambda^{\text{min}}(\lambda,\mu)$ for all $\lambda,\mu \in \Lambda \setminus \Lambda H$.
\end{proof}

\begin{definition}
Let $(\Lambda,d)$ be a finitely aligned $k$-graph, let $c\in \underline{Z}^2(\Lambda,\mathbb{T})$ and let $\mathcal{E}\subset\FE(\Lambda)$. Let $H \subset \Lambda^0$ and $I \subset C^*(\Lambda,c;\mathcal{E})$ be an ideal. Define $$\mathcal{E}_H:=\left \{ E \setminus EH: E\in \overline{\mathcal{E}} \text{ and } r(E)\notin H \right \}$$
and
$$\mathcal{B}_{I}:=\left \{  E\in \FE(\Lambda \setminus \Lambda H_I) : Q(s_\mathcal{E})^E \in I \right \}.$$
\end{definition}

\begin{lemma}\label{le:restrictedsatiatedsetsareFE}
Let $(\Lambda,d)$ be a finitely aligned $k$-graph, let $c\in \underline{Z}^2(\Lambda,\mathbb{T})$ and let $\mathcal{E}\subset\FE(\Lambda)$. Suppose that $H \subset \Lambda^0$ is hereditary and saturated relative to $\mathcal{E}$. Then $\mathcal{E}_H\subset \FE(\Lambda \setminus \Lambda H)$.
\end{lemma}

\begin{proof}
Suppose that $E\in \mathcal{E}_H$. Since $E\in \mathcal{E}_H$, there exists $F\in \overline{\mathcal{E}}$ such that $E=F\setminus F H$ and $r(F)\notin H$. We need to show that $E\in \FE(\Lambda \setminus \Lambda H)$. Since $H$ is saturated relative to $\mathcal{E}$, we have $s(F)\nsubseteq H$, which implies that $E=F\setminus F H \neq \emptyset$. Fix $\mu \in r(E)\Lambda \setminus \Lambda H$. We will show that $\Ext_{\Lambda \setminus \Lambda H}(\mu;E)\neq \emptyset$.

We claim $\Ext_\Lambda(\mu;F)\setminus \Lambda H=\Ext_{\Lambda \setminus \Lambda H}(\mu;E)$. If $\lambda \in \Ext_{\Lambda \setminus \Lambda H}(\mu;E)$, then $\lambda =\mu \mu'=\nu \nu'$ for some $\nu\in E$ where $(\mu',\nu')\in \Lambda^{\text{min}}(\mu,\nu)$ and $s(\lambda) \notin H$. As $H$ is hereditary and $s(\nu)\leq s(\lambda)$, we have $s(\nu)\notin H$. Hence $\nu \in F$, and so $\lambda \in \Ext_\Lambda(\mu;F)\setminus \Lambda H$. For the reverse containment, suppose $\lambda  \in \Ext_\Lambda(\mu;F)\setminus \Lambda H$. Then $\lambda = \mu \mu'=\nu \nu'$ for some $\nu \in F$ where $(\mu',\nu')\in \Lambda^{\text{min}}(\mu,\nu)$ and $s(\lambda) \notin H$. As $H$ is hereditary and $s(\nu)\leq s(\lambda)$, we have $s(\nu) \notin H$. Hence $\nu \in E$, and so $\lambda \in \Ext_{\Lambda \setminus \Lambda H}(\mu;E)$.

So it suffices to show $\Ext_\Lambda(\mu;F)\setminus \Lambda H\neq \emptyset$. If $\mu\in F \Lambda$, then $s(\mu)\in \Ext_\Lambda(\mu;F)\setminus \Lambda H$, and so $\Ext_\Lambda(\mu;F)\setminus \Lambda H \neq \emptyset$. So suppose $\mu \notin F \Lambda$. By (S2), we have $\Ext_\Lambda(\mu;F)\in \overline{\mathcal{E}}$, and $\Ext_\Lambda(\mu;F)\neq \emptyset$. Since $H$ is saturated relative to $\mathcal{E}$ and $r\left( \Ext_\Lambda(\mu;F)\right)=s(\mu) \notin H$, we have $s\left( \Ext_\Lambda(\mu;F)\right) \nsubseteq H$; that is, there is $ \lambda \in \Ext_\Lambda(\mu;F)$ with $s(\lambda)\notin H$. Hence $\Ext_\Lambda(\mu;F) \setminus \Lambda H \neq \emptyset$.
\end{proof}

\begin{lemma}\label{le:idealstosatiatedsets1}
Let $(\Lambda,d)$ be a finitely aligned $k$-graph, let $c\in \underline{Z}^2(\Lambda,\mathbb{T})$ and let $\mathcal{E}\subset\FE(\Lambda)$. Suppose that $I \subset C^*(\Lambda,c;\mathcal{E})$ is an ideal. Then $\mathcal{B}_{I}\subset \FE(\Lambda \setminus \Lambda H_I)$ is satiated.
\end{lemma}

\begin{proof}
Let $\Gamma:= \Lambda \setminus \Lambda H_I$. To see that $\mathcal{B}_{I}$ is satiated we will repeatedly apply Lemma~\ref{le:gapprojectioninsatiatedsets} and use that $I$ is an ideal. For (S1), suppose that $G\in \mathcal{B}_I$ and $E\subset r(G)\Gamma \setminus \Gamma^0$ is finite with $G \subset E$. Then as $Q(s_\mathcal{E})^G\in I$, we have $$Q(s_\mathcal{E})^E=Q(s_\mathcal{E})^G Q(s_\mathcal{E})^{E\setminus G}\in I.$$
So $E\in \mathcal{B}_I$.
For (S2), suppose that $G\in \mathcal{B}_I$ with $r(G)=v$ and $\mu \in v \Gamma \setminus G \Gamma$. Lemma~\ref{le:gapprojectioninsatiatedsets} implies that $$Q(s_\mathcal{E})^{\Ext(\mu;G)}=s_\mathcal{E}(\mu)^* Q(s_\mathcal{E})^G s_\mathcal{E}(\mu)\in I.$$
So $\Ext(\mu;G)\in \mathcal{B}_I$. For (S3), suppose that $G\in \mathcal{B}_I$, $0<n_\lambda \leq d(\lambda)$ for each $\lambda \in G$. Set $E= \left \{ \lambda(0,n_\lambda) : \lambda \in G \right \}$. Then Lemma~\ref{le:gapprojectioninsatiatedsets} implies that $Q(s_\mathcal{E})^E \leq Q(s_\mathcal{E})^G$. Hence
$$Q(s_\mathcal{E})^E=Q(s_\mathcal{E})^E Q(s_\mathcal{E})^G \in I.$$
So $E\in \mathcal{B}_I.$ For (S4), suppose that $G\in \mathcal{B}_I$, $G' \subset G$ and $G_\lambda'\in s(\lambda)\mathcal{B}_I$ for each $\lambda \in G'$. Put $E=\left( G \setminus G'\right)\cup \left(\bigcup_{\lambda \in G'} \lambda G_\lambda' \right)$. We need to show that $Q(s_\mathcal{E})^E\in I$. Lemma ~\ref{le:neededfors4} implies that for each $\lambda \in G'$, we have
\begin{equation}s_\mathcal{E}(r(G))-s_\mathcal{E}(\lambda)s_\mathcal{E}(\lambda)^*+s_\mathcal{E}(\lambda)Q(s_\mathcal{E})^{G_\lambda'}s_\mathcal{E}(\lambda)^*=\prod_{\mu \in G_\lambda '} \left( s_\mathcal{E}(r(G))-s_\mathcal{E}(\lambda \mu)s_\mathcal{E}(\lambda \mu)^* \right).\label{eq:glambdaprime}\end{equation}
Let $q_I:C^*(\Lambda,c;\mathcal{E})\rightarrow C^*(\Lambda,c;\mathcal{E})/I$ be the quotient map. For each $\lambda \in G'$ we have $G_\lambda'\in \mathcal{B}_I$ and so $Q(s_\mathcal{E})^{G_\lambda'}\in I$. Applying the quotient map to \eqref{eq:glambdaprime} gives
$$q_I\left(s_\mathcal{E}(r(G))-s_\mathcal{E}(\lambda)s_\mathcal{E}(\lambda)^*\right)=q_I \left(\prod_{\mu \in G_\lambda '} \left( s_\mathcal{E}(r(G))-s_\mathcal{E}(\lambda \mu)s_\mathcal{E}(\lambda \mu)^* \right)\right)$$
for each $\lambda \in G'$. The definition of $E$ implies that
\begin{align*}q_I\left(Q(s_\mathcal{E})^E\right)&=q_I\left( Q(s_\mathcal{E})^{G \setminus G'}\right) \prod_{\lambda \in G'} q_I\left ( \prod_{\mu \in G_\lambda'}\left( s_\mathcal{E}(r(G)-s_\mathcal{E}(\lambda \mu) s_\mathcal{E}(\lambda \mu)^* \right ) \right)\\
&=q_I\left( Q(s_\mathcal{E})^{G \setminus G'}\right) \prod_{\lambda \in G'}q_I \left(s_\mathcal{E}(r(G))-s_\mathcal{E}(\lambda)s_\mathcal{E}(\lambda)^*\right)\\
&=q_I\left( Q(s_\mathcal{E})^G \right)\\
&=0,
\end{align*}
since $Q(s_\mathcal{E})^G\in I$. So $Q(s_\mathcal{E})^E\in I$ and hence $E\in \mathcal{B}_I$.
\end{proof}

\begin{lemma}\label{le:idealstosatiatedsets2}
Let $(\Lambda,d)$ be a finitely aligned $k$-graph, let $c\in \underline{Z}^2(\Lambda,\mathbb{T})$ and let $\mathcal{E}\subset\FE(\Lambda)$. Suppose that $I \subset C^*(\Lambda,c;\mathcal{E})$ is an ideal. Then $\mathcal{E}_{H_I} \subset \mathcal{B}_{I}$.
\end{lemma}

\begin{proof}

Suppose that $E\in \mathcal{E}_{H_I}$. By Lemma~\ref{le:restrictedsatiatedsetsareFE} we have $E\in \FE(\Lambda \setminus \Lambda H_I)$, and so it suffices to show that $Q(s_\mathcal{E})^E\in I$. By definition of $\mathcal{E}_{H_I}$, there is $F\in \overline{\mathcal{E}}$ such that $E=F \setminus F H_I$ and $r(F)\notin H_I$. Since $F\in \overline{\mathcal{E}}$ we have $Q(s_\mathcal{E})^F=0$, and so
\begin{equation}0=Q(s_\mathcal{E})^F=Q(s_\mathcal{E})^E Q(s_\mathcal{E})^{FH_I}=Q(s_\mathcal{E})^E \prod_{\lambda \in F H_I}\left(s_\mathcal{E}\left({r(F H_I)}\right)-s_\mathcal{E}(\lambda) s_\mathcal{E}(\lambda^*) \right).\label{eq:orthgapproj}\end{equation}
Let $q_I:C^*(\Lambda,c;\mathcal{E})\rightarrow C^*(\Lambda,c;\mathcal{E})/I$ be the quotient map. Note that $\lambda \in F H_I$ implies $s(\lambda)\in H_I$, which in turn implies that $s_\mathcal{E}(s(\lambda))\in I$, by definition of $H_I$. Hence for $\lambda \in F H_I$ we have
$$q_I(s_\mathcal{E}(\lambda) s_\mathcal{E}(\lambda)^*)=q_I\left(s_\mathcal{E}(\lambda) s_\mathcal{E}(s(\lambda))s_\mathcal{E}(\lambda)^*\right)=0.$$
Applying the quotient map $q_I$ to \eqref{eq:orthgapproj} we have
\begin{align*}0&=q_I\left( Q(s_\mathcal{E})^E \right) \prod_{\lambda \in F H_I}\left(q_I\left(s_\mathcal{E}\left({r(F H_I)}\right)\right)-q_I\left(s_\mathcal{E}(\lambda) s_\mathcal{E}(\lambda^*) \right)\right)\\
&=q_I\left( Q(s_\mathcal{E})^E \right)q_I\left(s_\mathcal{E}\left({r(F H_I)}\right)\right)\\
&=q_I\left(Q(s_\mathcal{E})^E  s_\mathcal{E}\left({r(F H_I)}\right)\right)\\
&=q_I\left(Q(s_\mathcal{E})^E\right) \qquad \text{since } r(FH_I)=r(E).
\end{align*}
So $Q(s_\mathcal{E})^E \in \ker q_I=I$.
\end{proof}

\section{Quotients of twisted relative Cuntz-Krieger algebras}
Given a hereditary and relatively saturated subset $H$ and a satiated subset $\mathcal{B} \subset \FE(\Lambda \setminus \Lambda H)$, we construct a gauge-invariant ideal $I_{H,\mathcal{B}}$. Using the gauge-invariant uniqueness theorem, we show that $C^*(\Lambda,c;\mathcal{E})/I_{H,\mathcal{B}}$ is a canonically isomorphic to a twisted relative Cuntz-Krieger algebra of a subgraph.
\label{chptr:Quotients}
\begin{definition}
Let $(\Lambda,d)$ be a finitely aligned $k$-graph, let $c\in \underline{Z}^2(\Lambda,\mathbb{T})$ and let $\mathcal{E}\subset\FE(\Lambda)$. Let $H \subset \Lambda^0$ be hereditary and let $\mathcal{B}\subset \FE(\Lambda \setminus \Lambda H)$. We define $I_{H,\mathcal{B}}$ to be the ideal in $C^*(\Lambda,c;\mathcal{E})$ generated by
$$\left \{ s_\mathcal{E}(v):v\in H \right \} \cup \left \{ Q(s_\mathcal{E})^E : E\in \mathcal{B} \right \}.$$
\end{definition}

Observe that if $\Gamma$ is a $k$-subgraph of $\Lambda$ and if $c\in\underline{Z}^2(\Lambda,\mathbb{T})$, then $c|_{\Gamma^{*2}}$ is a 2-cocycle for $\Gamma$, where $\Gamma^{*2}=\left\{ (\lambda,\mu)\in \Gamma \times \Gamma: s(\lambda)=r(\mu) \right\}$; that is, $c|_{\Gamma^{*2}}\in \underline{Z}^2(\Gamma, \mathbb{T})$. For convenience we write $c$ for $c|_{\Gamma^{*2}}$.

\begin{lemma}\label{le:fromgraphtosubgraph}
Let $(\Lambda,d)$ be a finitely aligned $k$-graph, let $c\in \underline{Z}^2(\Lambda,\mathbb{T})$ and let $\mathcal{E}\subset\FE(\Lambda)$. Let $H \subset \Lambda^0$ be hereditary and saturated relative to $\mathcal{E}$. Suppose that $\mathcal{B}\subset \FE(\Lambda \setminus \Lambda H)$ is satiated and satisfies $\mathcal{E}_H \subset \mathcal{B}$. Suppose that $\left \{ t_\lambda : \lambda \in\Lambda \setminus\Lambda  H \right \}$ is a relative Cuntz-Krieger $(\Lambda \setminus \Lambda H,c;\mathcal{B})$-family. For $\lambda \in \Lambda \setminus \Lambda H$, let $s_\lambda =t_\lambda$, and for $\lambda \in \Lambda H$, let $s_\lambda=0$. Then $\left \{ s_\lambda : \lambda \in \Lambda \right \}$ is a relative Cuntz-Krieger $(\Lambda,c;\mathcal{E})$-family. There is a unique homomorphism $\pi:C^*(\Lambda,c;\mathcal{E})\rightarrow C^*( \Lambda \setminus \Lambda H, c;\mathcal{B})$ determined by
\begin{align}
&\pi(s_\mathcal{E}(\lambda))=\begin{cases} s_\mathcal{B}(\lambda) &\text{if } \lambda \in \Lambda \setminus \Lambda H \\ 0 &\text{if } \lambda \in \Lambda H .\end{cases}\label{eq:homotodownstairs}
\end{align}
We then have
\begin{align}
&v\in H  \text{ if and only if } s_\mathcal{E}(v)\in \ker \pi; \text{and }\label{eq:kernelideal1}\\
&E\in \mathcal{B} \text{ if and only if } Q(s_\mathcal{E})^E \in \ker \pi.\label{eq:kernelideal2}
\end{align}

\end{lemma}

\begin{proof}
We check axioms (TCK1) to (TCK4) and (CK) for $\left \{ s_\lambda: \lambda \in \Lambda  \right \}$. The set $$\left \{ t_v: v \in (\Lambda\setminus \Lambda H)^0 \right \}$$ is a collection of mutually orthogonal projections and $0$ is a projection satisfying $0t_v=0$ for any $v\in \Lambda^0$. So $\left \{ s_v: v \in \Lambda^0 \right \}$ is a collection of mutually orthogonal projections. This gives~(TCK1).

For (TCK2) we show \begin{equation}s_\lambda s_\mu =c(\lambda,\mu)s_{\lambda \mu} \text{ whenever }s(\lambda)=r(\mu).\label{eq:tck2forslambda}\end{equation} Fix $\lambda,\mu\in \Lambda$ satisfying $s(\lambda)=r(\mu)$. If $s(\mu) \in H$ then both sides of \eqref{eq:tck2forslambda} are zero. So suppose $s(\mu)\notin H$. Since $H$ is hereditary and $s(\lambda) \leq s(\mu)$ we have $s(\lambda) \notin H$. So $$s_\lambda s_\mu=t_\lambda t_\mu=c(\lambda,\mu)t_{\lambda \mu} =c(\lambda, \mu)s_{\lambda \mu}.$$

 For (TCK3), if $\lambda \in \Lambda \setminus \Lambda H$ then $s_\lambda^* s_\lambda =t_\lambda^* t_\lambda =t_{s(\lambda)}=s_{s(\lambda)}$. If $\lambda \in \Lambda H$ then $s(\lambda)\in \Lambda H$. Thus, $s_\lambda^* s_\lambda =0 =s_{s(\lambda)}$.

For (TCK4) we must check that
$s_\lambda s_\lambda^* s_\mu s_\mu^* = \sum_{\sigma\in \MCE_\Lambda(\lambda,\mu)} s_\sigma s_\sigma^*$ for each $\lambda,\mu \in \Lambda$. Suppose that $\lambda \in \Lambda H$ or that $\mu\in \Lambda H$. Then $s_\lambda s_\lambda^* s_\mu s_\mu^*=0$. As $H$ is hereditary, $\sigma \in \MCE_\Lambda(\lambda,\mu)$ implies $\sigma\in \Lambda H$. Hence $\sum_{\sigma\in \MCE(\lambda,\mu)} s_\sigma s_\sigma^*=0$. Now suppose that $\lambda,\mu \in \Lambda \setminus \Lambda H$. Then $$s_\lambda s_\lambda^* s_\mu s_\mu^*=t_\lambda t_\lambda^*t_\mu t_\mu^*=\sum_{\sigma\in \MCE_{\Lambda \setminus \Lambda H}(\lambda,\mu)} t_\sigma t_\sigma^*=\sum_{\sigma\in \MCE_{\Lambda \setminus \Lambda H}(\lambda,\mu)} s_\sigma s_\sigma^*=\sum_{\sigma\in \MCE_{\Lambda }(\lambda,\mu)} s_\sigma s_\sigma^*.$$
This gives (TCK4).

We check (CK). Suppose that $E\in \mathcal{E}$; we must show that $Q(s)^E=0$. If $r(E)\in H$, then $Q(s)^E=s_{r(E)}Q(s)^E.$ So suppose that $r(E)\notin H$. Then
\begin{equation} Q(s)^E=\prod_{\lambda \in E}\left( s_{r(E)}-s_\lambda s_\lambda^*\right)=\prod_{\lambda \in E \setminus EH}\left( t_{r(E)}-t_\lambda t_\lambda^*\right) \prod_{\mu\in EH}\left(t_{r(E)}-0 \right)=Q(t)^{E\setminus EH}\label{eq:CKforslambda}\end{equation}
Since $E\setminus EH\in \mathcal{E}_H \subset \mathcal{B}$, \eqref{eq:CKforslambda} implies that $Q(s)^E=0$.

We now show there is a unique homomorphism $\pi:C^*(\Lambda,c;\mathcal{E})\rightarrow C^*( \Lambda \setminus \Lambda H, c;\mathcal{B})$ satisfying \eqref{eq:homotodownstairs}. Write $\left \{ s_\mathcal{B}(\lambda) : \lambda \in \Lambda \setminus\Lambda H \right \}$ for the universal generating family of $C^*(\Lambda \setminus \Lambda H , c; \mathcal{B})$. For $\lambda \in \Lambda \setminus \Lambda H$, let $s_\lambda=s_\mathcal{B}(\lambda)$, and for $\lambda \in \Lambda H$, let $s_\lambda=0$. The first part of the Lemma implies that $\left\{s_\lambda : \lambda \in \Lambda\right \}$ is a relative Cuntz-Krieger $(\Lambda,c;\mathcal{E})$-family. By the universal property for $C^*(\Lambda,c;\mathcal{E})$ there is a unique homomorphism $$\pi_s^\mathcal{E}: C^*(\Lambda,c;\mathcal{E}) \rightarrow C^*\left( \left\{ s_\lambda : \lambda \in \Lambda \right \} \right)=C^*(\Lambda\setminus \Lambda H , c;\mathcal{B})$$ satisfying $\pi_s^\mathcal{E}(s_\mathcal{E}(\lambda))=s_\lambda$ for each $\lambda \in \Lambda$. So the homomorphism $\pi:=\pi_s^\mathcal{E}$ satisfies \eqref{eq:homotodownstairs}.
We now check~\eqref{eq:kernelideal1}. If $v\in H$ then $\pi(s_\mathcal{E}(v))=0$ by \eqref{eq:homotodownstairs}. Conversely, if $v\notin H$ then Theorem~\ref{th:nonzerogapprojection} implies that $\pi(s_\mathcal{E}(v))=s_\mathcal{B}(v)\neq 0$. We now check \eqref{eq:kernelideal2}. If $E\in \FE(\Lambda \setminus \Lambda H)$ then $$
\pi\left( Q(s_\mathcal{E})^E \right)= Q(s_\mathcal{B})^E .
$$
Theorem~\ref{th:nonzerogapprojection} implies that for $E\in \FE(\Lambda \setminus \Lambda H)$ we have $E\in \mathcal{B}$ if and only if $Q(s_\mathcal{B})^E=0$, and \eqref{eq:kernelideal2} follows.
\end{proof}

\begin{lemma}\label{le:invariantideal}
Let $A$ be a $C^*$-algebra and let $X\subset A$. Let $I_X$ denote the ideal generated by $X$, as in Lemma~\ref{le:ideal}. Suppose that $G$ is a compact abelian group and that $\alpha: G \rightarrow \Aut(A)$ is a group action. If $\alpha_g(x)=x$ for all $x\in X$ and $g\in G$, then $\alpha_g(I_X)=I_X$ for each $g\in G$; that is, $I_X$ is $\alpha$-invariant.
\end{lemma}

\begin{proof}
By Lemma~\ref{le:ideal} $I_X=\overline{\linspan}\lbrace  bxc : b,c\in A , x \in X \rbrace $. For any $b,c\in A$, $g\in G$ and $x\in X$, we have $\alpha_g(bxc)=\alpha_g(b)x\alpha_g(c)\in I_X$. Since $\alpha_g$ is linear and continuous, we have $\alpha_g(I_X)\subset I_X$ for each $g\in G$. Since $I_X =\alpha_g \left( \alpha_{g^{-1}}(I_X)\right)\subset \alpha_g(I_X)$, we have $I_X=\alpha_g(I_X)$.
\end{proof}

\begin{theorem}\label{th:quoofcuntzalgebra}
Let $(\Lambda,d)$ be a finitely aligned $k$-graph, let $c\in \underline{Z}^2(\Lambda,\mathbb{T})$ and let $\mathcal{E}\subset\FE(\Lambda)$. Let $H \subset \Lambda^0$ be hereditary and saturated relative to $\mathcal{E}$. Suppose that $\mathcal{B}\subset \FE(\Lambda \setminus \Lambda H)$ is satiated and that $\mathcal{E}_H \subset \mathcal{B}$. Then $\left \{ s_\mathcal{E}(\lambda)+I_{H,B}: \lambda \in \Lambda \setminus \Lambda H \right \}$ is a relative Cuntz-Krieger $(\Lambda \setminus \Lambda H, c ; \mathcal{B} )$-family in $C^*(\Lambda,c;\mathcal{E})/I_{H,\mathcal{B}}$. The canonical homomorphism $$\pi_{s_\mathcal{E}+I_{H,\mathcal{B}}}^{\mathcal{B}}: C^*(\Lambda\setminus \Lambda H,c ;\mathcal{B}) \rightarrow C^*(\Lambda,c;\mathcal{E})/I_{H,\mathcal{B}}$$ is an isomorphism.
\end{theorem}

\begin{proof}
First we check axioms (TCK1) to (TCK4) and (CK) for $\left \{ s_\mathcal{E}(\lambda)+I_{H,B}: \lambda \in \Lambda \setminus \Lambda H \right \}$. Relations (TCK1),(TCK2) and (TCK3) follow immediately since the quotient map $q:C^*(\Lambda,c;\mathcal{E})\rightarrow C^*(\Lambda,c;\mathcal{E})/I_{H,\mathcal{B}}$ is a homomorphism. For (TCK4), fix $\lambda, \mu \in \Lambda \setminus \Lambda H$. We have
\begin{align*}q\left( s_\mathcal{E}(\lambda) \right)q\left(s_\mathcal{E}(\lambda) \right)^*q\left(s_\mathcal{E}(\mu) \right)q\left(s_\mathcal{E}(\mu) \right)^*=\sum_{\sigma \in \MCE_\Lambda(\lambda,\mu)}q\left(s_\mathcal{E}(\sigma)\right) q\left(s_\mathcal{E}(\sigma)^*\right).
\end{align*}
If $\sigma \in \MCE_\Lambda(\lambda,\mu)\setminus \MCE_{\Lambda \setminus \Lambda H}(\lambda,\mu)$ then $s(\sigma)\in H$. So $s_\mathcal{E}(s(\sigma))\in I_{H,\mathcal{B}}$, which implies that $s_\mathcal{E}(\sigma)s_\mathcal{E}(\sigma)^*=s_\mathcal{E}(\sigma)s_\mathcal{E}(s(\sigma))s_\mathcal{E}(\sigma)^*\in I_{H,\mathcal{B}}$. Hence
\begin{align*}
q\left( s_\mathcal{E}(\lambda) \right)q\left(s_\mathcal{E}(\lambda) \right)^*q\left(s_\mathcal{E}(\mu) \right)q\left(s_\mathcal{E}(\mu) \right)^*=\sum_{\sigma \in \MCE_{\Lambda\setminus \Lambda H}(\lambda,\mu)}q\left(s_\mathcal{E}(\sigma)\right) q\left(s_\mathcal{E}(\sigma)^*\right).
\end{align*}
It remains to check (CK). If $E\in \mathcal{B}$ then $Q(s_\mathcal{E})^E\in I_{H,\mathcal{B}}$. Hence $Q(s_\mathcal{E})^E+I_{H,\mathcal{B}}=0$. We have shown that $\left \{ s_\mathcal{E}(\lambda)+I_{H,B}: \lambda \in \Lambda \setminus \Lambda H \right \}$ is a relative Cuntz-Krieger $(\Lambda \setminus \Lambda H, c ; \mathcal{B} )$-family in $C^*(\Lambda,c;\mathcal{E})/I_{H,\mathcal{B}}$. We now aim to use the gauge-invariant uniqueness theorem,  Theorem~\ref{th:gaugeinvariant}, to see that $\pi_{s_\mathcal{E}+I_{H,\mathcal{B}}}^\mathcal{B}$ is an isomorphism.

Let $\gamma$ denote the gauge action on $C^*(\Lambda,c;\mathcal{E})$. For $v\in H$ and $E\in \mathcal{B}$, we have $\gamma_z(s_\mathcal{E}(v))=s_\mathcal{E}(v)$ and $\gamma_z(Q(s_\mathcal{E})^E)=Q(s_\mathcal{E})^E$ for each $z\in \mathbb{T}^k$. Lemma~\ref{le:invariantideal} implies that $I_{H,\mathcal{B}}\subset C^*(\Lambda,c;\mathcal{E})$ is invariant under the gauge action. Hence there exists an action $\theta$ of $\mathbb{T}^k$ on $C^*(\Lambda,c;\mathcal{E})/I_{H,\mathcal{B}}$ satisfying $\theta_z\left(s_\mathcal{E}(\lambda)+I_{H,\mathcal{B}}\right)= z^{d(\lambda)}s_\mathcal{E}(\lambda)+I_{H,\mathcal{B}}$ for each $z\in \mathbb{T}^k$. We claim that
\begin{enumerate}\item[(1)] $s_\mathcal{E}(v)+I_{H,\mathcal{B}} \neq 0 \text{ for each } v\in \left(\Lambda \setminus \Lambda H \right )^0=\Lambda^0 \setminus H, \text{ and}$
\item[(2)] $Q(s_\mathcal{E})^E+I_{H,\mathcal{B}}\neq  0 \text{ for each } E\in \FE(\Lambda \setminus \Lambda H) \setminus \mathcal{B}.$
\end{enumerate}
Let $\pi:C^*(\Lambda,c;\mathcal{E}) \rightarrow C^*(\Lambda\setminus \Lambda H , c;\mathcal{B} )$ be the homomorphism from Lemma~\ref{le:fromgraphtosubgraph}. Equations \eqref{eq:kernelideal1} and \eqref{eq:kernelideal2} imply that $$\left\{ s_\mathcal{E}(v): v\in H \right \} \cup \left \{ Q(s_\mathcal{E})^E: E\in \mathcal{B} \right \}\subset \ker \pi.$$ Lemma~\ref{le:ideal} implies $I_{H,\mathcal{B}}\subset \ker \pi$. If $v\notin H$, then \eqref{eq:homotodownstairs} and \eqref{eq:kernelideal1} imply that $s_\mathcal{B}(v)=\pi\left( s_\mathcal{E}(v)\right) \neq0$. If $E\in \FE(\Lambda \setminus \Lambda H) \setminus \mathcal{B}$ then \eqref{eq:homotodownstairs} and \eqref{eq:kernelideal2} imply that $Q(s_\mathcal{B})^E=\pi \left( Q(s_\mathcal{E})^E \right)\neq 0$.

Since $\mathcal{B}$ is satiated, the gauge-invariant uniqueness theorem, Theorem~\ref{th:gaugeinvariant}, implies that $\pi_{s_\mathcal{E}+I_{H,\mathcal{B}}}^\mathcal{B}$ is an isomorphism.
\end{proof}

\section{Gauge-invariant ideals of the twisted relative Cuntz-Krieger algebras}

\label{chapt:result}
\begin{definition}
Let $(\Lambda,d)$ be a finitely aligned $k$-graph and let $\mathcal{E}\subset\FE(\Lambda)$. Let $\SH_\mathcal{E}\times \Sat(\Lambda)$ be set of all pairs $(H,\mathcal{B})$ where $H \subset \Lambda^0$ is hereditary and saturated relative to $\mathcal{E}$, and $\mathcal{B} \subset \FE(\Lambda \setminus \Lambda H)$ is satiated with $\mathcal{E}_H \subset \mathcal{B}$.
\end{definition}

\begin{theorem}\label{th:gauageinvariantidealstructure}
Let $(\Lambda,d)$ be a finitely aligned $k$-graph, let $c\in \underline{Z}^2(\Lambda,\mathbb{T})$ and let $\mathcal{E}\subset\FE(\Lambda)$. The map $(H,\mathcal{B}) \mapsto I_{H,\mathcal{B}}$ is a bijection between $\SH_\mathcal{E}\times \Sat(\Lambda)$ and the gauge-invariant ideals of $C^*(\Lambda,c;\mathcal{E})$.
\end{theorem}

\begin{proof}
First we check that $(H,\mathcal{B}) \mapsto I_{H,\mathcal{B}}$ is surjective. Fix a gauge-invariant ideal $I\subset C^*(\Lambda,c;\mathcal{E})$. We claim $I_{H_I,\mathcal{B}_{I}}=I$. We have $I_{H_I,\mathcal{B}_{I}}\subset I$ by definition. Hence there is a well-defined homomorphism $\pi$ from $C^*(\Lambda,c;\mathcal{E})/ I_{H_I,\mathcal{B}_{I}}$ onto $C^*(\Lambda,c;\mathcal{E})/ I$ such that $s_\mathcal{E}(\lambda)+I_{H_I,\mathcal{B}_{I}} \mapsto s_\mathcal{E}(\lambda)+I$ for each $\lambda \in \Lambda$. The set $H_I$ is hereditary and saturated relative to $\mathcal{E}$ by Lemma~\ref{le:idealstohereditarysets}. The set $B_{I}\subset \FE(\Lambda \setminus \Lambda H_I)$ is satiated by Lemma~\ref{le:idealstosatiatedsets1} and satisfies $\mathcal{E}_{H_I}\subset B_{I}$ by Lemma~\ref{le:idealstosatiatedsets2}. Theorem~\ref{th:quoofcuntzalgebra} implies that $\rho:=\pi_{s_\mathcal{E}+I_{H_I,\mathcal{B}_{I}}}^{\mathcal{B}_{I}}: C^*(\Lambda\setminus \Lambda H_I,c ;\mathcal{B}_{I}) \rightarrow C^*(\Lambda,c;\mathcal{E})/I_{H_I,\mathcal{B}_{I}}$ is an isomorphism. Since  \begin{equation}\left(\pi \circ \rho\right)\left( s_{\mathcal{B}_{I}}(\lambda) \right)=s_\mathcal{E}(\lambda)+I\label{eq:universalinducedhomo}\end{equation} for each $\lambda \in \Lambda \setminus \Lambda H_I$, $\left \{ s_\mathcal{E}(\lambda)+I : \lambda \in \Lambda \setminus \Lambda H_I \right \}$ is a relative Cuntz-Krieger $(\Lambda\setminus \Lambda H_I,c;\mathcal{B}_{I})$-family. The uniqueness clause in Theorem~\ref{th:relativecuntzkriegeralgebra} implies that $\pi_{s_\mathcal{E}+I}^{B_I}=\pi \circ \rho.$

Suppose $v\in \Lambda $ satisfies $s_\mathcal{E}(v)+I=0$. Then $s_\mathcal{E}(v)\in I$ and so $v\in H_I$ by definition. Thus, $v\in (\Lambda \setminus \Lambda H_I)^0$ implies $s_\mathcal{E}(v)+I\neq0$. Suppose $E\in \FE(\Lambda \setminus \Lambda H_I)$ satisfies $Q(s_\mathcal{E})^E+I=0$. Then $Q(s_\mathcal{E})^E\in I$ and so $E\in \mathcal{B}_{I}$ by definition. Hence, $E\in \FE(\Lambda \setminus \Lambda H_I) \setminus \mathcal{B}_I$ implies $Q(s_\mathcal{E})+I\neq 0$.

Since $I$ is gauge-invariant, there exists an action $\theta$ on $C^*(\Lambda,c;\mathcal{E})/I$ satisfying $\theta_z:s_\mathcal{E}(\lambda)+I \mapsto z^{d(\lambda)}s_\mathcal{E}(\lambda)+I$ for each $\lambda \in \Lambda \setminus \Lambda H_I$ and $z\in \mathbb{T}^k$. The gauge-invariant uniqueness theorem, Theorem~\ref{th:gaugeinvariant} implies that $\pi_{s_\mathcal{E}+I}^{\mathcal{B}_{I}}=\pi\circ \rho$ is an isomorphism. Hence $\pi:C^*(\Lambda,c;\mathcal{E})/ I_{H_I,\mathcal{B}_{I}}\rightarrow C^*(\Lambda,c;\mathcal{E})/ I$ is an isomorphism and the diagram

\begin{equation}
\begin{tikzpicture}
\matrix(m)[matrix of math nodes,
row sep=4em, column sep=2.5em,
text height=1.5ex, text depth=0.25ex]
{{}&C^*(\Lambda,c;\mathcal{E})&{}\\
C^*(\Lambda,c;\mathcal{E})/I_{H,\mathcal{B}_{I}}&{}&C^*(\Lambda,c;\mathcal{E})/I\\};
\path[->]
(m-1-2) edge node[left, above]{$q_1$}(m-2-1)
(m-1-2) edge node[auto]{$q_2$} (m-2-3)
(m-2-1) edge node[auto]{$\pi$} (m-2-3);
\end{tikzpicture}
\nonumber\end{equation}
commutes, where $q_1$ and $q_2$ denote the quotient maps. If $a\in \ker q_1$ then $q_2(a)=\pi \left( q_1(a) \right)=0$. So $\ker q_1 \subset \ker q_2$. A similar calculation shows the reverse containment. So $I_{H,\mathcal{B}_{I}}=\ker q_1 = \ker q_2 = I$.

Now we check that $(H,\mathcal{B}) \rightarrow I_{H,\mathcal{B}}$ is injective. Fix $(H,\mathcal{B})\in \SH_\mathcal{E}\times \Sat(\Lambda) $. We will show that $H=H_{I_{H,\mathcal{B}}}$ and $\mathcal{B}=\mathcal{B}_{I_{H,\mathcal{B}}}$. The containments $H\subset H_{I_{H,\mathcal{B}}}$ and $\mathcal{B}\subset \mathcal{B}_{I_{H,\mathcal{B}}}$ follow by definition.

If $v\in H_{I_{H,\mathcal{B}}}$ then $s_\mathcal{E}(v) \in I_{H,\mathcal{B}}$. The map $\pi_{s_\mathcal{E}+I_{H,\mathcal{B}}}^{\mathcal{B}}: C^*(\Lambda\setminus \Lambda H,c ;\mathcal{B}) \rightarrow C^*(\Lambda,c;\mathcal{E})/I_{H,\mathcal{B}}$  is an isomorphism by Theorem~\ref{th:quoofcuntzalgebra}. It follows that $s_\mathcal{B}(v)=0$. Since $\mathcal{B}$ is satiated Theorem~\ref{th:nonzerogapprojection} implies $s_\mathcal{B}(v)\neq0$ for all $v\in (\Lambda \setminus \Lambda H)^0$. Hence $v\in H$. So $H=H_{I_{H,\mathcal{B}}}$.

Now suppose that $E\in \mathcal{B}_{I_{H,\mathcal{B}}}\subset \FE(\Lambda \setminus \Lambda H_{I_{H,\mathcal{B}}})= \FE(\Lambda \setminus \Lambda H)$. Then $Q(s_\mathcal{E})^E\in I_{H,\mathcal{B}}$. We have $\pi_{s_\mathcal{E}+I_{H,\mathcal{B}}}^\mathcal{B}\left(Q(s_\mathcal{B})^E\right)=Q(s_\mathcal{E})^E+ I_{H,\mathcal{B}}=0$. As $\pi_{s_\mathcal{E}+I_{H,\mathcal{B}}}$ is injective $Q(s_\mathcal{B})^E=0$. Since $\mathcal{B}$ is satiated, Theorem~\ref{th:nonzerogapprojection} implies that $E\in \mathcal{B}$. So $\mathcal{B}=\mathcal{B}_{I_{H,\mathcal{B}}}$.
\end{proof}


%
%
%
%
%

\appendix

%
%


%
%
%

\end{document}